\documentclass[11pt]{amsart}

\usepackage[margin=1.5in]{geometry} 
\usepackage{tcolorbox}
\usepackage[utf8]{inputenc}
\usepackage[T1]{fontenc}
\usepackage[dvipsnames]{xcolor}
\usepackage{tcolorbox}
\usepackage{amsmath,amsthm,amssymb}
\usepackage{thmtools}
\usepackage{hyperref}
\hypersetup{linkcolor=blue}
\usepackage{cleveref}
\usepackage{stmaryrd}
\usepackage{dsfont}
\usepackage{tkz-euclide}
\usepackage{fancyhdr}
\usepackage{calligra,mathrsfs}
\usepackage[shortlabels]{enumitem}
\usepackage{tikz}
\usepackage{tikz-cd}
\usepackage{tikz-3dplot}
\usepackage{array}
\usepackage[backend=bibtex, sortcites=true, style=alphabetic]{biblatex}
\AtEveryBibitem{\clearfield{url}}

\newcommand{\af}{\mathfrak a}
\newcommand{\bfr}{\mathfrak b}

\newcommand{\ef}{\text{ess}}
\newcommand{\cf}{\mathfrak c}
\newcommand{\pf}{\mathfrak p}
\newcommand{\qf}{\mathfrak q}

\newcommand{\mf}{\mathfrak m}

\newcommand{\nf}{\mathfrak n}

\newcommand{\Df}{\mathfrak D}

\newcommand{\ub}{\mathbf{u}}
\newcommand{\ab}{\mathbf{a}}
\newcommand{\bb}{\mathbf{b}}
\newcommand{\cb}{\mathbf{c}}
\newcommand{\db}{\mathbf{d}}

\newcommand{\A}{\mathbb A}
\newcommand{\C}{\mathbb C}

\newcommand{\F}{\mathbb F}

\newcommand{\Q}{\mathbb Q}
\newcommand{\R}{\mathbb R}
\newcommand{\Z}{\mathbb Z}

\newcommand{\Lc}{\mathcal L}

\newcommand{\Oc}{\mathcal O}

\newcommand{\x}{\mathbf{x}}
\newcommand{\y}{\mathbf{y}}

\renewcommand{\a}{\alpha}
\renewcommand{\b}{\beta}

\renewcommand{\l}{\lambda}

\renewcommand{\i}{\iota}

\renewcommand{\emptyset}{\varnothing}

\DeclareMathOperator{\fpt}{fpt}

\DeclareMathOperator{\lct}{lct}
\DeclareMathOperator{\Spec}{Spec}

\DeclareMathOperator{\Mon}{Mon}
\DeclareMathOperator{\Hom}{Hom}
\DeclareMathOperator{\ini}{in}
\DeclareMathOperator{\gin}{gin}
\DeclareMathOperator{\GL}{GL}

\DeclareMathOperator{\vol}{vol}

\DeclareMathOperator{\conv}{conv}
\DeclareMathOperator{\id}{id}

\newcommand{\del}{\partial}

\DeclareMathOperator{\supp}{supp}
\DeclareMathOperator{\codim}{ht}
\DeclareMathOperator{\ord}{ord}

\newcommand{\ceil}[1]{\lceil #1\rceil}

\newcommand{\lr}[1]{\langle #1 \rangle}

\newcommand{\floor}[1]{\left\lfloor#1\right\rfloor}

\renewcommand{\char}{\textup{char }}

\renewcommand{\bar}{\overline}

\theoremstyle{plain}
\newtheorem{Theoremx}{Theorem}
 
\newtheorem{theorem}{Theorem}[section]
\newtheorem{prop}[theorem]{Proposition}
\newtheorem{lemma}[theorem]{Lemma}
\newtheorem{cor}[theorem]{Corollary}
\newtheorem{conj}[theorem]{Conjecture}

\newtheorem{notation}[theorem]{Notation}
\newtheorem{assumption}[theorem]{Assumption}

\newtheorem*{theorem*}{Theorem}
\newtheorem*{prop*}{Proposition}

\theoremstyle{definition}
\newtheorem{defn}[theorem]{Definition}
\newtheorem{example}[theorem]{Example}
\newtheorem{remark}[theorem]{Remark}

\newtheorem{defn-prop}[theorem]{Definition-Proposition}

\crefname{defn-prop}{definition-proposition}{definition-propositions}
\Crefname{defn-prop}{Definition-Proposition}{Definition-Propositions}

\crefname{prop}{Proposition}{Propositions}
\Crefname{prop}{Proposition}{Propositions}
\crefname{defn}{Definition}{Definitions}
\Crefname{defn}{Definition}{Definitions}
\Crefname{conj}{Conjecture}{Conjectures}

\begin{document}

\title{Homogeneous Ideals with Minimal Singularity Thresholds}
\begin{abstract}
Let $(\Oc_n, \mf)$ denote the ring of germs of holomorphic functions $\mathbb{C}^n\to \mathbb{C}$, and let $I\subseteq \mathcal{O}_n$ be an $\mf$-primary ideal. Demailly and Pham showed that $\lct(I) \geq \frac{1}{e_1(I)} + \dots + \frac{e_{n-1}(I)}{e_n(I)}$, where $e_j(I)$  is the mixed multiplicity $e(I,\dots, I, \mf,\dots, \mf)$, with $I$ repeated $j$ times and $\mf$ repeated $n-j$ times. 

	We generalize the lower bound to the case of an arbitrary ideal of an excellent regular local (or standard-graded) ring of equal characteristic, with $\lct(I)$ replaced by the $F$-threshold $c^{\mf}(I)$ in positive characteristic.  Our main result is a classification of homogeneous ideals in polynomial rings for which the lower bound is attained, resolving a conjecture of Bivi\`a-Ausina in the graded case. 
\end{abstract} 
\author{Benjamin Baily}
\thanks{The author was supported by NSF grants DMS-2101075 and DMS-1840234 and a Simons Dissertation Fellowship.}
\maketitle
\section{Introduction}
Let $X$ be a smooth variety, $Y\subseteq X$ a proper closed subscheme, and $y\in Y$ a point. We study numerical invariants measuring the singularities of the pair $(X, Y)$ at the point $y$: in characteristic zero, we consider the log canonical threshold; in positive characteristic, we consider the $F$-pure threshold.

The log canonical threshold (lct) has attracted considerable attention in algebraic geometry due to its connections with the Minimal Model Program and singularity theory. The $F$-pure threshold (fpt) is an analog of the lct in positive characteristic, related via reduction to characteristic $p \gg 0$ \cite{hara_on_a_2004}. The fpt has provided insights into positive characteristic geometry \cite{bhatt_f_2015,kadyrsizova_lower_2022,page_maximal_2024} and reduction modulo $p$ has been leveraged to better understand singularities in characteristic zero \cite{takagi_f-singularities_2004,takagi_formulas_2006,benito_cluster_2015}.

In recent years, many authors have found lower bounds on the lct (and fpt) in terms of more classical invariants of the embedding \cite{mustata_multiplicities_2002,de_fernex_multiplicities_2004,de_fernex_bounds_2003,takagi_f-pure_2004,kim_teissier_2021,elduque_conjecture_2021,kadyrsizova_lower_2022}. The first of these bounds is due to Skoda \cite{skoda_psh_1972}, who showed that
\begin{equation}\label{eqn:skoda-lb}
\frac{1}{\text{mult}_y(Y)} \leq \lct_y(X, Y).
\end{equation}
The case of equality in \Cref{eqn:skoda-lb} has long been known to experts: $\frac{1}{\text{mult}_y(Y)} = \lct_y(X, Y)$ if and only if $Y$ is a smooth divisor at $y$. Skoda's original formulation of \Cref{eqn:skoda-lb} concerned germs of plurisubharmonic functions, and the equality case in the analytic formulation was settled by Guan and Zhou \cite{guan_lelong_2015}. We consider the equality case in a generalization of Skoda's lower bound.

Let $I$ denote the defining ideal of $Y$ in $\Oc_{X,y}$ and $n = \dim \Oc_{X,y}$. In the case that $\Oc_{X,y}/I$ is zero-dimensional, Demailly and Pham \cite{demailly_sharp_2014} strengthened \Cref{eqn:skoda-lb} by incorporating higher-codimension information about the embedding, namely the mixed multiplicities $e_j(I):= e(I,\dots, I, \mf_y, \dots, \mf_y)$; see \Cref{defn:mixed-mult}. We have 
\begin{equation}\label{eqn:demailly-pham-lb}
	\frac{1}{e_1(I)} + \frac{e_1(I)}{e_2(I)} + \dots + \frac{e_{n-1}(I)}{e_n(I)}\leq \lct_y(X, Y).
\end{equation}

Moving beyond the case considered by Demailly and Pham, we relax the assumption that $\Oc_{X,y}/I$ is zero-dimensional. We recall Rashkovskii's definition of the invariant $E_l(I)$ for any $l\leq \codim(I)$. As in \cite[Theorem 1.4]{rashkovskii_extremal_2015}, we may let $L = V(h_{l+1},\dots, h_n)$ denote the vanishing locus of $n-l$ general linear forms in $\Oc_{X,y}$ and let $I|_L$ denote the expansion of $I$ to the ring $\Oc_{L,y}$. As $\Oc_{L,y}/I|_L$ is zero-dimensional, we apply \Cref{eqn:demailly-pham-lb} and \cite[Example 9.5.4]{lazarsfeld_positivity_2004} to $I|_L$. This yields 
\begin{equation}\label{eqn:rashkovskii_lb}
    \frac{1}{e_1(I|_L)} + \dots + \frac{e_{l-1}(I|_L)}{e_l(I|_L)} \leq \lct(I|_L)\leq \lct(I).
\end{equation}
We then define 
\[E_l(I):= \frac{1}{e_1(I|_L)} + \dots + \frac{e_{l-1}(I|_L)}{e_l(I|_L)},\] which is well-defined for sufficiently general $L$ by \Cref{prop:sigma-properties}. The content of \Cref{eqn:rashkovskii_lb} is that $E_l(I)\leq \lct(I)$. 

More generally, if $(R, \mf)$ is an arbitrary regular local ring and $I\subseteq R$ with $\codim(I) \geq l$, we define $E_l(I)$ similarly --- in particular, when $R/\mf$ is infinite, we let $L$ denote the vanishing set of $n-l$ elements of $\mf$ whose images in $\mf/\mf^2$ are sufficiently general and set $E_l(I):= E_l(I|_L)$. See \Cref{defn:dp-invariant} for a precise definition.

Our first contribution is a generalization of \Cref{eqn:rashkovskii_lb} to arbitrary characteristic. 
\begin{Theoremx}[\Cref{thm:bound}]
	Let $(R, \mf)$ be a regular local ring or a polynomial ring over a perfect field considered with the standard grading.  When $\char R = 0$, assume that $R$ is excellent. Let $I\subseteq R$ be an ideal of height at least $l$, which is assumed to be either homogeneous or $\mf$-primary when $R$ is a polynomial ring. Let $c(I)$ denote the lct (resp. fpt) of $I$. Then $E_l(I) \leq c(I)$.
\end{Theoremx}
In contrast to Skoda's bound, little is known about the case of equality in \Cref{eqn:demailly-pham-lb}. The best result in this direction is due to Bivi\`a-Ausina \cite{bivia-ausina_log_2016}, who classified cases of equality under the additional assumption that $\lct(I) = \lct(I^0)$, where $I^0$ is the smallest monomial ideal containing $I$. Our main result is coordinate-agnostic and resolves the conjecture of \cite[Remark 3.7]{bivia-ausina_Lojasiewicz} in the graded case.
\begin{Theoremx}[\Cref{thm:min-char}]
    Let $k$ be a perfect field of characteristic zero (resp $p > 0$). Let $R = k[x_1,\dots, x_n]$ and let $I\subseteq R$ be a homogeneous ideal of height at least $l$. If $E_l(I) = \lct(I)$ (resp. $E_l(I) = \fpt(I)$), then there exist integers $d_1,\dots, d_l$  such that, in suitable coordinates, we have
    \[
    \overline{I} = \overline{\left(x_1^{d_1},\dots, x_l^{d_l}\right)}.
    \]
\end{Theoremx}
We briefly outline the proof of \Cref{thm:min-char}, which reduces to the case that $k = \bar{k}$ and $l = n$. Assume $E_n(I) = c(I)$. Write $I = I_1 + \dots + I_r$, where $I_j$ is generated by $d_j$-forms and $d_1 < \dots < d_r$.
\begin{enumerate}[(1)]
    \item Using techniques from \cite{bivia-ausina_log_2016,mayes_limiting_2014}, we control the generic initial ideals $\{\gin(I^n)\}_{n\geq 1}$. 
    \item Using (1), we obtain a formula for the numbers $e_j(I)$ in terms of the $d_j$ and the heights of the ideals  $I_1+\dots+I_j$. 
    \item Our formula for the $e_j(I)$ implies that $E_n(I) = B(I)$, where $B(I)$ is defined in terms of mixed \L{}ojasiewicz exponents as in \cite{bivia-ausina_Lojasiewicz}. This allows us to reduce to the case of a homogeneous complete intersection.
    \item We prove the complete intersection case (\Cref{prop:CI-min-char}) by induction on the number of distinct degrees $d_1,\dots, d_r$. 
\end{enumerate}
In the case $r = 1$, any $\mf$-primary ideal $I$ generated by $d$-forms automatically satisfies $\overline{I} = \mf^d$, so there is no way to use $r = 1$ as a useful base case for our induction. Instead, we use $r = 2$. In this case, we show (\Cref{lemma:two-degree-formula}) that $c(I) = E_n(I)$ if and only if $c(I_1) = \codim(I_1)/d_1$. As $\overline{I} = \overline{\overline{I_1} + \mf^{d_2}}$, Theorem B is equivalent in this case to the claim that $\overline{I_1} = (x_1,\dots, x_{\codim(I_1)})^{d_1}$ in suitable coordinates. In characteristic zero, this follows from \cite[Theorem 3.5]{de_fernex_bounds_2003}. In positive characteristic, this fact is \cite[Theorem 3.17]{baily_equigenerated_2025}.

By \Cref{ex:must-be-perfect}, the assumption that $k$ is perfect cannot be weakened. The other main assumption in \Cref{thm:min-char} is that $R$ is graded and $I$ is homogeneous, the relaxation of which is a subtler question. In characteristic zero, we conjecture that \Cref{thm:min-char} holds 
	in the local case up to an \textit{analytic} change of coordinates; see \Cref{conj:ideal-theoretic} for the precise statement and \Cref{conj:valuation-theoretic} for an even stronger statement in terms of valuations. We verify both conjectures in dimension 2. In contrast, in characteristic $p > 0$, \Cref{ex:pos-char-failure} shows that analytic changes of coordinates do not suffice to characterize when $c(I) = E_l(I)$. 
	
	We also remark on two (a priori) possible improvements to Theorems A and B. The first is regarding \cite[Theorem 1.2]{hiep_psh_2025}, which defines for each $1\leq k\leq n$ an invariant $H_k(I)$ satisfying $E_k(I) \leq H_k(I) \leq \lct(I)$ for all ideals $I\subseteq \Oc_n$ of height at least $k$. It would be interesting to see whether the bound $H_k(I)\leq c(I)$ remains true in positive characteristic, and moreover whether there exist ideals for which $E_k(I) < H_k(I) = c(I)$.  
	
	The second remark concerns \cite[Question 1.3]{kim_teissier_2021}, which conveys a conjecture of Hoang Hiep Pham that (implies) $c(I) - c(I|_H) \geq \frac{e_{n-1}(I)}{e_n(I)}$ for every height-$n$ ideal $I\subseteq \Oc_n$ and every hyperplane $H$ through $0\in \C^n$. Tommaso de Fernex shared a counterexample to this conjecture with the author, which we record in \Cref{ex:pinched-polytope}. 

\subsection*{Acknowledgments}
The author is grateful to his advisor, Karen Smith, for her support and guidance through many iterations of this project. The author is additionally grateful to Havi Ellers and Mattias Jonsson for feedback on an earlier draft, and to Tommaso de Fernex for helpful discussions.
\section{Preliminaries}
\begin{assumption}
	Throughout this article, all rings are commutative, Noetherian, and unital. All polynomial rings are considered with the standard grading.
\end{assumption}
\subsection{\texorpdfstring{$F$}{F}-Pure and Log Canonical Thresholds}
We begin with a formal definition of the lct. For a detailed introduction, see \cite{pragacz_impanga_2012}. 
\begin{defn}[Log Resolution]
    Let $X$ be a smooth variety and $Y\subseteq X$ a proper closed subvariety with defining ideal $\af$. Let $W$ be a smooth variety. A projective morphism $\pi: W\to X$ is a \textbf{log resolution} of $(X, Y)$ if $\pi$ is an isomorphism over $X\setminus Y$ and the inverse image $\af\cdot \Oc_W$ is the ideal of a Cartier divisor $D$ such that $D + K_{W/X}$ has simple normal crossings.
\end{defn}
\begin{defn-prop}[Log canonical threshold, \cite{pragacz_impanga_2012} Theorem 1.1]\label{defn:lct}
    Let $X$ be a smooth variety over a characteristic zero field and $Y\subseteq X$ a closed subvariety with defining ideal $\af$. By Hironaka's theorem on resolution of singularities, there exists a log resolution $\pi: W\to X$ of the pair $(X,Y)$. In particular, there exists a simple normal crossings divisor $E_1 + \dots + E_N$ such that
    \[
    D = \sum_{i=1}^N a_iE_i\qquad \text{and} \qquad K_{W/X} = \sum_{i=1}^N k_iE_i.
    \]
    For $y\in Y$, the quantity $\min_{i: y\in \pi(E_i)} \frac{k_i + 1}{a_i}$ does not depend on the log resolution $\pi$ and is denoted $\lct_y(X,Y)$, the \textbf{log canonical threshold} of $(X,Y)$ at $y$. If a point $y$ is not specified, we define
    \[
    \lct(X,Y) = \min_{y\in Y}\lct_y(X,Y) = \min_i \frac{k_i+1}{a_i}.
    \]
    We make the convention that $\lct_y(X,Y) = \infty$ when $y\notin Y$. Additionally, when $k$ is a characteristic zero field, $R = k[x_1,\dots, x_n],$ and $I\subseteq R$ is an ideal, we set $\lct(R, I):= \lct(\A^n, V(I))$ and $\lct_0(R, I):= \lct_0(\A^n, V(I))$. 
    
    In fact, one can define the log canonical threshold of $(X, Y)$ whenever $X$ is an integral \textit{excellent} scheme; see \cite[\S 2]{de_Fernex_limits_2009}. 
\end{defn-prop}
For background on the $F$-pure threshold, we direct the reader to \cite{takagi_f-pure_2004}. In this subsection, we summarize several key definitions and results.
\begin{defn}
    Let $R$ be a ring of characteristic $p>0$. We let $F_*R$ denote the $R$-module structure on $R$ given by restriction of scalars along the Frobenius map $F:R\to R$. We say $R$ is $F$-finite if $F_*R$ is module-finite over $R$. 
\end{defn}
\begin{defn}[\cite{takagi_f-pure_2004}]
    Let $R$ be an $F$-finite ring, $I\subseteq R$ an ideal, and $t\in\R_{\geq 0}$. The pair $(R, I^t)$ is \textit{$F$-split} if for $e\gg 0$, the map
    \[
    I^{\floor{t(p^e-1)}}\cdot \Hom(F^e_*R, R)\to R
    \]
    is surjective. The \textit{$F$-pure threshold} of the pair $(R,I)$ is the supremum of all $t$ such that $(R, I^t)$ is $F$-split. We denote this quantity by $\fpt(R, I)$, or $\fpt(I)$ when the ambient ring is clear. 

Let $k$ be an $F$-finite field of characteristic $p>0, R = k[x_1,\dots, x_n]$, and $I\subseteq R$. Let $\mf:= (x_1,\dots, x_n)$ denote the homogeneous maximal ideal of $R$. We then let $\fpt_0(I)$ denote the quantity $\fpt(R_{\mf},IR_{\mf}).$
\end{defn}
In practice, we do not use the above definitions. Instead, we use the characterization of the $F$-pure threshold as the $F$-threshold at the maximal ideal and the lct as a limit of fpt. 
\begin{defn}
	Let $R$ be a ring of characteristic $p > 0$. Let $\af, J$ be ideals of $R$ such that $\af\subseteq \sqrt{J}$. For each nonnegative integer $e$, we define
	\[
	\nu_{\af}^J(p^e):= \max\{t \in \Z^+: \af^t\not\subseteq J^{[p^e]}\}.
	\]
	By \cite{de_stefani_existence_2018}, the sequence $\dfrac{\nu_{\af}^J(p^e)}{p^e}$ has a limit as $e\to\infty$; we denote this limit by $c^J(\af)$ and refer to it as the \textit{$F$-threshold of $\af$ at $J$}.
\end{defn}
\begin{prop}
    Let $(R,\mf)$ be an $F$-finite regular local ring. Then the $F$-pure threshold of the pair $(R, I)$ is equal to $c^{\mf}(I)$. If instead $R$ is a polynomial ring over an $F$-finite field and $I\subseteq R$ is a homogeneous ideal, then the same holds when we let $\mf$ denote the homogeneous maximal ideal of $R$.
    \end{prop}
\begin{proof}
    The local case is \cite[Remark 1.5]{mustata_f-thresholds_2004} and the graded case is \cite[Proposition 3.10]{de-stefani_graded_2018}.
\end{proof}
\begin{prop}[\cite{hara_generalization_2003}, Theorem 6.8]\label{prop:spread_out}
    Let $A$ be a finite-type $\Z$-algebra and $\af\subseteq A[x_1,\dots, x_n]$ an ideal. Set $k= \text{Frac}(A)$. Then we have
    \[
    \lct(\af\otimes_A k) = \lim_{\substack\mu\in\max\Spec A, |A/\mu|\to\infty}\fpt(A/\mu[x_1,\dots, x_n], \af\otimes_A A/\mu).
    \]
\end{prop}
Many of our results make sense for both fpt and lct, so we introduce the following notation to avoid stating the same results once each for characteristic zero and positive characteristic.
\begin{notation}\label{notation:singularity_threshold}
    Let $k$ be a field, $R = k[x_1,\dots, x_n], \mf = (x_1,\dots, x_n)$, and $I\subseteq R$ an ideal. We define the quantity $c(R, I)$ as follows:
    \[
    c(R, I) = \begin{cases}
        c^{\mf}(I) & \char R = p > 0\\
        \lct_0(R, I) & \char R = 0.
    \end{cases}
    \]
    Similarly, let $(R, \mf)$ be a regular local ring and $I\subseteq R$ an ideal. We define
    \[
    c(R, I) = \begin{cases}
    	c^{\mf}(I) & \char R = p > 0\\
    	\lct(R, I) & \char R = 0 \text{ and $R$ is excellent}.
    \end{cases}
    \]
If the ambient ring is clear, we will use $c(I)$ as shorthand. 
\end{notation}
\begin{prop}[Properties of the singularity threshold]\label{prop:basic-properties-of-c}
    Assume either setting of \Cref{notation:singularity_threshold}, and let $I,J$ be nonzero ideals of $R$. The following properties hold for $c(-)$:
    \begin{enumerate}[(i)]
        \item If $I\subseteq J$, then $c(I)\leq c(J)$.
        \item For all $m > 0$, we have $c(I^m) = m^{-1}c(I)$.
        \item We have $c(I) = c(\bar{I})$, where $\bar{I}$ denotes the integral closure of $I$. 
        \item If $J\subseteq R$ is an ideal such that $R/J$ is regular, then 
        \[c(R, I)\geq c(R/J, (I+J)/J).\]
        \item We have $c(I+J)\leq c(I) + c(J)$.
        \item Suppose $R\subseteq S$ is a ring extension. If $\char R = 0$, assume $S$ is excellent. In both of the following cases, $c(I) = c(IS)$.       \begin{itemize}
        	\item $(R, \mf)$ and $(S, \nf)$ are local and $\mf S = \nf$;
        	\item $R = k[x_1,\dots, x_n]$, $L/k$ is a field extension, and $S = L[x_1,\dots, x_n]$. 
        \end{itemize} 
        
    \end{enumerate}
\end{prop}
\begin{proof}
First, we verify (i)-(v). For characteristic zero, see \cite[Properties 1.12, 1.13, 1.15, 1.17, 1.20]{pragacz_impanga_2012}. For characteristic $p>0$, see \cite[Proposition 2.2 (1), (2), (6), Proposition 4.4]{takagi_f-pure_2004}.

For (vi), note that both cases are regular and faithfully flat extensions. In characteristic zero, see \cite[Proposition 1.9]{jonsson_valuations_2012}.
In characteristic $p > 0$, see \cite[Proposition 2.2 (v)]{huneke_closure_2008}.
\end{proof}
\subsection{Monomial Ideals}
\begin{notation}
    For a vector of ring elements $\mathbf{f} = f_1,\dots, f_r$ and a vector of nonnegative integers $\mathbf{a} = a_1,\dots, a_r$, we let $\mathbf{f}^{\ab}$ denote the element $f_1^{a_1}\dots f_r^{a_r}$. A boldface, lowercase letter always refers to a vector of integers or ring elements.
    \end{notation}
\begin{defn}
	Let $k$ be a field and let $R = k[x_1,\dots, x_n]$. For a tuple $\x = x_{i_1},\dots, x_{i_r}$ with $1\leq i_1<\dots<i_r\leq n$, we let $\Mon(\x)$ denote the monoid of monomials in the variables $\x$. In particular, $\Mon(x_1,\dots, x_n)$ denotes the set of all monomials in $R$. 
\end{defn}
When working with monomial ideals, one often identifies a monomial $x_1^{a_1}\cdots x_n^{a_n}$ with the point $(a_1,\dots, a_n)\in \Z_{\geq 0}^{n}$. For future reference, it will help to give a name to this identification. 
\begin{defn}
    Let $k$ be a field. We define the map 
    \[\log: \Mon(x_{i_1},\dots, x_{i_r})\to \Z_{\geq 0}^{r},\qquad\log(x_{i_1}^{a_1}\cdots x_{i_r}^{a_r}) = (a_1,\dots, a_r).\]
\end{defn}
\begin{defn}
	Let $\af\subseteq k[x_1,\dots, x_n]$ be a monomial ideal. Then the \textit{Newton Polytope} of $\af$, denoted $\Gamma(\af)$, is the convex hull in $\R^{n}_{\geq 0}$ of $\log(\af)$. We let $\conv(-)$ denote the convex hull of a set.
\end{defn}
\begin{remark}
    We record several properties of $\Gamma(\af)$.
    \begin{enumerate}[(i)]
        \item $\Gamma(\af)$ is a closed, convex, unbounded subset of the first orthant of $\R^n$.
        \item When $\af$ is an $\mf$-primary ideal, the closure of the complement of $\Gamma(\af)$ inside the first orthant is an open, bounded polyhedron.
        \item For two ideals $\af, \bfr$, the Minkowski sum of $\Gamma(\af)$ and $\Gamma(\bfr)$ is equal to $\Gamma(\af\bfr)$. In particular, $\Gamma(\af^n) = n\Gamma(\af)$.
    \end{enumerate}
\end{remark}
The following proposition shows that Newton polytope of a monomial ideal determines the singularity threshold.
\begin{prop}\label{prop:monomial-threshold}
     Let $\af\subseteq  k[x_1,\dots, x_n]$ be a monomial ideal. Then 
     \[
     c(\af) = \frac{1}{\mu}\text{, where }\mu = \inf\{t: t\mathbf{1}\in \Gamma(\af)\}.
     \]
\end{prop}
\begin{proof}
    See \cite{howald_multiplier_2001}, Example 5 for characteristic zero and \cite{hernandez_f-purity_2016}, Proposition 36 for positive characteristic.
\end{proof}
\begin{example}\label{ex:monomial-threshold}
	Let $R = k[x,y]$ and set $\af = (x^6, x^5y, x^3y^2, x^2y^3, xy^4, y^6)$. In \Cref{fig:mu_a} we compute $\mu = 2.5$, so $c(\af) = 0.4$. 
	\end{example}
	\begin{figure}
		\centering
\tikzset{every picture/.style={line width=0.75pt}} 

\begin{tikzpicture}[x=0.75pt,y=0.75pt,yscale=-1,xscale=1]

\draw  [draw opacity=0][fill={rgb, 255:red, 0; green, 0; blue, 255}  ,fill opacity=0.25 ] (201,90) -- (440,90) -- (440,310) -- (341,310) -- (261,270) -- (241,250) -- (221,210) -- (201,150) -- cycle ;
\draw    (201,92) -- (201,310) ;
\draw [shift={(201,90)}, rotate = 90] [color={rgb, 255:red, 0; green, 0; blue, 0 }  ][line width=0.75]    (10.93,-3.29) .. controls (6.95,-1.4) and (3.31,-0.3) .. (0,0) .. controls (3.31,0.3) and (6.95,1.4) .. (10.93,3.29)   ;
\draw    (439,310) -- (201,310) ;
\draw [shift={(441,310)}, rotate = 180] [color={rgb, 255:red, 0; green, 0; blue, 0 }  ][line width=0.75]    (10.93,-3.29) .. controls (6.95,-1.4) and (3.31,-0.3) .. (0,0) .. controls (3.31,0.3) and (6.95,1.4) .. (10.93,3.29)   ;
\draw [color={rgb, 255:red, 208; green, 2; blue, 27 }  ,draw opacity=1 ]   (201,310) -- (251,260) ;
\draw    (261,270) -- (241,250) ;
\draw    (241,250) -- (221,210) ;
\draw    (221,210) -- (201,150) ;
\draw    (341,310) -- (261,270) ;
\draw  [fill={rgb, 255:red, 0; green, 0; blue, 0 }  ,fill opacity=1 ] (239,250) .. controls (239,248.9) and (239.9,248) .. (241,248) .. controls (242.1,248) and (243,248.9) .. (243,250) .. controls (243,251.1) and (242.1,252) .. (241,252) .. controls (239.9,252) and (239,251.1) .. (239,250) -- cycle ;
\draw  [fill={rgb, 255:red, 0; green, 0; blue, 0 }  ,fill opacity=1 ] (219,210) .. controls (219,208.9) and (219.9,208) .. (221,208) .. controls (222.1,208) and (223,208.9) .. (223,210) .. controls (223,211.1) and (222.1,212) .. (221,212) .. controls (219.9,212) and (219,211.1) .. (219,210) -- cycle ;
\draw  [fill={rgb, 255:red, 0; green, 0; blue, 0 }  ,fill opacity=1 ] (199,150) .. controls (199,148.9) and (199.9,148) .. (201,148) .. controls (202.1,148) and (203,148.9) .. (203,150) .. controls (203,151.1) and (202.1,152) .. (201,152) .. controls (199.9,152) and (199,151.1) .. (199,150) -- cycle ;
\draw  [fill={rgb, 255:red, 0; green, 0; blue, 0 }  ,fill opacity=1 ] (259,270) .. controls (259,268.9) and (259.9,268) .. (261,268) .. controls (262.1,268) and (263,268.9) .. (263,270) .. controls (263,271.1) and (262.1,272) .. (261,272) .. controls (259.9,272) and (259,271.1) .. (259,270) -- cycle ;
\draw  [fill={rgb, 255:red, 0; green, 0; blue, 0 }  ,fill opacity=1 ] (339,310) .. controls (339,308.9) and (339.9,308) .. (341,308) .. controls (342.1,308) and (343,308.9) .. (343,310) .. controls (343,311.1) and (342.1,312) .. (341,312) .. controls (339.9,312) and (339,311.1) .. (339,310) -- cycle ;
\draw  [fill={rgb, 255:red, 0; green, 0; blue, 0 }  ,fill opacity=1 ] (319,290) .. controls (319,288.9) and (319.9,288) .. (321,288) .. controls (322.1,288) and (323,288.9) .. (323,290) .. controls (323,291.1) and (322.1,292) .. (321,292) .. controls (319.9,292) and (319,291.1) .. (319,290) -- cycle ;
\draw    (250,305) -- (250,315) ;
\draw    (196,260) -- (206,260) ;

\draw (223,182.4) node [anchor=north west][inner sep=0.75pt]    {$xy^{4}$};
\draw (211,132.4) node [anchor=north west][inner sep=0.75pt]    {$y^{6}$};
\draw (241,238.6) node [anchor=south west] [inner sep=0.75pt]    {$x^{2} y^{3}$};
\draw (271,267.6) node [anchor=south west] [inner sep=0.75pt]    {$x^{3} y^{2}$};
\draw (323,278.6) node [anchor=south west] [inner sep=0.75pt]    {$x^{5} y$};
\draw (352,307.6) node [anchor=south west] [inner sep=0.75pt]    {$x^{6}$};
\draw (326,154.4) node [anchor=north west][inner sep=0.75pt]  [font=\Large]  {$\Gamma (\mathfrak{a})$};
\draw (211,253.4) node [anchor=north west][inner sep=0.75pt]    {$\mu \mathbf{1}$};
\draw (240,322.4) node [anchor=north west][inner sep=0.75pt]    {$2.5$};
\draw (167,252.4) node [anchor=north west][inner sep=0.75pt]    {$2.5$};
\end{tikzpicture}
\caption{Computation of $\mu$ for \Cref{ex:monomial-threshold}}
\label{fig:mu_a}
	\end{figure}

Following the proof of \cite{de_fernex_multiplicities_2004}, Theorem 1.4 and the terminology of \cite{mayes_limiting_2014}, we also define the \textbf{limiting polytope} of a graded system of monomial ideals.
\begin{defn}\label{defn:limiting-polytope}
    Let $\af_\bullet$ be a graded system of monomial ideals in $k[x_1,\dots, x_n]$ --- that is, suppose $\af_r\af_s\subseteq \af_{r+s}$ for all $r,s\in \Z^+$. We define $\Gamma(\af_\bullet)$ as the closure in $\R^n$ of the ascending union $\{\frac{1}{2^m}\Gamma(\af_{2^m})\}_{m>0}$. 
\end{defn}
\begin{remark}
	The set $\overline{\bigcup_m \frac{1}{2^m}\Gamma(\af_{2^m})}$ in \Cref{defn:limiting-polytope} coincides with $\overline{\bigcup_m \frac{1}{m} \Gamma(\af_m)}$ in the settings we consider (where $\af_m\supseteq \af_{m+1}$), but we find it easier to work with an ascending union than a filtered union.
\end{remark}

We fix our conventions for working with monomial orders and initial terms in polynomial rings. For further background, see \cite[Chapter 15]{eisenbud_commutative_1995}.
\begin{defn}\label{defn:monomial_order_and_support}
Let $k$ be a field and let $R = k[x_1,\dots, x_n]$. A \textit{monomial order} $>$ on $R$ is a partial order on $\Mon(x_1,\dots, x_n)$ compatible with the divisibility relations: that is, for monomials $\x^\ab, \x^\bb, \x^\cb$, if $\x^\ab > \x^\bb$ then $\x^{\ab+\cb}> \x^{\bb+\cb}$. 

	For an element $f\in R$, write $f = \sum_{\ab\in \Z_{\geq 0}^n} \beta_{\ab}\x^\ab$ where $\beta_{\ab}\in k$. The \textbf{support} of $f$, denoted $\supp(f)$, is the set $\{\x^\ab: \beta_\ab\neq 0\}$. The \textit{initial term} of $f$ with respect to $>$, denoted $\ini_>(f)$, is given by
	\[\sum_{\substack{\ab: \x^\ab \in \supp(f)\\\x^{\ab} \text{ maximal with respect to }>}} \beta_{\ab}\x^\ab.\]
	For an ideal $I\subseteq R$, the \textit{initial ideal} $\ini_>(I)$ is the ideal generated by the elements $\{\ini_>(f): f\in I\}$.
\end{defn}
Valuations provide an important source of monomial orders.
\begin{defn}
	Assume the setting of \Cref{defn:monomial_order_and_support}. A \textit{monomial valuation} is a map of monoids $v:(\Mon(x_1,\dots, x_n),\cdot)\to (\R, +)$. To describe a monomial valuation $v$, it suffices to give the values $v(x_1),\dots, v(x_n)$. We can then describe a monomial order $>_v$ on $R$ by
	\[
	\x^\ab >_v \x^\bb \text{ if and only if }v(\x^\ab) > v(\x^\bb)
	\]
	When $v(x_i)\geq 0$ for all $i$, we can extend $v$ to an $\R$-valuation on $R$: we define $v: R\to [0, \infty]$ by 
	\[v(f) = \min_{\x^\ab\in \supp(f)} v(\x^\ab).\]
\end{defn}
A crucial property of initial ideals is semicontinuity with respect to the singularity threshold.
\begin{prop}\label{prop:semicontinuity}
    Let $R = k[x_1,\dots, x_n], \mf = (x_1,\dots, x_n)$ and let $>$ be a monomial order. If $I\subseteq R$ is an ideal, then $c(\ini_>(I))\leq c(I).$
\end{prop}
\begin{proof}
    For characteristic zero, see \cite{demailly_semi-continuity_2001} for the semicontinuity of the lct. For positive characteristic, see the claim preceding \cite[Remark 4.6]{takagi_f-pure_2004}.
\end{proof}
When $I$ is homogeneous, $k$ is infinite, and $>$ is a total order compatible with the partial order by degree, there is a well-defined notion of the initial ideal in ``generic coordinates.''
\begin{defn-prop}[\cite{eisenbud_commutative_1995}, Theorem 15.18]
Let $k$ be an infinite field, $R = k[x_1,\dots, x_n]$, and $I\subseteq R$ a homogeneous ideal. Suppose that $>$ is a monomial total order on $R$ such that $x_1>\dots >x_n$ and $\x^\ab > \x^\bb$ whenever $\deg(\x^\ab) > \deg(\x^\bb)$. Then there exists a nonempty open subset $U\subseteq \GL_n(k)$ and a monomial ideal $J$ such that:
	\begin{itemize}
		\item For all $\gamma\in U$, we have $\ini_>(\gamma I) = J$;
		\item $U$ is \textit{Borel-fixed}: if $B\subseteq \GL_n(k)$ denotes the subgroup of lower triangular matrices, then $BU = U$.
	\end{itemize}
	The ideal $J$ is called the \textit{generic initial ideal} of $I$ with respect to $>$, and is denoted $\gin_>(I)$.
\end{defn-prop}
\subsection{Integral Closure of Ideals}
\begin{defn}
    Let $I$ be an ideal in a ring $R$. An element $r\in R$ is integral over $I$ if there exists an integer $n$ and elements $a_1,\dots, a_n, a_i\in I^i$ such that
    \[
    r^n + a_1r^{n-1} + \dots + a_n = 0.
    \]
    We then define the integral closure $\bar{I}$ of $I$ as the set of elements $r\in R$ which are integral over $I$.
\end{defn}
\begin{defn-prop}[\cite{huneke_integral}, Corollary 1.2.5]
Let $J\subseteq I$ be ideals in ring $R$. We say $J$ is a \textbf{reduction} of $I$ if either of the following equivalent conditions holds:
\begin{enumerate}[(i)]
	\item There exists a positive integer $n$ such that $I^nJ= I^{n+1}$.
	\item We have $I\subseteq \overline{J}$;
	\item We have $\overline{I} = \overline{J}$.
\end{enumerate}
\end{defn-prop}
Those hoping for an exhaustive discussion of the integral closure of ideals should consult \cite{huneke_integral}. For now, we list some basic properties of $\bar{I}$.
\begin{prop}[Properties of the Integral Closure, \cite{huneke_integral} Chapter 1]\label{prop:int-closure}
    Let $R$ be a ring and $I\subseteq R$ an ideal. Let $\varphi: R\to S$. Then we have
    \begin{enumerate}[(i):]
        \item $\bar{I}$ is an ideal.
        \item $\overline{(\bar{I})} = \bar{I}$.
        \item $\bar{I}S\subseteq \overline{IS}$.
        \item If $J\subseteq S$ is an ideal, then $\varphi^{-1}(\bar{J}) = \overline{\varphi^{-1}(J)}$.
        \item For any multiplicatively closed subset $W\subseteq R$, we have $W^{-1}\bar{I} = \overline{W^{-1}I}$.
        \item The integral closure of a monomial ideal $\af$ in a polynomial ring $k[x_1,\dots, x_n]$ is generated by the set $\x^\ab: \ab\in \Gamma(\af)$.
        \item If $\varphi$ is faithfully flat or an integral extension, then $\bar{I}S\cap R = \bar{I}$.
    \end{enumerate}
\end{prop}
Integral closure is an operation which respects many numerical invariants we are interested in this paper. 
\begin{theorem}[\cite{huneke_integral}, Proposition 11.2.1, Theorem 11.3.1]\label{thm:rees}
    Let $(R,\mf)$ be a formally equidimensional local ring and $I\subseteq J$ two $\mf$-primary ideals. Then $e(I) = e(J)$ if and only if $\bar{I} = \bar{J}$.
\end{theorem}
The same result holds in the case that $(R,\mf)$ is instead standard-graded.
\begin{prop}
    Let $I\subseteq k[x_1,\dots, x_n]$ be an ideal. Then $c(I) = c(\bar{I})$. 
\end{prop}
\begin{proof}
    For characteristic zero, see \cite{pragacz_impanga_2012}, Property 1.15. For positive characteristic, see \cite{takagi_f-pure_2004}, Proposition 2.2 (6).
\end{proof}
To conclude this subsection, we recall a version of the Brian\c{c}on-Skoda theorem,  due to Lipman and Sathaye in the level of generality we need.
\begin{theorem}\label{thm:Briancon-Skoda}
	Let $(R, \mf)$ be a regular local ring of dimension $d$ and $I\subseteq R$ an ideal. Then for all integers $t>0$, we have
	\[
	\overline{I}^{t+d}\subseteq \overline{I^{t+d}}\subseteq I^{t+1}.
	\]
\end{theorem}
\begin{proof}
	The first containment is \cite[Proposition 5.3.1]{huneke_integral}. For the second, we reduce to the case of an infinite residue field; we let $R(X):= R[X]_{\mf R[X]}$ as in op. cit. By \Cref{prop:int-closure} we have $\overline{I} = \overline{IR(X)}\cap R$, so it suffices to prove the claim for $R = R(X)$. 
	
	 By \cite[Proposition 5.1.6 and Theorem 8.6.6]{huneke_integral} there exists an ideal $J\subseteq I$ with $\overline{J^t} = \overline{I^t}$ for all $t>0$ and such that $J$ can be generated by at most $d$ elements. By \cite[Theorem 1'']{lipman_briancon_1981} we deduce
	\[
	\overline{I^{t+d}} = \overline{J^{t+d}} \subseteq J^{t+1}\subseteq I^{t+1}.
	\]
\end{proof}
\subsection{Mixed Multiplicities and the Demailly-Pham Invariant}\label{sec:dp-invariant}
We recall the definition of the mixed multiplicity symbol $e(I_1,\dots, I_d; M)$. 
\begin{defn}
    Let $M$ be a finite-length $R$-module. We let $\l_R(M)$ denote the length of $M$ as an $R$-module.
\end{defn}
\begin{theorem}[\cite{huneke_integral}, Theorem 17.4.2]\label{thm:mixed-mult-exists}
    Let $(R, \mf)$ be a Noetherian local ring, $I_1,\dots, I_k$ ideals of $R$ primary to $\mf$, and $M$ a finitely-generated $R$-module. Then there exists a polynomial $P(n_1,\dots, n_k)$ with rational coefficients and total degree at most $\dim R$ such that for all $n_1,\dots, n_k\gg 0$, we have
    \begin{equation}\label{eqn:multi-hilbert-function}
    P(n_1,\dots, n_k) = \l_R\left(\frac{M}{I_1^{n_1}\dots I_k^{n_k}M}\right).
    \end{equation}
\end{theorem}
\begin{defn}[Mixed Multiplicity]\label{defn:mixed-mult}
    Let $(R, \mf)$ be a Noetherian local ring of dimension $d$. Let $I_1,\dots, I_k$ be $\mf$-primary ideals of $R$ and let $P(n_1,\dots, n_k)$ be as in \Cref{eqn:multi-hilbert-function}. Write $Q(n_1,\dots, n_k)$ for the degree-$d$ part of $P(n_1,\dots, n_k)$. The coefficients of $Q$ define the \textbf{mixed multiplicities} $e(I_1^{\lr{d_1}},\dots, I_k^{\lr{d_k}}; M)$:
    \begin{equation}\label{eqn:mixed-mult-def}
        Q(n_1,\dots, n_k) = \sum_{d_1 + \dots + d_k = d}\binom{d}{d_1,\dots, d_k}^{-1}e(I_1^{\lr{d_1}},\dots, I_k^{\lr{d_k}}; M)
    \end{equation}
    The expression $e(I_1^{\lr{d_1}},\dots, I_k^{\lr{d_k}}; M)$ is shorthand for the expression \[e(I_1,\dots, I_1,\dots, I_k,\dots, I_d; M),\] where $I_j$ is repeated $d_j$ times. 
\end{defn}

\begin{remark}
    Other authors \cite{huneke_integral} have used the notation $e(I_1^{[d_1]},\dots, I_k^{[d_k]};M)$. To avoid confusion with the Frobenius powers $I_j^{[p^e]}$ of the ideals $I_j$, we use angle brackets in the exponent.
\end{remark}
We now define the mixed multiplicities $e_j(I)$.
\begin{defn}\label{defn:teissier-mm}
    Let $(R, \mf)$ be a Noetherian local ring of dimension $d$ and let $I$ denote an $\mf$-primary ideal. We define
    \[
    e_j(I) = e(I^{\lr{j}}, \mf^{\lr{d-j}}; R).
    \]
\end{defn}
We record a few basic properties of the numbers $e_j(I)$.
\begin{prop}\label{prop:mixed-mult-properties}
	Let $(R, \mf)$ be a regular local ring of dimension $n$ with infinite residue field. Let $I\subseteq R$ be an $\mf$-primary ideal.
    \begin{enumerate}[(i)]
        \item We have $e_0(I) = 1, e_1(I) = \text{ord}_\mf(I),$ and $e_n(I) = e(I)$.
        \item If $h_1,\dots, h_n$ are elements of $\mf$ whose images in $\mf/\mf^2$ are sufficiently general, then for all $0\leq j\leq n$ we have 
        \[
        e_j(I)  = e\left(\frac{I+ (h_1,\dots, h_{n-j})}{(h_1,\dots, h_{n-j})}\right)= e\left(I+ (h_1,\dots, h_{n-j})\right),
        \] where $e(-)$ denotes the usual Hilbert-Samuel multiplicity.
        \item Minkowski inequality: for  $j\geq 1$, we have $e_j(I)^2\leq e_{j-1}(I)e_{j+1}(I)$.
    \end{enumerate}
\end{prop}
\begin{proof}\mbox{}
    \begin{enumerate}[(i)]
        \item Follows from (ii).
        \item The first equality is \cite[Corollary 17.4.7]{huneke_integral}; the second is \cite[Proposition 11.1.19]{huneke_integral}.

        \item Follows from \cite[Theorem 17.7.2]{huneke_integral}.
\end{enumerate}
\end{proof}
To state \Cref{thm:bound}, we must extend \Cref{defn:mixed-mult} to the case of an ideal $I\subseteq R$ which is not necessarily $\mf$-primary. 
\begin{defn}[\cite{bivia-ausina_joint_2008}, Definition 2.4]\label{defn:sigma}
	Let $(R, \mf)$ be an $n$-dimensional local ring and $I_1,\dots, I_n$ ideals of $R$. We define
	\begin{equation}\label{eqn:sup-sigma}
	\sigma(I_1,\dots, I_n) = \sup_{t>0}\ e(I_1+\mf^t, \dots, I_n + \mf^t).
	\end{equation}
	As a special case, for $1\leq j\leq n$ and $I\subseteq R$, we define $\sigma_j(I):= \sigma(I^{\lr{j}}, \mf^{\lr{n-j}})$ as in \Cref{defn:teissier-mm}.
\end{defn}
The quantity (\ref{eqn:sup-sigma}) may be infinite. The following proposition summarizes basic properties of $ \sigma(I_1,\dots, I_n)$, among which is a characterization of when the quantity (\ref{eqn:sup-sigma}) is finite.
\begin{prop}[\cite{bivia-ausina_joint_2008}]\label{prop:sigma-properties}
	Let $(R, \mf)$ be an $n$-dimensional regular local ring and $I_1,\dots, I_n$ ideals of $R$. 
	\begin{enumerate}[(i)]
		\item If $I_1,\dots, I_n$ have height $n$, then $\sigma(I_1,\dots, I_n) = e(I_1,\dots, I_n)$.
		\item We have $\sigma(I_1,\dots, I_n) < \infty$ if and only if there exist $g_1\in I_1,\dots, g_n\in I_n$ such that $(g_1,\dots, g_n)$ is $\mf$-primary. In this case,  $\sigma(I_1,\dots, I_n) = e(g_1,\dots, g_n)$ for elements $g_i\in I_i$ whose images in $I_i/\mf I_i$ are sufficiently general. 
		\item In particular, for an ideal $I\subseteq R$, we have $\sigma_j(I) < \infty$ if and only if $\codim(I) \geq j$. In this case, if $h_1,\dots, h_{n-j}\in \mf$ are elements whose images in $\mf/\mf^2$ are sufficiently general, we have
		\begin{equation}\label{eqn:sigma-via-general-hyperplanes}
			\sigma_j(I) = e\left(\frac{I+ (h_1,\dots, h_{n-j})}{(h_1,\dots, h_{n-j})}\right)= e\left(I+ (h_1,\dots, h_{n-j})\right).
			\end{equation}
			\item The Minkowski inequalities hold: if $\sigma_{j-1}(I),\sigma_j(I)$ are finite then $\sigma_j(I)^2\leq \sigma_{j-1}(I)\sigma_{j+1}(I)$.
	\end{enumerate}
\end{prop}
\begin{proof}
	\mbox{}
	\begin{enumerate}[(i)]
		\item For $t\gg 0$, we have $\mf^t\subseteq I_i$ for all $1\leq i\leq n$. See the remark after Definition 2.4 in op. cit.
		\item This is Proposition 2.4 in op. cit.
		\item When $j\leq \codim I$, use prime avoidance to choose $g_1,\dots, g_j\in I$ such that $\codim((g_1,\dots, g_j))=j$ and $g_{j+1},\dots, g_n\in \mf$ such that $\codim((g_1,\dots, g_n))=n$. When $j \geq \codim(I) + 1$, for any $g_1,\dots, g_j\in I, g_{j+1},\dots, g_n\in \mf$ we have 
		 \[
		 \codim((g_1,\dots, g_n))\leq \codim(I + (g_{j+1},\dots, g_n)) \leq \codim(I) + n-j \leq n-1.
		 \]
		 To see that the first equality in \Cref{eqn:sigma-via-general-hyperplanes} holds, pick $t\gg 0$ such that $\sigma_j(I) = e_j(I + \mf^t)$. Use (ii) to pick $g_1,\dots, g_j\in I, h_1,\dots, h_{n-j}\in \mf$ generally such that $\sigma_j(I) = e((g_1,\dots, g_j,h_1,\dots, h_{n-j}))$. Use \Cref{prop:mixed-mult-properties}(ii) to pick the $h_1,\dots, h_{n-j}$ even more generally so that $e_j(I+\mf^t) = e(I + \mf^t + (h_1,\dots, h_{n-j}))$. Using monotonicity of Hilbert-Samuel multiplicity, we compute
		 \begin{align*}
		 \sigma_j(I) &= e_j(I + \mf^t) = e(I + \mf^t + (h_1,\dots, h_{n-j})) \leq e(I + (h_1,\dots, h_{n-j}))
		 \\&\leq e(g_1,\dots, g_j, h_1,\dots, h_{n-j}) = \sigma_j(I).
		 \end{align*}
		\item Follows immediately from \Cref{prop:mixed-mult-properties} (iii).
	\end{enumerate}
\end{proof}
We will now define the Demailly-Pham invariant \cite{demailly_sharp_2014}.
\begin{defn}\label{defn:dp-invariant}
    Let $(R, \mf)$ be a regular local ring, $l\in \Z^+$, and $I$ an ideal with $\codim(I)\geq l$. Then we set
    \[
    E_l(I):= \frac{1}{\sigma_1(I)} + \dots + \frac{\sigma_{l-1}(I)}{\sigma_l(I)}.
    \]

\end{defn}

	Suppose instead $R = k[x_1,\dots, x_n]$ and $\mf = (x_1,\dots, x_n)$. Let $I\subseteq R$ be an $\mf$-primary ideal. For all $r,s>0$, we have 
	\[
	\lambda_R\left(\frac{R}{I^r\mf^s}\right) = \lambda_{R_\mf}\left(\frac{R_\mf}{I^r\mf^sR_\mf}\right),
	\]
	so \Cref{thm:mixed-mult-exists} holds for $\lambda_R\left(\frac{R}{I^r\mf^s}\right)$ without assuming that $R$ is local. In particular, for $0\leq j\leq n$, the numbers $e(I^{\lr{j}}, \mf^{\lr{n-j}})$ can be defined intrinsically from the polynomial $P(r,s)$ and agree with the quantities $e_j(IR_\mf)$. Additionally, in this setting we have an analog of \Cref{thm:rees}.
\begin{prop}\label{prop:dp-rees}
    Let $R$ be a regular ring and $\mf\subseteq R$ a maximal ideal such that $\dim R_\mf = n$. Let $I_1\subseteq I_2\subseteq R$ be $\mf$-primary ideals. Then $E_n(I_1 R_\mf)\leq E_n(I_2 R_\mf)$ with equality if and only if $\bar{I_1} = \bar{I_2}$.
\end{prop}
\begin{proof}
    The case that $(R, \mf)$ is local is \cite[Corollary 11]{bivia-ausina_log_2016}. In the non-local case, we have $\overline{I_1R_\mf} = \overline{I_2R_\mf},$ so $I_1R_\mf$ is a reduction of $I_2R_\mf$ by \cite[Corollary 1.2.5]{huneke_integral}. By definition of reduction, there exists $t > 0$ such that $I_1I_2^tR_\mf = I_2^{t+1}R_\mf$. The ideals $I_1I_2^t, I_2^{t+1}$ are $\mf$-primary, so we have
    \[
    I_1I_2^{t} = I_1I_2^tR_\mf \cap R = I_2^{t+1}R_\mf\cap R = I_2^{t+1},
    \]
    so $I_1$ is a reduction of $I_2$. By \cite[Corollary 1.2.5]{huneke_integral} again, we deduce that $\overline{I_1} = \overline{I_2}$.
\end{proof}
We extend \Cref{defn:dp-invariant} to the ring $k[x_1,\dots, x_n]$.
\begin{defn}\label{defn:mm-dp-poly-ring}
	Let $R = k[x_1,\dots, x_n], \mf = (x_1,\dots, x_n)$ and let $I\subseteq \mf$ be an ideal. For $1\leq j\leq \codim(IR_\mf)$ we define $\sigma_j(I):= \sigma_j(IR_\mf)$ and $E_j(I):= E_j(IR_\mf)$. 
\end{defn}
\begin{remark}
	Even though \Cref{defn:mm-dp-poly-ring} defines a local invariant, we suppress the reference to the maximal ideal $\mf$, as we will only ever consider singularities at the origin.
\end{remark}
\begin{prop}\label{prop:sigma-poly-ring}
	Let $R = k[x_1,\dots, x_n], \mf = (x_1,\dots, x_n)$, and let $I\subseteq R$ be an ideal with $\codim(I) \geq l$. 
		\begin{enumerate}[(a)]
					\item If $I$ is $\mf$-primary, then \Cref{prop:mixed-mult-properties} holds for $I$.
					\item If $I$ is homogeneous, then \Cref{prop:sigma-properties} (i), (iii), (iv) hold, provided we assume the elements $h_1,\dots, h_{n-j}$ in (iii) are \textit{homogeneous}.
		\end{enumerate}
\end{prop}
\begin{proof}
	\mbox{}
	\begin{enumerate}[(a)]
		\item As in \Cref{prop:mixed-mult-properties}, (i) follows from (ii). For (ii), for any $h_1,\dots, h_{n-j}\in \mf$, the ideal $(I + (h_1,\dots, h_{n-j}))$ is $\mf$-primary as well. As $\mf/\mf^2 \cong \mf R_\mf/\mf^2 R_\mf$, choosing $h_1,\dots, h_{n-j}\in \mf$ whose images in $\mf R_\mf/\mf^2 R_\mf$ are sufficiently general, we have
		\[
		e_j(I):= e_j(IR_\mf) = e\left(I + (h_1,\dots, h_{n-j})R_\mf\right) = e\left(I + (h_1,\dots, h_{n-j})\right).
		\]
		The other equality in (ii) is similar, and (iii) is immediate: \[e_{j-1}(I)e_{j+1}(I):= e_{j-1}(IR_\mf)e_{j+1}(IR_\mf)\leq e_j(IR_\mf)^2 =: e_j(I)^2.\]
		\item For \Cref{prop:mixed-mult-properties} (i), we note that a homogeneous ideal of height $n$ is automatically primary to $\mf$ and appeal to part (a). For (iii), we note that when $h_1,\dots, h_{n-j}$ are homogeneous and sufficiently general, $\codim(I + (h_1,\dots, h_{n-j}))$ is homogeneous of height $n$ and thus $\mf$-primary. Additionally, we note that $I + \mf^t$ is $\mf$-primary for all $t > 0$, after which we appeal to part (a) (ii) of this proposition. Similarly to part (a), \Cref{prop:sigma-properties} (iv) follows from the local case.	
\end{enumerate}
\end{proof}
\begin{lemma}\label{lemma:sigma_extension}
	Suppose we are in one of the following situations:
	\begin{enumerate}
		\item $L/k$ is a field extension, $R = k[x_1,\dots, x_n], \mf = (x_1,\dots, x_n), S = L[x_1,\dots, x_n]$, and $I\subseteq R$ is $\mf$-primary or homogeneous. 
		\item $(R, \mf)\to (S, \nf)$ is an extension of $n$-dimensional regular local rings, $\mf S = \nf$, and $I\subseteq R$ is an ideal.
	\end{enumerate}
	Then for all $1\leq j\leq n$ such that $\sigma_j(I)$ is defined, we have $\sigma_j(I) = \sigma_j(IS)$.
\end{lemma}
\begin{proof}
	The assumptions imply that $\lambda_R(J) = \lambda_S(JS)$ for every $\mf$-primary ideal $J\subseteq R$. In particular, this applies to the ideals $(I + \mf^t)^a\mf^b$ for all $a,b,t>0$; the result follows from \Cref{defn:mixed-mult,defn:sigma}.
\end{proof}

\section{Proof of Theorem A}
\subsection{\texorpdfstring{$F$}{F}-Pure Thresholds and the Demailly-Pham Invariant}
In this subsection, we require an asymptotic version of \cite[Theorem 13]{bivia-ausina_log_2016} in arbitrary characteristic. Our result follows from \cite{bivia-ausina_log_2016,demailly_sharp_2014} with the necessary changes being made.
\begin{defn}
	Let $k$ be a field and $R = k[x_1,\dots, x_n]$.  The \textbf{reverse lexicographic order} on $R$, denoted by $>_{\text{rlex}}$, is the monomial order such that $\x^\ab > \x^\bb$ if and only if there exists an index $0\leq i\leq n-1$ such that $(a_n,\dots, a_{n-i+1}) = (b_n,\dots, b_{n-i+1})$ and $a_{n-i} < b_{n-i}$. The \textbf{graded reverse lexicographic order}, denoted $>_{\text{grlex}}$, is the monomial order given by $\x^\ab >_{\text{grlex}} \x^\bb$ if and only if $\deg(\x^\ab) > \deg(\x^\bb)$ or else $\deg(\x^\ab) = \deg(\x^\bb)$ and $\x^\ab>_{\text{rlex}} \x^\bb$. 
\end{defn}
\begin{lemma}\label{lem:revlex-proj}
    Let $k$ be a field, $R = k[x_1,\dots, x_n]$, and $J\subseteq R$ an ideal. For $1\leq j\leq n$, we define $\pi_j: R\to R/(x_{j+1},\dots, x_n)\cong k[x_1,\dots, x_j]$. Then 
    \[
    \ini_{>_{\text{rlex}}}\pi_j(J) = \pi_j(\ini_{>_{\text{rlex}}}(J)).
    \]
\end{lemma}
\begin{proof}
    Let $f\in J$. Write $f = g + h$, where $h\in (x_{j+1},\dots, x_n)$ and $g\in k[x_1,\dots, x_j]$. If $g = 0$, then $\pi_j(f) = 0$. If $g\neq 0$, then $\ini_>(f) = \ini_>(g)$. In both cases, we have $\pi_j(\ini_>(f)) = \ini_>(\pi_j(f)).$
\end{proof}
The following is a general lemma which will be used repeatedly throughout the rest of this article.
\begin{lemma}\label{lemma:maxpower-subset}
    Let $L$ be a field, $S = L[x_1,\dots, x_n]$, and $J\subseteq S$ an $\mf$-primary homogeneous ideal generated by forms of degree $\leq d$. Then $\mf^d\subseteq \bar J$.
\end{lemma}
\begin{proof}
    We first prove the result in the case that $L$ is infinite. First, choose forms $f_1,\dots, f_n$ from among the generators of $J$ such that $(f_1,\dots, f_n)$ is $\mf$-primary. If $h_1,\dots, h_n$ are general linear forms, then \[J':= (h_1^{d-\deg(f_1)}f_1,\dots, h_n^{d-\deg(f_{n})}f_{n})\] is an $\mf$-primary $(d,\dots, d)$-complete intersection contained in $J$. As $J'\subseteq \mf^{d}$ and $e(J)' = d^n = e(\mf^{d})$, we have $\mf^{d} = \overline{\mf^{d}}\subseteq \bar{J'}\subseteq \bar J$ by \Cref{thm:rees}.

    Now, let $L$ be an arbitrary field, and set $S' = \bar{L}[x_1,\dots, x_n]$. We then have $\bar{J} = \bar{JS'}\cap S\supseteq \mf^dS'\cap S = \mf^d$ by \Cref{prop:int-closure} (vii) and the infinite field case. 
\end{proof}
\begin{lemma}[\cite{bivia-ausina_log_2016}, Theorem 4]\label{lemma:mixed-mult-limit}
    Let $k$ be an infinite field, $R = k[x_1,\dots, x_n]$, and $I$ an $\mf$-primary ideal. If $>$ denotes the monomial order $>_{\text{rlex}}$ and $\gamma\in \GL_n(k)$ is general, then for all $1\leq j\leq n$ we have 
    \begin{equation}\label{eqn:revlex-mm}
    \lim_{t\to\infty} \frac{e_j(\ini_>(\gamma^{-1}I^t))}{t^j} = e_j(I).
    \end{equation}
\end{lemma} 
To reassure the reader that Bivi\`a-Ausina's arguments hold in arbitrary characteristic, we summarize the main argument below.
\begin{proof}[Proof of \Cref{lemma:mixed-mult-limit}]
By \Cref{prop:mixed-mult-properties} (ii), choose $\gamma\in \GL_n(k)$ so that $e_j(I) = e_j\left(\frac{I + \gamma(x_n,\dots, x_{n-j+1})}{\gamma(x_n,\dots, x_{n-j+1})}\right)$. Set $J = \gamma^{-1} I$. By \Cref{prop:mixed-mult-properties} (ii), let $\gamma_t\in \GL_n(k)$ such that $e_j(\ini_>(J^t)) = e(\pi_j(\mathfrak{\gamma_t^{-1}}\ini_>(J^t)))$.  
	By \cite[Theorem 3.4 (ii)]{smirnov_semicontinuity_2020}, Hilbert-Samuel multiplicity is upper semicontinuous. Applying this semicontinuity to the specializations $\pi_j(\gamma_t^{-1}\ini_>(J^t))\rightsquigarrow \pi_j(\ini_>(J^t))$ and $J^t\rightsquigarrow \ini_>(J^t)$ as in \cite[Proposition 2 and Theorem 4]{bivia-ausina_log_2016}, we deduce
\begin{equation}\label{eqn:double-semicont}
    e(\pi_j(\ini_>(J^t))) \geq  e(\pi_j(\gamma_t^{-1}\ini_>(J^t)) = e_j(\ini_>(J^t))\geq e_j(J^t).
\end{equation}
By \cite[Corollary 1.13]{mustata_multiplicities_2002}, \Cref{eqn:double-semicont}, and \Cref{lem:revlex-proj} we have
\begin{align}\label{align:cf-bivia}
    e_j(I)&= e_j(J) = e(\pi_j(J)) = \lim_{t\to\infty} \frac{e(\pi_j(J)^t)}{t^{j}}\nonumber \\
    &= \lim_{t\to\infty} \frac{e(\ini_>(\pi_j(J^t))}{t^{j}} = \lim_{t\to\infty} \frac{e(\pi_j(\ini_>(J^t)))}{t^j}\nonumber
    \\&\geq\lim_{t\to\infty} \frac{e_j(\ini_>(J^t))}{t^j}\geq \lim_{t\to\infty}\frac{e_j(J^t)}{t^j} = e_j(J).
\end{align}
It follows that the inequalities in \Cref{align:cf-bivia} must be sharp, proving the claim.
\end{proof}
\begin{remark}\label{rmk:cov-necessary}
    The generic change of coordinates in \Cref{lemma:mixed-mult-limit} is necessary when $\dim(R) \geq 3$. Consider $R = \Q[x,y,z]$ and $I = (xy, x^2 + z^2, x^4, y^4, z^4)$. Let $\succ$ denote the graded reverse lexicographic order with $x \succ y \succ z$. Using the ReesAlgebra package in Macaulay2 \cite{M2,ReesAlgebraSource}, we compute $e(\ini_\succ(I^5)) = 2000$. Separately, we have $2000 = 5^3\cdot e(I) = e(I^5)$. For all $t > 0$, we have $\ini_\succ(I^{5t})\subseteq \ini_\succ(I^5)^t$, hence
    \[
    t^3e(I^5) = e(I^{5t})\leq e(\ini_\succ(I^{5t})) \leq e(\ini_\succ(I^5)^t) = t^3e(I^5),
    \]
    Set $\af = \ini_\succ(I^5)$. By \Cref{thm:rees}, we deduce that $\overline{\ini_\succ(I^{5t})} = \overline{\af^t}$ for all $t > 0$. It follows from \Cref{prop:dp-rees} that
    \[
    \lim_{t\to\infty} \frac{e_2(\ini_\succ(I^t))}{t^2} = \lim_{t\to\infty} \frac{e_2(\ini_\succ(I^{5t}))}{(5t)^2} = \frac{e_2(\af^t)}{(5t)^2} = \frac{e_2(\af)}{5^2}.
    \]

    To compute $e_2(\af)$, we use \Cref{prop:mixed-mult-properties} (ii). Suppose $h = a x + b y + c z$ is such that $e_2(\af) = e\left(\frac{\af + (h)}{(h)}\right)$. Let $\gamma\in \GL_n(k)$ denote the map $(x,y,z)\mapsto (ax, by, cz)$. As $\af$ is a monomial ideal, we have $\gamma^{-1}\af = \af$, hence \begin{equation}\label{eqn:general-enough}
    e\left(\frac{\af + \gamma(x+y+z)}{\gamma(x+y+z)}\right) = e\left(\frac{\af + (x+y+z)}{(x+y+z)}\right).
    \end{equation}
    The quantity in \Cref{eqn:general-enough} can be computed directly using the ReesAlgebra package, giving us
    \[
    \lim_{t\to\infty} \frac{e_2(\ini_\succ(I^t))}{t^2} = \frac{e_2(\ini_\succ(I^5))}{5^2} = \frac{22}{5}.
    \]
    Using \Cref{lemma:same-closure-by-maxpower,lemma:mm-hyperplane-section}, we compute $e_2(I) = 4 < \frac{22}{5}$.
\end{remark}

\begin{defn}
    Let $k$ be a field, $R = k[x_1,\dots, x_n], \mf = (x_1,\dots, x_n)$, and let $\af_\bullet$ be a graded system of $\mf$-primary ideals. We define:
    \begin{itemize}
        \item The asymptotic mixed multiplicities: $e_j(\af_\bullet) = \liminf_m \frac{e_j(\af_m)}{m^j}$.
        \item The asymptotic Demailly-Pham invariant: $E_n(\af_\bullet) = \frac{1}{e_1(\af_\bullet)} + \dots + \frac{e_{n-1}(\af_\bullet)}{e_n(\af_\bullet)}$
        \item The asymptotic singularity threshold: $c(\af_\bullet) = \liminf_m mc(\af_m)$.
    \end{itemize}
\end{defn}

Before we prove \Cref{thm:bound} and our asymptotic version of \cite[Theorem 13]{bivia-ausina_log_2016}, we require the following standard facts.
\begin{lemma}\label{lemma:same-closure-by-maxpower}
    Let $L$ be a field and $S = L[x_1,\dots, x_n]$. Let $I$ be a homogeneous ideal of $S$ and $J\subseteq I$ denote the ideal of $S$ generated by the homogeneous forms in $I$ of degree $\leq d$. If $\codim(J) = n$, then $\overline{J} = \overline{I}$.
\end{lemma}
\begin{proof}
    It is clear that $\overline{J}\subseteq \overline{I}$. Let $\mf:= (x_1,\dots, x_n)$. For the reverse containment, note that $I\subseteq J + \mf^{d+1}$. By \Cref{lemma:maxpower-subset} we have
    \[\overline{I} \subseteq \overline{J + \mf^{d+1}}\subseteq \overline{J + \mf^d} \subseteq \overline{J}.\]
\end{proof}
\begin{lemma}\label{lemma:mm-hyperplane-section}
    Let $L$ be a field and $S = L[x_1,\dots, x_n]$. Let $J = (f_1,\dots, f_n)$ be a complete intersection where $\deg f_i = d_i$ and $d_1 \leq \dots \leq d_n$. Then we have the following:
    \begin{enumerate}[(i)]
        \item If $L$ is infinite, then for a general hyperplane section $H\subseteq \Spec S$, we have $e(J|_H) = d_1\cdots d_{n-1}$.
        \item Let $1\leq j\leq n$. With no assumption on $|L|$, we have $e_j(J) = d_1\dots d_j$, and hence $E_j(J) = \frac{1}{d_1} + \cdots + \frac{1}{d_j}$.
    \end{enumerate}
\end{lemma}
\begin{proof}
    As the $f_1,\dots, f_n$ form a homogeneous system of parameters for $S$, we have
\begin{equation}\label{eqn:Bezout}
        e(J) = d_1\dots d_n.
    \end{equation}
    Choosing $H$ generally so that $\codim(f_1,\dots, f_{n-1})|_H = n-1$, 
    \Cref{lemma:same-closure-by-maxpower} gives $e(J|_H) = e((f_1|_H,\dots, f_{n-1}|_H))$, so (i) follows from \Cref{eqn:Bezout}. For (ii), \Cref{lemma:sigma_extension} shows that $e_j(J)$ is invariant under extension of the base field, so it suffices to consider the case of an infinite field. The result follows from (i) and \Cref{prop:mixed-mult-properties} (ii).
\end{proof}
\begin{theorem}\label{thm:bound}
	Let $n$ be a positive integer and $1\leq l\leq n$. Suppose either:
	\begin{enumerate}
		\item $(R,\mf)$ is a regular local ring of dimension $n$, excellent if $\char R = 0$, and that $I\subseteq R$ is an ideal of height at least $l$;
		\item $k$ is a field, $R = k[x_1,\dots, x_n], \mf = (x_1,\dots, x_n)$, and $I\subseteq R$ is an ideal which is $\mf$-primary or homogeneous of height at least $l$.
	\end{enumerate}
	Then $E_l(I)\leq c(I)$.
\end{theorem}
\begin{proof}

We reduce to the case of an $\mf$-primary ideal in a polynomial ring over an infinite field. In the case (2), \Cref{prop:basic-properties-of-c} (vi) and \Cref{lemma:sigma_extension} show that $E_l, c$ are invariant under field extension. If $J$ is an ideal generated by $n-l$ general linear forms, then $E_l(\frac{I+J}{J}) = E_l(I)$ and $c(\frac{I+J}{J}) \leq c(I)$ by \Cref{prop:basic-properties-of-c} (iv) and \Cref{prop:sigma-poly-ring}, so we may assume $l = n$.

	Suppose we are in the case (1).
	By \Cref{prop:basic-properties-of-c} (vi) and \Cref{lemma:sigma_extension}, the quantities $c, E_l$ are invariant under completion and field extension, so we may without loss of generality assume $R$ is complete local with infinite residue field $k$. As $c(I)\geq c((I+J)/J)$ for any ideal $J\subseteq R$, we let $J$ be an ideal generated by $n-l$ elements of $\mf$ whose images in $\mf/\mf^2$ are sufficiently general; it suffices to show that $c((I+J)/J)\geq E_l((I+J)/J)$, where $(I+J)/J$ is an ideal of height $l$ in a complete regular local ring of dimension $l$. Therefore, without loss of generality, we may assume $l = n$, and hence $I$ is $\mf$-primary. 
	Set $R' = k[x_1,\dots, x_n], \mf' = (x_1,\dots, x_n)$. Since $I$ is $\mf$-primary, there exists $I'\subseteq R'$ with $I'R = I$. Another application of \Cref{prop:basic-properties-of-c} (vi) and \Cref{lemma:sigma_extension} gives $c(I) = c(I'R'_{\mf'})$ and $E_n(I) = E_n(I'R'_{\mf'})$, so by \Cref{notation:singularity_threshold} and \Cref{defn:mm-dp-poly-ring} we have $c(I) = c(I'), E_n(I) = E_n(I')$. 
	
	Let $\gamma\in \GL_n(k)$ such that \Cref{eqn:revlex-mm} holds. Set $\af_m = \ini_>(\gamma^{-1}I^m)$ for $m > 0$. By \Cref{lemma:mixed-mult-limit} we have $E_n(\af_\bullet) = E_n(I)$ and by \Cref{prop:semicontinuity} we have $c(\af_\bullet) \leq c(I)$. Let $\mu = \inf \{t: t\mathbf{1} \in \Gamma\}$. Since $\Gamma$ is convex and $\mu\mathbf{1}\in \del\Gamma$, by \cite[Corollary 11.6.1]{Rockafellar_1970} there exists a half-space $H^+\subseteq \R^n$ such that $\Gamma\subseteq H^+$ and that $\mu\mathbf{1}\in \del H^+$. Since $\Gamma$ is closed under translation by elements of $R_{\geq 0}^n$ and the complement of $\Gamma$ in $\R_{\geq 0}^n$ is bounded, the same is true for $H^+$. The closure of the complement of $H^+$ in $\R_{\geq 0}^n$ is therefore a simplex, which we denote by $\conv(\mathbf{0}, (a_1,0,\dots, 0),\dots, (0,\dots, 0, a_n))$. See \Cref{fig:H_with_Gamma_a,fig:H_with_a'} for a depiction of this step.
	\begin{figure}
		\begin{minipage}{0.49\textwidth}
			\centering

\tikzset{every picture/.style={line width=0.75pt}} 

\begin{tikzpicture}[x=0.75pt,y=0.75pt,yscale=-1,xscale=1]

\draw    (190,122) -- (190,310) ;
\draw [shift={(190,120)}, rotate = 90] [color={rgb, 255:red, 0; green, 0; blue, 0 }  ][line width=0.75]    (10.93,-3.29) .. controls (6.95,-1.4) and (3.31,-0.3) .. (0,0) .. controls (3.31,0.3) and (6.95,1.4) .. (10.93,3.29)   ;
\draw    (338,310) -- (190,310) ;
\draw [shift={(340,310)}, rotate = 180] [color={rgb, 255:red, 0; green, 0; blue, 0 }  ][line width=0.75]    (10.93,-3.29) .. controls (6.95,-1.4) and (3.31,-0.3) .. (0,0) .. controls (3.31,0.3) and (6.95,1.4) .. (10.93,3.29)   ;
\draw    (190,150) -- (234.32,245.51) ;
\draw  [draw opacity=0][fill={rgb, 255:red, 0; green, 0; blue, 255}  ,fill opacity=0.25 ] (249.54,258.13) .. controls (243.15,255.75) and (237.8,251.26) .. (234.32,245.51) -- (260,230) -- cycle ; \draw   (249.54,258.13) .. controls (243.15,255.75) and (237.8,251.26) .. (234.32,245.51) ;  
\draw    (249.54,258.13) -- (310,270) ;
\draw    (310,270) -- (338.01,272.8) ;
\draw [shift={(340,273)}, rotate = 185.71] [color={rgb, 255:red, 0; green, 0; blue, 0 }  ][line width=0.75]    (10.93,-3.29) .. controls (6.95,-1.4) and (3.31,-0.3) .. (0,0) .. controls (3.31,0.3) and (6.95,1.4) .. (10.93,3.29)   ;
\draw  [draw opacity=0][fill={rgb, 255:red, 0; green, 0; blue, 255}  ,fill opacity=0.25 ] (190,120) -- (340,120) -- (340,273) -- (310,270) -- (249.54,258.13) -- (260,230) -- (234.32,245.51) -- (190,150) -- cycle ;
\draw [color={rgb, 255:red, 208; green, 2; blue, 27 }  ,draw opacity=1 ]   (190,310) -- (244,256) ;
\draw [color={rgb, 255:red, 0; green, 0; blue, 0 }  ,draw opacity=1 ] [dash pattern={on 3.75pt off 3pt on 7.5pt off 1.5pt}]  (190,218.73) -- (247,258) ;
\draw    (338,310) -- (190,310) ;
\draw [shift={(340,310)}, rotate = 180] [color={rgb, 255:red, 0; green, 0; blue, 0 }  ][line width=0.75]    (10.93,-3.29) .. controls (6.95,-1.4) and (3.31,-0.3) .. (0,0) .. controls (3.31,0.3) and (6.95,1.4) .. (10.93,3.29)   ;
\draw    (190,122) -- (190,310) ;
\draw [shift={(190,120)}, rotate = 90] [color={rgb, 255:red, 0; green, 0; blue, 0 }  ][line width=0.75]    (10.93,-3.29) .. controls (6.95,-1.4) and (3.31,-0.3) .. (0,0) .. controls (3.31,0.3) and (6.95,1.4) .. (10.93,3.29)   ;
\draw [color={rgb, 255:red, 0; green, 0; blue, 0 }  ,draw opacity=1 ] [dash pattern={on 3.75pt off 3pt on 7.5pt off 1.5pt}]  (247,258) -- (304,297.27) ;
\draw [color={rgb, 255:red, 0; green, 0; blue, 0 }  ,draw opacity=1 ] [dash pattern={on 3.75pt off 3pt on 7.5pt off 1.5pt}]  (265.2,270.55) -- (322.2,309.82) ;

\draw (280,180) node [anchor=north west][inner sep=0.75pt]    {$\Gamma (\mathfrak{a}_{\bullet })$};
\draw (220.67,276.07) node [anchor=north west][inner sep=0.75pt]  [color={rgb, 255:red, 208; green, 2; blue, 27 }  ,opacity=1 ]  {$\mu \mathbf{1}$};
\draw (194.67,236.07) node [anchor=north west][inner sep=0.75pt]    {$H$};

\end{tikzpicture}
\caption{$\Gamma(\af_\bullet)$, together with $\mu\mathbf{1}$ and $H$.}
\label{fig:H_with_Gamma_a}
		\end{minipage}
		\begin{minipage}{0.49\textwidth}
	\centering

\tikzset{every picture/.style={line width=0.75pt}} 

\begin{tikzpicture}[x=0.75pt,y=0.75pt,yscale=-1,xscale=1]

\draw    (190,122) -- (190,310) ;
\draw [shift={(190,120)}, rotate = 90] [color={rgb, 255:red, 0; green, 0; blue, 0 }  ][line width=0.75]    (10.93,-3.29) .. controls (6.95,-1.4) and (3.31,-0.3) .. (0,0) .. controls (3.31,0.3) and (6.95,1.4) .. (10.93,3.29)   ;
\draw    (338,310) -- (190,310) ;
\draw [shift={(340,310)}, rotate = 180] [color={rgb, 255:red, 0; green, 0; blue, 0 }  ][line width=0.75]    (10.93,-3.29) .. controls (6.95,-1.4) and (3.31,-0.3) .. (0,0) .. controls (3.31,0.3) and (6.95,1.4) .. (10.93,3.29)   ;
\draw  [draw opacity=0][fill={rgb, 255:red, 0; green, 0; blue, 255}  ,fill opacity=0.25 ] (190,120) -- (340,120) -- (340,273) -- (340,310) -- (322.2,309.82) -- (244,256) -- (190,218.73) -- (190,150) -- cycle ;
\draw [color={rgb, 255:red, 208; green, 2; blue, 27 }  ,draw opacity=1 ]   (190,310) -- (244,256) ;
\draw [color={rgb, 255:red, 0; green, 0; blue, 0 }  ,draw opacity=1 ] [dash pattern={on 3.75pt off 3pt on 7.5pt off 1.5pt}]  (190,218.73) -- (247,258) ;
\draw    (338,310) -- (190,310) ;
\draw [shift={(340,310)}, rotate = 180] [color={rgb, 255:red, 0; green, 0; blue, 0 }  ][line width=0.75]    (10.93,-3.29) .. controls (6.95,-1.4) and (3.31,-0.3) .. (0,0) .. controls (3.31,0.3) and (6.95,1.4) .. (10.93,3.29)   ;
\draw    (190,122) -- (190,310) ;
\draw [shift={(190,120)}, rotate = 90] [color={rgb, 255:red, 0; green, 0; blue, 0 }  ][line width=0.75]    (10.93,-3.29) .. controls (6.95,-1.4) and (3.31,-0.3) .. (0,0) .. controls (3.31,0.3) and (6.95,1.4) .. (10.93,3.29)   ;
\draw [color={rgb, 255:red, 0; green, 0; blue, 0 }  ,draw opacity=1 ] [dash pattern={on 3.75pt off 3pt on 7.5pt off 1.5pt}]  (247,258) -- (304,297.27) ;
\draw [color={rgb, 255:red, 0; green, 0; blue, 0 }  ,draw opacity=1 ] [dash pattern={on 3.75pt off 3pt on 7.5pt off 1.5pt}]  (265.2,270.55) -- (322.2,309.82) ;

\draw (200,180) node [anchor=north west][inner sep=0.75pt]    {$\overline{\R_{\geq }^{n}\setminus H^+}=\Gamma (\mathfrak{a} '_{\bullet })$};
\draw (220.67,276.07) node [anchor=north west][inner sep=0.75pt]  [color={rgb, 255:red, 208; green, 2; blue, 27 }  ,opacity=1 ]  {$\mu \mathbf{1}$};
\draw (194.67,236.07) node [anchor=north west][inner sep=0.75pt]    {$H$};

\end{tikzpicture}
\caption{$\Gamma(\af'_\bullet)$ in terms of $H$.}
\label{fig:H_with_a'}
	\end{minipage}	
	\end{figure}
	
	Define a graded system of monomial ideals $\af_\bullet$ by $\af'_m = \{\x^\ab: \ab\in mH^+\}$. By assumption that $\Gamma\subseteq H^+$, we have $\af_m\subseteq \af'_m$ for all $m$. Consequently, we have $E_n(\af_\bullet)\leq E_n(\af'_\bullet)$ by \Cref{prop:dp-rees}. By \Cref{prop:monomial-threshold}, we also have $c(\af_\bullet) = c(\af'_\bullet)$. We can compute $E_n(\af'_\bullet)$ from the numbers $a_1,\dots, a_n$: by observing that 
	    \[
    \overline{\left(x_1^{\ceil{ma_1}},\dots, x_n^{\ceil{ma_n}}\right)}\subseteq \af'_m \subseteq \overline{\left(x_1^{\floor{ma_1}},\dots, x_n^{\floor{ma_n}}\right)},
    \]
	it follows from \Cref{lemma:mm-hyperplane-section} that $E_n(\af'_\bullet) = \frac{1}{a_1}+\dots + \frac{1}{a_n}$. All together, we have
	    \[
    E_n(I) = E_n(\af_\bullet) \leq E_n(\af'_\bullet) = \frac{1}{a_1}+\dots+\frac{1}{a_n} = c(\af'_\bullet) = c(\af_\bullet) \leq c(I).
    \]
\end{proof}
\begin{remark}
    The insight depicted in \Cref{fig:H_with_Gamma_a,fig:H_with_a'} appears earlier than the work of Demailly, Pham and Bivi\`a-Ausina, in the proof \cite[Theorem 1.1]{de_fernex_multiplicities_2004}. 
\end{remark}
\begin{cor}\label{cor:asymptotic-monomial}
    Let $k$ be an uncountably infinite field, $R = k[x_1,\dots, x_n]$, and $I$ an $\mf$-primary homogeneous ideal. Let $>$ denote the graded reverse lexicographic order and let $\af_m:= \gin(I^m)$. Lastly, let $\bb_1,\dots, \bb_n$ denote the standard unit vectors in $\R^n$. If $E_n(I) = c(I)$, then
    \begin{equation}\label{eq:simplex-complement}
        \overline{\R^n_{\geq 0}\setminus \Gamma(\af_\bullet)} = \conv\left(\mathbf{0}, e_1(I)\bb_1, \frac{e_2(I)}{e_1(I)}\bb_2,\dots,\frac{e_{n}(I)}{e_{n-1}(I)}\bb_n\right).
    \end{equation}
\end{cor}
\begin{proof}
    By construction, $>$ agrees with the reverse lexicographic order on homogeneous ideals. Consequently, we may choose $\gamma\in \GL_n(k)$ very generally so that \Cref{eqn:revlex-mm} holds and $\ini_>(\gamma^{-1}I^m) = \gin_>(I^m)$ for all $m > 0$. With this choice of $\gamma$, let $\af_\bullet, \af'_\bullet$ be as in \Cref{thm:bound}. 

    Suppose $E_n(I) = c(I)$. Then we also have $E_n(\af_\bullet) = E_n(\af'_{\bullet})$. By \cite[Proposition 10]{bivia-ausina_log_2016}, we further have that $e_j(I) = e_j(\af_{\bullet}) = e_j(\af'_\bullet)$ for all $1\leq j\leq n$. In particular, $e_n(\af_\bullet) = e_n(\af'_\bullet)$, so by \cite[Theorem 2.12 and Lemma 2.13]{mustata_multiplicities_2002}, we have $\vol(\R_{\geq 0}^n\setminus \Gamma(\af_\bullet)) = \vol(\R_{\geq 0}^n\setminus \Gamma(\af'_\bullet))$. Since $\Gamma(\af_\bullet), \Gamma(\af'_\bullet)$ are closed and convex with positive volume, it follows that $\Gamma(\af_\bullet) = \Gamma(\af'_\bullet).$ Since the generic initial ideal is Borel-fixed, we have $a_1\leq \dots \leq a_n$, hence $e_j(\af'_\bullet) = a_1\dots a_j$ for all $1\leq j\leq n$, proving the claim.
\end{proof}
\begin{remark}
   \Cref{eq:simplex-complement} is a necessary condition to have $c(I)=E_n(I)$, but not a sufficient condition. By \cite{mayes_limiting_2014} in characteristic zero, \Cref{eq:simplex-complement} holds for any homogeneous complete intersection $J = (f_1,\dots, f_n)$. 
\end{remark}
In fact, \Cref{eq:simplex-complement} holds for any homogeneous complete intersection in positive characteristic, though a small modification to Mayes's argument in \cite{mayes_limiting_2014} must be made. In particular, Lemma 3.6 in op. cit. does not hold in positive characteristic: in the ring $\overline{\F_p}[x,y]$ we have $\gin(x^p, y^p) = (x^p, y^p)$. Instead, one appeals to \Cref{thm:Briancon-Skoda}. The details will appear in the author's thesis.

\begin{example}[\cite{dF_counter_2026}]\label{ex:pinched-polytope}
Here we give a counterexample to \cite[Question 1.3]{kim_teissier_2021}, which also functions as a counterexample to the analogous statement in positive characteristic. Let $k$ be a field and set $R = k[x,y,z]$. The idea is to start with an ideal $\cf\subseteq R$ and a general hyperplane $H$ for which the equality $c(\cf) - c(\cf|_H) = \frac{e_2(\cf)}{e_3(\cf)}$ holds and to produce an ideal $\af\supseteq \cf$ for which $c(\af) = c(\cf), c(\af|_H) = c(\cf|_H), e_2(\af) = e_2(\cf)$, and $e(\af) < e(\cf)$. 

We'll define a family of monomial ideals satisfying the desired conditions, which we depict in \Cref{fig:gamma_c,fig:gamma_a}. Let $3\leq a < b$. We define the auxiliary ideal $\bfr = (x^a, xy, y^a)$. Set $\cf =  (\bfr, z^b)$ and $\af = (\cf, zx^{a-1})$. As $a \geq 3$, we have $\af\subseteq \overline{(x^2, y^2, z^b)}$. Using \Cref{prop:monomial-threshold}, we compute
\[
1 + \frac{1}{b} = c((xy, z^b)) \leq c(\cf) \leq c(\af) \leq c(\overline{(x^2,y^2,z^b)}) = 1 + \frac{1}{b}.
\]
As in \Cref{rmk:cov-necessary}, the linear form $h = x-y-z$ is sufficiently general to compute $e_2(\af)$ and $\max_\ell c(\af|_{\ell=0})$. Identifying $R/hR$ with $k[x,y]$ and letting $\af',\bfr',\cf'$ denote the images of $\af,\bfr,\cf$ in $R/hR$, we have
\[
\bfr' \subseteq \cf' \subseteq \af' = (\bfr',(x+y)x^{a-1}, (x+y)^b) \subseteq \overline{\bfr'},
\]
so we have $\overline{\af'} = \overline{\bfr'} = \overline{\cf'}$, which implies that $e_2(\af) = e_2(\cf)$ and $c(\af') = c(\cf')$. 

 By \Cref{lemma:maxpower-subset} and \Cref{lemma:mm-hyperplane-section} respectively, we have $\bfr' = \overline{(xy, x^a + y^a)}$ and $e(\bfr') = 2a$. A similar computation yields $\cf = \overline{(xy, x^a+y^a, z^b)}$, so $e(\cf) = 2ab$. On the other hand, we claim that the containment $\overline{\cf}\subseteq \overline{\af}$ is strict --- in particular, we claim that $zx^{a-1}\notin \overline{\cf}$. To see this, we note that the monomials in $k[x,z]\cap \overline{\cf}$ are precisely those $x^iz^j$ for which $i/a + j/b \geq 1$, and we have $(a-1)/a + 1/b < 1$. By \Cref{thm:rees}, it follows that $e(\af) < e(\cf)$. We deduce that
 \[
 c(\af) = 1 + \frac{1}{b} = c(\cf) + \frac{e_2(\cf)}{e(\cf)} < c(\af|_{h=0}) + \frac{e_2(\af)}{e(\af)},
 \]
 violating the conjectured inequality.

\begin{figure}
    \centering
    \begin{minipage}{0.49\textwidth}\label{fig:ctrex}
        \centering
\tdplotsetmaincoords{70}{120}
\begin{tikzpicture}[tdplot_main_coords]
  \def\laxis{5}
  \def\ltriangle{3}
  \def\ltick{.2}
  \draw [->] (0,0,0) -- (\laxis,0,0) node [below] {$x$};
  \draw [->] (0,0,0) -- (0,\laxis,0) node [right] {$y$};
  \draw [->] (0,0,0) -- (0,0,\laxis) node [left] {$z$};
  \pgfmathtruncatemacro{\nticks}{floor(\laxis)-1}
  \begin{scope}[
    help lines,
    every node/.style={inner sep=1pt,text=black}
    ]
    \draw (3, \ltick, 0) -- ++(0,-\ltick,0) -- ++(0,0,\ltick)
      node [pos=1,left] {$a$};
   	\draw (\ltick, 3, 0) -- ++(-\ltick,0,0) -- ++(0,0,\ltick)
      node [pos=1,right] {$a$};
    \draw (\ltick,0,4) -- ++(-\ltick,0,0) -- ++(0,\ltick,0)
      node [at start,above right] {$b$};
  \end{scope}
  \filldraw [opacity=.25,blue] (3,0,0) -- (1,1,0) -- (0,3,0)
  -- (0,0,4) -- cycle;
  \draw [opacity=1,black,dashed] (1,1,0)--(0,0,4) -- (0,3,0) -- (1,1,0)--(3,0,0) -- (0,0,4);
  \draw[opacity=1,red,thick] (0,0)--(0.8,0.8,0.8);
\end{tikzpicture}
        \caption{$\Gamma(\cf)$, with $\mu\mathbf{1}$ computing $c(\cf)$ in red.}
        \label{fig:gamma_c}
    \end{minipage}\hfill
    \begin{minipage}{0.49\textwidth}
        \centering
 \tdplotsetmaincoords{70}{120}
\begin{tikzpicture}[tdplot_main_coords]
  \def\laxis{5}
  \def\ltriangle{3}
  \def\ltick{.2}
  \draw [->] (0,0,0) -- (\laxis,0,0) node [below] {$x$};
  \draw [->] (0,0,0) -- (0,\laxis,0) node [right] {$y$};
  \draw [->] (0,0,0) -- (0,0,\laxis) node [left] {$z$};
  \pgfmathtruncatemacro{\nticks}{floor(\laxis)-1}
  \begin{scope}[
    help lines,
    every node/.style={inner sep=1pt,text=black}
    ]
    \draw (3, \ltick, 0) -- ++(0,-\ltick,0) -- ++(0,0,\ltick)
      node [pos=1,left] {$a$};
   	\draw (\ltick, 3, 0) -- ++(-\ltick,0,0) -- ++(0,0,\ltick)
      node [pos=1,right] {$a$};
    \draw (\ltick,0,4) -- ++(-\ltick,0,0) -- ++(0,\ltick,0)
      node [at start,above right] {$b$};
  \end{scope}
  \filldraw [opacity=.25,green] (0,0,4)  -- (1,1,0) -- (2,0,1) -- cycle;
  \filldraw [opacity=.25,blue] (0,3,0) -- (0,0,4) -- (1,1,0) -- cycle;
  \filldraw [opacity=.25,yellow] (3,0,0) -- (2,0,1) -- (1,1,0) -- cycle;
  \draw[opacity=1,red, thick] (0,0)--(0.8,0.8,0.8);
\draw [opacity=1,black,dashed] (0,0,4)--(2,0,1) -- (3,0,0) -- (1,1,0) --(0,3,0) -- cycle;
\draw [opacity=1,black,dashed] (1,1,0)--(2,0,1);
\draw [opacity=1,black,dashed] (0,0,4)-- (1,1,0);
\end{tikzpicture}\caption{$\Gamma(\af)$, with $\mu\mathbf{1}$ computing $c(\af)$ in red.}
\label{fig:gamma_a}
    \end{minipage}
\end{figure}
\end{example}
In light of \cite[Proposition 2.4 (i)]{kim_teissier_2021}, we might instead consider the weaker bound $c(I) - c(I|_H) \geq \frac{1}{\Lc_0(I)}$, where $\Lc_0(I)$ is the \L{}ojasiewicz exponent (\Cref{defn:lojasiewicz-exponent}). This does  hold; see \Cref{lemma:lojasiewicz-bound}.
\section{Theorem B in the Complete Intersection Case}
\subsection{Behavior of the Singularity Threshold Under Modification}
Fix the following notation throughout this subsection. 
\begin{assumption}\label{ass:mod-setup}
    Let $k$ be a field, $R = k[x_1,\dots, x_n]$, and let $\mf$ denote the homogeneous maximal ideal. Let $I\subseteq R$ be an $\mf$-primary homogeneous ideal. Write $I = I_1 + \dots + I_r$, where $I_j$ is generated by forms of degree $d_j$ and $d_1 < \dots  < d_r$.
\end{assumption}

\begin{lemma}[\cite{bhatt_f_2015}, Lemma 3.2]\label{lem:colon}
    Let $R = k[x_1,\dots, x_n]$ and let $\mf$ denote the homogeneous maximal ideal. For any $e,t\in \Z^+$, we have
\[
(\mf^{[p^e]}: \mf^t) = \begin{cases}
    R & t \geq np^e - n + 1\\
    \mf^{[p^e]} + \mf^{np^e-n+1-t} & t < np^e-n+1
\end{cases}
\]
\end{lemma}
More generally, we have the following.
\begin{lemma}\label{lem:valuation-colon}
    Let $R = k[x_1,\dots, x_n]$. Let $v$ be a monomial valuation on $R$ with $v(x_i)\geq 0$ for all $1\leq i\leq n$. For $\l\in \R^+$, let $\af_\l$ denote the ideal $\{f\in R: v(f)\geq \l\}$ and $\af_\l^+ = \{f\in R: v(f) > \l\}$. Let $q\in \Z^+, \l\in \R^+$. Then we have
    \begin{equation}\label{eq:valuation-colon}
        ((x_1^q,\dots, x_n^q): \af_\l) = (x_1^q,\dots, x_n^q) + \af^+_{(q-1)v(x_1\dots x_n) - \l}.
    \end{equation}
\end{lemma}
\begin{proof}
    The argument is the same as \Cref{lem:colon}. Let $m\notin (x_1^q,\dots, x_n^q)$ be a monomial. Then $m\mid (x_1\cdots x_n)^{q-1}$, so $\af_\l m\not\subseteq (x_1^q,\dots, x_n^q)$ if and only if $\frac{(x_1\dots x_n)^{q-1}}{m}\in \af_\l$, which holds if and only if $v((x_1\cdots x_n)^{q-1}) - v(m) \leq \l$. We've shown that the two sides of \Cref{eq:valuation-colon} contain the same monomials; both sides are monomial ideals, so the result follows.
\end{proof}
\begin{lemma}[\cite{baily_equigenerated_2025}, Lemma 4.2]\label{lemma:nu-I-H}
    Let $k$ be a field of characteristic $p>0$, let $R = k[x_1,\dots, x_n]$, and $I\subseteq R$ a homogeneous ideal. For a hyperplane $H$ cut out by a linear form $\ell$, we let $I|_H$ denote the image of $I$ in $R/\ell R$. In this case, we have
    \[
    \nu_{I|_H}(p^e) \leq \max \{r: I^r\not\subseteq \mf^{[p^e]} + \mf^{(n-1)(p^e-1)+1}\},
    \]
\end{lemma}
\begin{defn}\label{defn:lojasiewicz-exponent}
	Let $R = k[x_1,\dots, x_n]$ and let $I\subseteq R$ be an ideal. The \textbf{\L{}ojasiewicz exponent} of $I$, denoted $\Lc_0(I)$, is equal to the least integer $t$ such that $\mf^t\subseteq \overline{I}$, or else $\infty$ in the case that $I$ is not $\mf$-primary. 
\end{defn}
\begin{lemma}\label{lemma:lojasiewicz-bound}
	Let $R = k[x_1,\dots, x_n]$ and let $I\subseteq R$ be an ideal. If $H\subseteq \Spec R$ is a hyperplane, then we have $c(I) - c(I|_{H})\geq \frac{1}{\Lc_0(I)}$.
\end{lemma}
\begin{proof}
	If $\Lc_0(I) = \infty$ then there is nothing to prove, so suppose $\Lc_0(I) = d < \infty$.
	
	We first prove the claim in characteristic $p>0$. Combining \Cref{lem:colon} and \Cref{lemma:nu-I-H}, we have
    \[
    \nu_{I|_H}(p^e) \leq \max\{s: \mf^{p^e}I^s\not\subseteq \mf^{[p^e]}\}.
    \]
    Consequently, we have $\max\{s: \mf^{p^e}I^s\not\subseteq \mf^{[p^e]}\} \leq \nu_{I}(p^e) - \floor{\frac{p^e}{d}}$, which implies that $\nu_{I|_H}(p^e) \leq \nu_{I}(p^e) - \floor{\frac{p^e}{d}}$. Dividing by $p^e$ and taking the limit as $e\to\infty$ gives the result. 
    
    We now prove the claim in characteristic 0. The result follows from \Cref{prop:spread_out}, the details of which we spell out explicitly in this case (though moving forward, we will be more terse). Let $A\subseteq k$ be a finitely-generated $\Z$-algebra and $J\subseteq A[x_1,\dots, x_n]$ an ideal such that $JR = I$. Such a subring $A$ can always be constructed by adjoining to $\Z$ the field coefficients appearing in a generating set for $I$; we'll also choose $A$ to contain the coefficients of the linear form $\ell$. If $\mu$ is a maximal ideal of $A$, we let $I_\mu$ denote the image of $J$ in $(A/\mu)[x_1,\dots, x_n]$. Applying \Cref{prop:spread_out} to both $I$ and $I|_H$, we obtain
	\[
    c(I) - c(I|_H) = \lim_{\substack{\mu\in\Spec A\\\text{char} A/\mu\to\infty}} c(I|_\mu) - c(I_\mu|_{H_\mu}) \geq 1/d.
    \]
\end{proof}
\begin{cor}\label{cor:m-primary-restriction}
    Assume the setup of \Cref{ass:mod-setup} and let $H\subseteq \Spec R$ be a hyperplane cut out by a linear form $\ell$. Then $c(I) - c(I|_H)\geq 1/d_r$.
\end{cor}
\begin{proof}
    By \Cref{lemma:maxpower-subset}, we have $\mf^{d_r}\subseteq \bar{I}$, so $\Lc_0(I) \leq d_r$. The result follows from \Cref{lemma:lojasiewicz-bound}.
\end{proof}
\begin{lemma}\label{lemma:two-degree-formula}
    Assume the setup of \Cref{ass:mod-setup}. Suppose $r=2$. Then we have $c(I) = \frac{n}{d_2} + c(I_1)\frac{d_2-d_1}{d_2}$.
    In particular, for any $0\leq s\leq n$, we have $c(I) = \frac{s}{d_1} + \frac{n-s}{d_2}$ if and only if $c(I_1) = \frac{s}{d_1}$.
\end{lemma}
\begin{proof}
    We prove the claim first in positive characteristic. By \Cref{lemma:maxpower-subset}, we have $\mf^{d_2}\subseteq \overline{I}$, so ${I} \subseteq I_1+ \mf^{d_2}\subseteq \overline{I}$, so $c(I) = c(I_{1} + \mf^{d_2})$. Consequently, we have
    \[
    \nu_{\bar{I}}(p^e) = \max\left\{a+b: I_{1}^a\mf_\mu^{bd_n}\not\subseteq \mf^{[p^e]}\right\} = \max\{a+b: I_1^a \not\subseteq (\mf_p^{[p^e]}:\mf_p^{bd_n})\}.
    \]
    By \Cref{lem:colon}, we obtain
    \begin{equation}\label{eqn:application_of_colon_lemma}
    \nu_{\bar{I}}(p^e) = \max\{a+b: I_{1}^a \not\subseteq \mf^{[p^e]} + \mf^{np^e - n + 1 - bd_2} \}.
    \end{equation}
    The ideal $I_1^a$ is generated in degree $ad_1$, so $I_1^a\not\subseteq \mf^{[p^e]}+\mf^t$ if and only $I_1^a\not\subseteq \mf^{[p^e]}$ and $ad_1 < t$. With this insight, we refine \Cref{eqn:application_of_colon_lemma}:
    
   \begin{equation}\label{eqn:nu-in-terms-of-a}
     \nu_{I}(p^e) = \max_{0\leq a\leq \nu_{I_{1}}(p^e)} a + \frac{np^e-n - ad_1}{d_2}.
    \end{equation}
    The quantity being maximized in \Cref{eqn:nu-in-terms-of-a} is an increasing function of $a$, so the maximum occurs at $a = \nu_{I_1}(p^e)$ and
    \[
    \nu_{\bar{I}}(p^e) = \frac{np^e-n}{d_2} + \nu_{I_{1}}(p^e)\frac{d_2-d_1}{d_2}.
    \]
    Dividing by $p^e$ and letting $e\to \infty$, we obtain
    \[
    c(I) = \frac{n}{d_2} + c(I_{1})\frac{d_2-d_1}{d_2}.
    \]
    For characteristic zero, we can spread out to positive characteristic as in \Cref{cor:m-primary-restriction} and compute
    \[
c(I) = \lim_\mu c(I_\mu) = \lim_\mu \left(\frac{n}{d_2} + c(I_{1,\mu})\frac{d_2-d_1}{d_2}\right) = \frac{n}{d_2} + c(I_1)\frac{d_2-d_1}{d_2}.
\]
\end{proof}

\subsection{Induction Setup}
In this subsection, we begin the proof of \Cref{thm:min-char} in the case that $I$ is a complete intersection and $k = \bar{k}$.
\begin{assumption}\label{ass:mod-setup-2}
    We assume a setup similar to \Cref{ass:mod-setup}. Let $k$ be an algebraically closed field. Let $a_1,\dots, a_r, d_1,\dots, d_r\in \Z^+$. For $1\leq i\leq r$, let $\x_i$ denote the tuple of variables $x_{i,1},\dots, x_{i,a_i}$, and let $R = k[\x_1,\dots, \x_r]$. Let $n = a_1+\dots + a_r = \dim R$. Let $I\subseteq R$ be a complete intersection of the form $(f_{1,1},\dots, f_{r,a_r})$ such that $f_{i,j}$ is a $d_i$-form. For $1\leq j\leq r$, write $I_j = (f_{j,1},\dots, f_{j,a_j})$. Let $w_d$ denote the monomial valuation with $w_d(x_{i,j}) = 1/d_i$. We let $\Df$ denote the ideal $\overline{(\x_1^{d_1},\dots, \x_r^{d_r})}$, which coincides with the set of elements $z$ with $w_d(z)\geq 1$. 
\end{assumption}
\begin{assumption}
        Assume the setup of \Cref{ass:mod-setup-2}. We define the following condition on the ideal $I$:
    \begin{equation}\label{eqn:aligned}
            I_1 \text{ is extended from } k[\x_1] \text{ and } I\subseteq \Df + (\x_1)\tag{$\dagger$}
    \end{equation}
\end{assumption}
\begin{defn}
    For $r\in \Z^+$, we define the statements $A_r, B_r$.
\begin{equation}\tag{$A_r$}
    \begin{split}
         &\text{For all $I$ as in \Cref{ass:mod-setup-2}, if $E_n(I) = c(I)$ then}\\
         &\text{there exists $\gamma\in \GL_n(k)$such that $\gamma I\subseteq \Df$.}
    \end{split}
\end{equation}
\begin{equation}\tag{$B_r$}
    \begin{split}
         &\text{For all $I$ as in \Cref{ass:mod-setup-2}, if $E_n(I) = c(I)$ then }\\
         &\text{there exists $\gamma\in \GL_n(k)$ such that $\gamma I$ satisfies \Cref{eqn:aligned}.}
    \end{split}
\end{equation}
\end{defn}
The goal of this section is to prove $A_r$ for all $r$. We accomplish this via the following steps:
\begin{enumerate}[(1)]
    \item $A_1,A_2$ hold;
    \item For $r \geq 3$, $A_{r-1}$ combined with $A_2$ implies $B_r$;
    \item For $r\geq 3$, $B_r$ implies $A_r$.
\end{enumerate}
The proof of step (1) is short, but steps (2) and (3) will each have a dedicated subsection. 
\begin{proof}[Proof of $A_1$]
    If $I$ is an $\mf$-primary complete intersection generated by forms of degree $d$ for some $d$, by \Cref{lemma:maxpower-subset} we have $\overline{I} = \mf^d$. 
\end{proof}
\begin{remark}
   When $r = 1$, the above argument shows that $c(I) = E_n(I)$ is satisfied for all choices of $I$. Unfortunately, this means that the proof of our induction step $A_2,A_{r-1}\implies A_r$ can't be modified to show $A_1\implies A_2$. We instead expand the scope of our base case.
\end{remark}
To prove $A_2$, we introduce the notion of \textit{essential dimension}.
\begin{defn}[Essential Dimension]\label{defn:ess-dim}
        Let $J\subseteq R = k[x_1,\dots, x_d]$ be a homogeneous ideal. The essential dimension $\ef(J)$ is equal to the minimal $r$ for which there exist linear forms $\ell_1,\dots, \ell_r$ such that $J$ is extended from $I\subseteq k[\ell_1,\dots, \ell_r]$. 
    \end{defn}
    We have the following result.
   \begin{prop}[\cite{baily_equigenerated_2025}, Proposition 3.3]\label{prop:essential-Bertini}
       Let $k$ be an algebraically closed field, $R = k[x_0,\dots, x_n]$, and $J \subseteq R$ a homogeneous ideal. Set $r = \codim(J)$. Let $L = (\ell_{r+1},\dots, \ell_n)$, where the $\ell_i$ are chosen generally. For $r\leq t\leq n$, set $L_t = (\ell_{t+1}, \dots, \ell_n)$ and $J_t= \frac{J+L_t}{L_t}$. Then for all $r\leq t\leq n$, we have $\ef(J_t) = \max(t+1, \ef(J))$.
   \end{prop}
\begin{proof}[Proof of $A_2$]
    By \Cref{lemma:two-degree-formula}, if $c(I) = E_n(I)$ then $c(I_1) = a_1/d_1$. By \cite[Theorem 3.5]{de_fernex_bounds_2003} in characteristic zero and \cite[Theorem 3.17]{baily_equigenerated_2025} in positive characteristic, it follows that $\ef(I_1) = a_1$. By \cite[Lemma 3.18]{baily_equigenerated_2025}, there exists $\gamma\in\GL_n(k)$ such that $\overline{\gamma(I_1)} = (\x_1)^{d_1}$, hence $\gamma I\subseteq (\x_1)^{d_1} + \mf^{d_2}\subseteq \Df$.
\end{proof}

\subsection{Step (2): \texorpdfstring{For $r \geq 3$, $A_{r-1}$ combined with $A_2$ implies $B_r$}{For r>=3, Ar-1 combined with A2 implies Br}.}
\begin{lemma}\label{lemma:lifting-I1}
    Assume the setup of \Cref{ass:mod-setup-2} and suppose $r\geq 3, c(I) = E_n(I)$. If $A_{r-1}$ holds, there exists $\gamma\in \GL_n(k)$ such that $\gamma \overline{I_1} = (\x_1)^{d_1}$.
\end{lemma}
\begin{proof}
    Note that $E_n(I) = a_1/d_1+\dots +a_r/d_r$ by \Cref{lemma:mm-hyperplane-section}. Let $L$ be an ideal of $R$ generated by $a_3+\dots + a_d$ general linear forms. Since $I$ is a complete intersection, $I_1 + I_2 + L$ is $\frac{\mf}{L}$-primary, so by \Cref{lemma:same-closure-by-maxpower} we have $c(R/L, \frac{I+L}{L}) = c(R/L, \frac{I_1+I_2+L}{L})$. Consequently, by repeated application of \Cref{cor:m-primary-restriction}, we have
    \begin{equation}\label{eqn:slicebound}
        c(I) \geq c\left(R/L, \frac{I_1+I_2+L}{L}\right) + \frac{a_3}{d_3} + \dots + \frac{a_r}{d_r}.
    \end{equation}
   Assuming $c(I)=E_n(I)$, we have 
    \[
    \frac{a_1}{d_1} + \frac{a_2}{d_2}\leq c\left(R/L, \frac{I_1+I_2+L}{L}\right)\leq \frac{a_1}{d_1} + \frac{a_2}{d_2},\] 
    where the left-hand side is by \Cref{thm:bound} and the right-hand side is by \Cref{eqn:slicebound}. Both inequalities are therefore equalities, so by the proof of $A_2$, we have
     \(\ef\left(\frac{I_1+L}{L}\right) = a_1\).
      By \Cref{prop:essential-Bertini}, it follows that $\ef(I_1) = a_1$. The result then follows from \cite[Lemma 3.18]{baily_equigenerated_2025}.
\end{proof}

\begin{lemma}\label{lemma:change-of-vars}
    Assume the setup of \Cref{ass:mod-setup-2}. Suppose $c(I) = E_n(I)$. Then there exists $\gamma\in \GL_n(k)$ such that $\gamma I$ satisfies \Cref{eqn:aligned}.
\end{lemma}
\begin{proof}
    By \Cref{lemma:lifting-I1}, we may assume $I_1$ is extended from $k[\x_1]$.
    Let $\succ$ denote the monomial partial order induced by the monomial valuation $w(x_{1,i}) = 0$ and $w(x_{i,j})=1$ for $i\geq 2$. 
      For $2\leq i\leq r, 1\leq j\leq a_i$, let $g_{i,j}:= \ini_\succ(f_{i,j})$. Since $I$ is a complete intersection, we have $f_{i,j}\notin \sqrt{I_1} = (\x_1)$, hence $g_{i,j}\notin (\x_1)$ and moreover $f_{i,j} - g_{i,j}\in (\x_1)$. Observe that
    \begin{equation}\label{eqn:separation-from-x_1}
    \ini_\succ(I) \supseteq I_1 + \ini_\succ(I_2 + \dots + I_r)\supseteq I_1 + (g_{2,1},\dots, g_{r,a_r}).
    \end{equation}
    Let $I'$ denote the right-hand side of \Cref{eqn:separation-from-x_1}. Because $g_{i,j}$ and $f_{i,j}$ have the same image modulo $(\x_1) = \sqrt{I_1}$, the ideal $I'$ is a complete intersection. In particular, $I'$ is a complete intersection of type $(\underbrace{d_1,\dots, d_1}_{a_1},\dots, \underbrace{d_r,\dots, d_r}_{a_r})$. By \Cref{lemma:mm-hyperplane-section} and \Cref{prop:semicontinuity}, we have    
    \begin{equation}\label{eqn:precursor-to-thom-sebastiani}
    E_n(I) = E_n(I') \leq c(I') \leq c(\ini_\succ(I))\leq c(I) = E_n(I).
\end{equation}
    As $\overline{I_1} = (\x_1)^{d_1}$, we have $c(I_1) = a_1/d_1$. Since $I_1$ and $J:= (g_{2,1},\dots, g_{r,a_r})$ are defined in terms of disjoint sets of variables, we have by \cite[Theorem 2.4 (1)]{tao_bernstein-sato_2024} that 
    \begin{equation}\label{eqn:thom-sebastiani}
        c(R, I') = c(k[\x_1], I_1) + c(k[\x_2,\dots, \x_r], J)= \frac{a_1}{d_1}+c(k[\x_2,\dots, \x_r], J).
    \end{equation}
    
    It follows from \Cref{eqn:precursor-to-thom-sebastiani,eqn:thom-sebastiani} that
    $J$, which is a complete intersection in $k[\x_2,\dots, \x_r]$, satisfies $E_{a_2+\dots+a_r}(J) = c(J)$. By $A_{r-1}$, there exists $\gamma\in \GL_{n-a_1}(k)$ 
 such that $\gamma J \subseteq \overline{(\x_2^{d_2},\dots, \x_r^{d_r})}$. 
 
 We define $\gamma':= \begin{bmatrix}
    \id_{a_1} & 0\\
    0 & \gamma
 \end{bmatrix}$ and we claim that $\gamma'I$ satisfies \Cref{eqn:aligned}. By construction, $\gamma'I'$ satisfies \Cref{eqn:aligned}. Since $g_{i,j} - f_{i,j}\in (\x_1)$ for all $2\leq i\leq r, 1\leq j\leq a_i$, we have \[\gamma'(I_2+\dots+I_r) \subseteq \gamma'(g_{1,1},\dots, g_{r,a_r}) + (\x_1),\] which proves that $\gamma'I$ also satisfies \Cref{eqn:aligned}. 
\end{proof}
\begin{remark}
	In the proof of \Cref{lemma:change-of-vars}, we used the assumption $A_{r-1}$ to show that $I|_H \subseteq \Df|_H$. This condition is necessary, but not sufficient, to have $I\subseteq \Df$ --- even if $I$ also satisfies \Cref{eqn:aligned}. For example, consider the ideal $I = (x^2, y^3 + xyz, z^7)$ and $\Df = \overline{(x^2, y^3, z^7)}$.
\end{remark}
\subsection{Step (3): \texorpdfstring{$B_r$ implies $A_r$}{Br implies Ar}.}
\subsubsection{Proof Sketch}
By $B_r$, any ideal $I$ as in \Cref{ass:mod-setup-2} with $c(I) = E_n(I)$ can be written in the form
\[
I = (f_{1,1},\dots, f_{1,a_1}) + ((f'_{2,1}+f''_{2,1}) + \dots + (f'_{r,a_r} + f''_{r,a_r}))
\]
where $f'_{i,j}\in k[\x_2,\dots, \x_r]$ and $f''_{i,j}\in (\x_1)$, and $f_{1,1},\dots, f_{1,a_1},f'_{2,1},\dots,f'_{r,a_r}$ are all contained in $\Df$ and form a regular sequence. If $f''_{i,j}\in \Df$ for all $i,j$, then $I\subseteq \Df$ and \Cref{prop:dp-rees} implies that $\overline{I} = \Df$ as desired.

To handle the case that $I\not\subseteq \Df$, we suppose that $I$ is as close to $\Df$ as possible without having $I\subseteq \Df$. Specifically, we consider an ideal $I$ with a unique index $2\leq m\leq r-1$ such that $f''_{i,j} = 0$ for all $i\neq m$ and $f'_{i,j} = x_{i,j}^{d_i}$. We also suppose that for this fixed index $m$, the $f''_{m,j}$ are not all zero and no monomial summand of $f''_{m,j}$ is contained in $\Df$. Because of the (relatively) uncomplicated structure of $I$, we are able to  use positive-characteristic arguments to directly show $c(I) > E_n(I)$ (\Cref{lemma:simple-sufficient-strictness-condition}). 

In the general case (\Cref{lem:application-of-weight-order}), we prove that any complete intersection satisfying \Cref{eqn:aligned} but with $I\not\subseteq \Df$ can be degenerated to an ideal $I'$ which is as described in the previous paragraph: close to $\Df$ while satisfying $I'\not\subseteq \Df$. By \Cref{prop:semicontinuity}, we have $c(I) \geq c(I') > E_n(I') = E_n(I)$.
\subsubsection{Constructing the  Degeneration Order}
\begin{notation}
    Let $\bb_1,\dots, \bb_n$ denote the standard basis vectors for $\Z^n$. For $\ub = (u_1,\dots, u_n)\in (\Z^+)^n$, set $|\ub| = u_1+\dots + u_n$. For $\db = (d_1,\dots, d_n)\in \Z^n$, we also define $\nu_\db: \Z^n\to \Q$ by $\nu_\db(\ub) = \frac{u_1}{d_1} + \dots + \frac{u_n}{d_n}$. 
\end{notation}
\begin{lemma}\label{lemma:final-weight-partial-order}
Let $0<d_1<\dots<d_n\in (\Z^+)^n$ and set $\db = (d_1,\dots, d_n)$. Let $S\subseteq (\Z_{\geq 0})^n$. Write $S = S_2 \sqcup \dots \sqcup S_{n-1}$. Assume that $S$ satisfies the following:
\begin{enumerate}
    \item For all $\ub\in S_i$, we have $|\ub| = d_i$.
    \item For all $\ub\in S$, either $u_1 \geq 1$ or $\nu_\db(\ub)\geq 1$.
    \item For all $2\leq i\leq n-1$, we have $d_i\bb_i\notin S_i$.
    \item There exists $\ub\in S$ with $\nu_\db(\ub) < 1$ (and therefore $u_1\geq 1$). 
\end{enumerate}
Then there exists a linear map $\lambda: \Z^n\to \Z$ and a unique index $2\leq m\leq n-1$ such that
\begin{enumerate}[label=\textbf{(A.\roman*)}]
    \item For all $2\leq i\leq n-1, i\neq m, \ub\in S_i$ we have $\lambda(\ub) < \lambda(d_m\bb_m)$.
    \item We have
    \begin{equation}\label{eq:lambda-max}
    \max_{\ub\in S_m} \lambda(\ub) = \lambda(d_m\bb_m).
    \end{equation}
    \item For any $\ub$ achieving the maximum in \Cref{eq:lambda-max} we have $\nu_\db(\ub) < 1$ and $u_1\geq 1$.
\end{enumerate}
\end{lemma}
To prove \Cref{lemma:final-weight-partial-order}, we construct a sequence of auxiliary linear maps $\lambda_0,\dots, \lambda_s: \Z^n\to \Q$.
\begin{lemma}\label{lemma:zeroth-weight-partial-order}
    Assume the setup of \Cref{lemma:final-weight-partial-order}. There exists a linear map $\lambda_0: \Z^n\to \Q$ such that:
    \begin{enumerate}[label=\textbf{(B.\roman*)}]
        \item For all $2\leq i\leq n-1, \ub\in S_i$ we have $\lambda_0(\ub) \leq \lambda_0(d_i\bb_i)$.
        \item For $\ub \in S_i$, if $\lambda_0(\ub) = \lambda_0(d_i\bb_i)$ then $\nu_\db(\ub) \leq 1$. 
        \item For $\ub\in S_i$, if $\lambda_0(\ub) = \lambda_0(d_i\bb_i)$ \textit{and} $u_1\geq 1$, then $\nu_\db(\ub) < 1$. 
        \item There exists an index $2\leq j\leq n-1$ and some $\ub\in S_j$ such that $u_1\geq 1$ and $\lambda_0(\ub) = d_j\lambda_0(\bb_j)$.
    \end{enumerate}
\end{lemma}
\begin{proof}
    To start, we set
    \begin{equation}\label{eq:def-t-0}
    t_0:= \min_{\ub\in S \text{ with } u_1 \geq 1} \frac{\nu_\db(\ub) - 1}{u_1}.
    \end{equation}
    We define $\lambda_0: \Z^n\to \Q$ by
    \[
    \lambda_0(\ub) = t_0u_1 - \nu_\db(\ub).
    \]
    We claim that $\lambda_0$ satisfies (B.i)-(B.iv).
    \begin{enumerate}[label=\textbf{(B.\roman*)}]
        \item Let $2\leq i\leq n-1, \ub\in S_i$. Our proof that $\lambda_0(\ub)\leq \lambda_0(d_i\bb_i)$ splits into two cases.
        \begin{itemize}
            \item[$u_1 = 0$:] By assumption (2) of \Cref{lemma:final-weight-partial-order} we have $\lambda_0(\ub) = -\nu_\db(\ub) \leq -1$.
            \item[$u_1 \geq 1$:] We have $\lambda_0(\ub) \leq \left(\frac{\nu_\db(\ub) - 1}{u_1}\right)u_1 - \nu_\db(\ub) = -1$. The property (B.i) then follows from the fact that $\lambda_0(d_i\bb_i) = \nu_\db(d_i\bb_i) = -1$ for all $2\leq i\leq n-1$.
        \end{itemize}
        \item First, we observe that if $\lambda_0(\ub) = \lambda_0(d_i\bb_i) = -1$, then
        \begin{equation}\label{eq:t-0-negative}
            \nu_\db(\ub) = t_0u_1 - \lambda_0(\ub) = 1 + t_0u_1.
        \end{equation}
        By assumption (4) of \Cref{lemma:final-weight-partial-order}, there exists $\ub\in S$ with $u_1\geq 1$ and $\nu_\db(\ub) < 1$. It follows that the minimum in \Cref{eq:def-t-0} is negative, proving (B.ii).
        \item The claim (B.iii) follows from \Cref{eq:t-0-negative} together with the fact that $t_0u_1 < 0$.
        \item If $\ub\in S_j$ realizes the minimum in \Cref{eq:def-t-0}, then 
        \[
        \lambda_0(\ub) = \left(\frac{\nu_\db(\ub) - 1}{u_1}\right)u_1 + \nu_\db(\ub) = -1 = \lambda_0(d_j\bb_j).
        \]
    \end{enumerate}
    
\end{proof}
\begin{notation}\label{notation:S-0}
Assume the notation of \Cref{lemma:final-weight-partial-order}, and let $\lambda_0$ be as in \Cref{lemma:zeroth-weight-partial-order}. For $2\leq i\leq n-1$, we define 
\[
S_i^{(0)} = \{\ub\in S_i: \lambda_0(\ub) = \lambda_0(d_i\bb_i)\}
\]
and we set $S^{(0)} = S_2^{(0)}\sqcup \dots \sqcup S_{n-1}^{(0)}$. 
\end{notation}
By \Cref{lemma:zeroth-weight-partial-order}, we have some control over $S^{(0)}$. 
\begin{itemize}
    \item By definition, we have $S^{(0)}\subseteq S$.
    \item By (B.ii) and (B.iii), for every $\ub\in S^{(0)}$ we have $\nu_\db(\ub)\leq 1$ with equality only if $u_1 = 0$.
    \item By (B.iv), $S^{(0)}$ is nonempty. In particular, $S^{(0)}$ contains an element $\ub$ with $u_1\geq 1$.
\end{itemize}
\begin{lemma}\label{lemma:first-weight-partial-order}
Assume the notation of \Cref{notation:S-0}. There exists a linear map $\lambda_1: \Z^n\to \Q$ such that:
\begin{enumerate}[label=\textbf{(C.\roman*)}]
    \item For all $2\leq i\leq n-1, \ub\in S_i^{(0)}$, if $u_1 = 0$ then $\lambda_1(\ub) < \lambda_1(d_i\bb_i)$.
    \item For all $2\leq i\leq n-1, \ub\in S_i^{(0)}$, if $u_1\geq 1$ then we have $\lambda_1(\ub)\leq \lambda_1(d_i\bb_i)$.
    \item There exists an index $2\leq j\leq n-1$ and some $\ub\in S_j^{(0)}$ such that $u_1\geq 1$ and $\lambda_1(\ub) = \lambda_1(d_j\bb_j)$.
\end{enumerate}
\end{lemma}
\begin{proof}
    We set
    \begin{equation}\label{eq:def-t-1}
        t_1:= \min_{2\leq i\leq n-1}\min_{\substack{\ub\in S_i^{(0)}\\ \text{with }u_1 \neq 0}}\frac{d_n^2u_2 + \dots + d_n^nu_n-d_id_n^i}{u_1}, \quad \lambda_1(\ub) := t_1u_1 - d_n^2u_2 - \dots - d_n^nu_n.
    \end{equation}
    We claim that $\lambda_1$ satisfies properties (C.i)-(C.iii).
    \begin{enumerate}[label=\textbf{(C.\roman*)}]
    \item By assumption (2) and (B.ii), we have $\nu_\db(\ub)\leq 1$. Consequently, we deduce
    \begin{equation}\label{eq:partial-nu-d}
        1\geq \nu_\db(\ub) = \frac{u_1}{d_1} + \dots + \frac{u_i}{d_i} \geq \frac{u_1+\dots+u_i}{d_i}.
    \end{equation}
    If $u_1+\dots + u_i = d_i$, then each inequality in \Cref{eq:partial-nu-d} is an equality and hence $\ub = d_i\bb_i$. This would violate assumption (3), so there must instead be some index $i < j \leq n$ such that $u_j\geq 1$. We then have 
    \[
        \lambda_1(\ub) = -d_n^2u_2 - \dots - d_n^nu_n \leq -d_n^j < -d_id_n^i = \lambda_1(d_i\bb_i).  
    \]
    \item We compute
    \begin{equation}\label{eq:lambda-1-bound}
    \lambda_1(\ub) \leq \left(\frac{d_n^2u_2 + \dots + d_n^nu_n - d_id_n^i}{u_i}\right)u_i - d_n^2u_2 - \dots - d_n^nu_n = -d_id_n^i = \lambda_1(d_i\bb_i).
    \end{equation}
    \item By (B.iv), the set $\{\ub \in S^{(0)}: u_1\geq 1\}$ is nonempty. If $\ub$ achieves the minimum in \Cref{eq:def-t-1}, then the inequality in \Cref{eq:lambda-1-bound} is an equality.
\end{enumerate}
\end{proof}

\begin{notation}[Subsuming \Cref{notation:S-0}]\label{notation:S-ell}
    Let $k\geq 0$. Suppose that $\lambda_0,\dots, \lambda_k$ have been defined. For $1\leq \ell\leq k, 2\leq i\leq n-1$, we define
    \[
    S^{(\ell)}_i = \{\ub\in S_i^{\ell-1}: \lambda_\ell(\ub) = d_i\lambda_\ell(\bb_i)\}
    \]
    and set $S^{(\ell)} = S^{(\ell)}_2\sqcup \dots \sqcup S^{(\ell)}_{n-1}$. Finally, we set $\Lambda_\ell = \{2\leq i\leq n-1: S_i^{(\ell)}\neq \emptyset\}$.
\end{notation}
Similarly to the case of $S^{(0)}$, \Cref{lemma:first-weight-partial-order,lemma:kth-weight-partial-order} give us control over $S^{(\ell)}$ for $\ell\geq 1$. Assume we have constructed $\lambda_{k+1}$ and proven \Cref{lemma:kth-weight-partial-order} for $k\leq \ell-1$.
\begin{itemize}
    \item By definition, we have $S^{(\ell)}\subseteq S^{(\ell-1)}$.
    \item By (B.iii), (C.i), and (C.ii), for all $\ub\in S^{(\ell)}$ we have $u_1\geq 1$ and $\nu_\db(\ub) < 1$.
    \item By (C.iii) if $\ell = 1$ and (D.iii) if $\ell \geq 2$, there exists an element $\ub\in S^{(\ell)}$ with $u_1\geq 1$.
\end{itemize}
\begin{lemma}\label{lemma:kth-weight-partial-order}
    Let $k\geq 0$ and suppose that $\lambda_0,\dots, \lambda_k$ have been defined. Assume the notation of \Cref{notation:S-ell} and suppose that $|\Lambda_k| \geq 2$. Set $i_k = \min \Lambda_k$. There exists $\lambda_{k+1}: \Z^n\to \Q$ such that:
    \begin{enumerate}[label=\textbf{(D.\roman*)}]
    \item For $\ub\in S_{i_k}^{(k)}$, we have $\lambda_{k+1}(\ub) <  \lambda_{k+1}(d_{i_k}\bb_{i_k})$.
    \item For $j > i_k, \ub\in S_j^{(k)}$, we have $\lambda_{k+1}(\ub)\leq \lambda_{k+1}(d_j\bb_j)$.
    \item There exists some index $j > i_k$ and some $\ub\in S_j^{(k)}$ such that $\lambda_{k+1}(\ub)= \lambda_{k+1}(d_j\bb_j)$.
    \end{enumerate}
\end{lemma}
\begin{proof}
     We define
    \begin{equation}\label{eq:def-t-k}
    t_{k+1} := \max_{j>i_k, \ub\in S^{(k)}_j}\frac{u_{i_k}}{u_1}, \quad \lambda_{k+1}(\ub):= -t_{k+1}u_1 + u_{i_k}.
    \end{equation}
    We claim that $\lambda_{k+1}$ satisfies (D.i)-(D.iii).
    \begin{enumerate}[label=\textbf{(D.\roman*)}]
    \item For $\ub\in S_{i_k}^{(k)}$, since $\ub\neq d_{i_k}\bb_{i_k}$, we have
    \[
    \lambda_{k+1}(\ub) \leq u_{i_k} < d_{i_k} = d_{i_k}\lambda_{k+1}(\bb_{i_k}).
    \] 
    \item For $j > k, \ub \in S_j^{(k)}$ we have
    \begin{equation}\label{eq:lambda-k-1-bound}
    \lambda_{k+1}(\ub) \leq -\left(\frac{u_{i_k}}{u_1}\right)u_1 + u_{i_k} = 0 = \lambda_{k+1}(d_j\bb_j).
    \end{equation}
    \item If $\ub$ attains the maximum in \Cref{eq:def-t-k}, then the inequality in \Cref{eq:lambda-k-1-bound} is sharp. 
    \end{enumerate}
\end{proof}

\begin{lemma}[c.f. \cite{eisenbud_commutative_1995}, Exercise 15.12]\label{lemma:weight-order-from-partial-order}
    Let $\mu_0,\dots, \mu_s: \Z^n\to \Q$ be linear maps and $U\subseteq \Z^n$ a finite set. There exists a map $\mu: \Z^n\to \Z$ such that for all $\ub,\mathbf{v} \in U$, we have $\mu(\ub) < \mu(\mathbf{v})$ if and only if
    \begin{equation}\label{eq:mu-leq-condition}
    \begin{split}
    	&\text{there exists $0\leq k\leq s$ such that }\mu_k(\ub) < \mu_k(\mathbf{v})  \\
    	&\text{and }\mu_\ell(\ub) = \mu_\ell(\mathbf{v}) \text{ for all } 0\leq \ell \leq k-1 
    \end{split}
    \end{equation}
    \end{lemma}
    
    We may now prove \Cref{lemma:final-weight-partial-order}.
\begin{proof}
    First, we construct the auxiliary maps $\lambda_0,\dots, \lambda_s$. Apply \Cref{lemma:zeroth-weight-partial-order,lemma:first-weight-partial-order} to $S$ to produce linear maps $\lambda_0,\lambda_1$. By (B.iv) and (C.iii), we have $|\Lambda_1|\geq 1$. We continue inductively: suppose we have constructed $\lambda_0,\dots, \lambda_k$. If $|\Lambda_k| = 1$, we set $s = k$ and proceed to the start of the following paragraph for the next step of the argument. Otherwise, apply \Cref{lemma:kth-weight-partial-order} to produce a map $\lambda_{k+1}$. By (D.i) we have $\Lambda_{k+1}\subseteq \Lambda_k\setminus \{i_k\}$ and by (D.iii) we have $|\Lambda_{k+1}|\geq 1$. As $n-2 \geq |\Lambda_1| > |\Lambda_2| > \dots \geq 1$, we eventually arrive at an index $s$ such that $|\Lambda_s| = 1$.

    Apply \Cref{lemma:weight-order-from-partial-order} to the sequence $\lambda_0,\dots, \lambda_s$ to produce a map $\lambda: \Z^n\to \Z$. We show that $\lambda$ satisfies (A.i) and (A.ii), where the index $m$ is the unique element of $\Lambda_s$. By \Cref{lemma:zeroth-weight-partial-order,lemma:first-weight-partial-order,lemma:kth-weight-partial-order}, for all $2\leq i\leq n-1,\ub\in S_i, 0\leq k\leq s$ we have $\lambda_k(\ub) \leq \lambda_k(d_i\bb_i)$. It follows that $\lambda(\ub) \leq \lambda(d_i\bb_i)$. If $\lambda(\ub) = \lambda(d_i\bb_i)$, then we must have $\lambda_k(\ub) = \lambda_k(d_i\bb_i)$ for all $0\leq k\leq s$. Each successive equality $\lambda_k(\ub) = \lambda_k(d_i\bb_i)$ creates further restrictions on $\ub$:
    \begin{enumerate}[(a)]
        \item As $\lambda_0(\ub) = \lambda_0(d_i\bb_i)$, by \Cref{lemma:zeroth-weight-partial-order} (B.ii) and (B.iii) we have either $u_1 = 0$ and $\nu_\db(\ub) \leq 1$ or $u_1\geq 1$ and $\nu_\db(\ub) < 1$. 
        \item As $\lambda_1(\ub) = \lambda_1(d_i\bb_i)$, by \Cref{lemma:first-weight-partial-order} (C.i), we must in fact have $u_1\geq 1$ and $\nu_\db(\ub) < 1$. 
        \item As $\lambda_k(\ub) = \lambda_k(d_i\bb_i)$ for all $1\leq i\leq s$, we have $i\in \Lambda_1\cap \dots \cap \Lambda_s = \{m\}$. 
    \end{enumerate}
    Point (c) in the above list implies (A.i) and points (a), (b) imply (A.iii). For (A.ii), we have by construction that $\Lambda_s = \{m\}$ is nonempty, so there exists $\ub^*\in S_m^{(s)}\subseteq S_m$. By \Cref{notation:S-ell}, the set $S_m^{(s)}$ is inductively defined as the set of $\ub\in S_m$ such that $\lambda_k(\ub) = \lambda_k(d_m \bb_m)$ for all $0\leq k\leq s$, so we have $\lambda(\ub^*) = \lambda(d_m\bb_m)$, hence \Cref{eq:lambda-max} holds.
\end{proof}
\subsubsection{Applying the Degeneration Order}
\begin{lemma}\label{lemma:rees-subalgebra}
    Let $R$ be as in \Cref{ass:mod-setup-2}. Suppose that $f_1,\dots, f_{a_i}$ are $d_i$-forms comprising a regular sequence in $R$, and suppose that $f_j\in k[\x_i]$ for all $1\leq j\leq a_i$. Then the integral closure $J$ of $(f_1,\dots, f_{a_i})$ in $R$ is equal to $(\x_i)^{d_i}$. 
\end{lemma}
\begin{proof}
    Since $k[\x_i]\to R$ is faithfully flat, $f_1,\dots, f_{a_i}$ form a regular sequence in $k[\x_i]$. By \Cref{thm:rees}, the integral closure of $(f_1,\dots, f_{a_i})$ in $k[\x_i]$ is $(\x_i)^{d_i}$. By \Cref{prop:int-closure} (vii), we have $(\x_i)^{d_i}\subseteq J$. On the other hand, we have $(f_1,\dots, f_{a_i})\subseteq (\x_i)^{d_i}$. By \cite[Proposition 1.4.6]{huneke_integral}, $(\x_i)^{d_i}$ is integrally closed in $R$, so $J = (\x_i)^{d_i}$. 
\end{proof}
\begin{lemma}\label{lemma:nonempty-supp-element}
    Assume the setting of \Cref{ass:mod-setup-2}. Suppose $I$ satisfies \Cref{eqn:aligned}. Then we have the following:
    \begin{enumerate}[(a)]
        \item For all $1\leq i\leq r,1\leq j\leq a_i$, there exists $f'_{i,j}\in k[\x_i]$ and $f''_{i,j}\in (\x_1,\dots, \x_{i-1})$ such that $f_{i,j} = f'_{i,j} + f''_{i,j}$.
        \item For all $1\leq i\leq r$, we have $\sqrt{I_1 + \dots + I_i} = (\x_1,\dots, \x_i)$.
        \item For all $1\leq i\leq r$, the elements $f'_{i,1},\dots, f'_{i,a_i}$ form a regular sequence.
    \end{enumerate}
\end{lemma}
\begin{proof}
        The key to this proof is the following simple observation. Let $g = x_{1,1}^{e_{1,1}}\dots x_{r,a_r}^{e_{r,a_r}}$ be a monomial of degree $d_i$ such that $w_d(g)\geq 1$. If $g\notin (\x_1,\dots, \x_{i-1})$, then we have
        \begin{align}\label{eq:nu-d-g-bound}
            1 \leq w_d(g)  &\leq \frac{e_{i,1} + \dots + e_{i,a_i}}{d_i} + \frac{e_{i+1,1},\dots + e_{r,a_r}}{d_{i+1}} \\\notag &= \frac{e_{i,1},\dots + e_{r,a_r}}{d_i} + (d_{i+1}-d_i)\left(\frac{e_{i+1,1},\dots + e_{r,a_r}}{d_id_{i+1}}\right)
        \end{align}
        As $e_{i,1} + \dots + e_{r,a_r} = d_i$ and $\deg(g) = d_i$ we have $e_{k,\ell} = 0$ for all $k < i$. At the same time, \Cref{eq:nu-d-g-bound} implies that $e_{i+1,1}+\dots+ e_{r,a_r} = 0$, so $g\in k[\x_i]$. 

        For any $1\leq i\leq r,1\leq j\leq a_i$, by \Cref{eqn:aligned} we have $\supp(f_{i,j})\subseteq \Df\sqcup (\x_1)$. By the previous paragraph, we have $\supp(f_{i,j})\cap \Df\subseteq k[\x_i]\sqcup (\x_1,\dots, \x_{i-1})$. Writing $f_{i,j} = \sum_{y\in \supp(f_{i,j})}\beta_y^{(i,j)}y$, we may therefore define
        \[
        f'_{i,j}:= \sum_{y\in \supp(f_{i,j})\cap k[\x_i]}\beta^{(i,j)}_yy, \qquad f''_{i,j} = \sum_{y\in \supp(f_{i,j})\cap (\x_1,\dots, \x_{i-1})}\beta^{(i,j)}_yy,
        \]
        which proves (a). 

        We prove (b) and (c) simultaneously by induction. By \Cref{eqn:aligned}, $I_1$ is a homogeneous ideal extended from $k[\x_1]$ and $\codim(I_1) = a_1$, so (b) holds for $i = 1$. As $f_{1,j} = f'_{1,j}$ for $1\leq j\leq a_1$, (c) also holds for $i = 1$. Suppose that (b), (c) hold for $i - 1$. Then $\sqrt{I_1 + \dots + I_{i-1}} = (\x_1,\dots, \x_{i-1})$, so we have
        \begin{align}\label{eqn:radical-comparison}
        \sqrt{I_1 + \dots + I_i} & = \sqrt{(\x_1,\dots, \x_{i-1}) + (f_{i,1},\dots, f_{i,a_i})} \\\notag &= \sqrt{(\x_1,\dots, \x_{i-1}) + (f'_{i,1},\dots, f'_{i,a_i})}.
        \end{align}
        Taking the heights of all terms in \Cref{eqn:radical-comparison}, we observe that the elements $f_{1,1},\dots, f_{i-1,a_{i-1}},f'_{i,1},\dots, f'_{i,a_i}$ form a regular sequence. In particular, $f'_{i,1},\dots, f'_{i,a_i}$ is a regular sequence, proving claim (c). As $\codim((f'_{i,1},\dots, f'_{i,a_i})) = a_i$ and $(f'_{i,1},\dots, f'_{i,a_i})$ is a homogeneous ideal extended from $k[\x_i]$, we have $\sqrt{(f'_{i,1},\dots, f'_{i,a_i})} = (\x_i)$. Applying \Cref{eqn:radical-comparison} again, we deduce claim (b). 
\end{proof}

\begin{lemma}\label{lem:application-of-weight-order}
    Assume the setting of \Cref{ass:mod-setup-2}. Suppose $I$ satisfies \Cref{eqn:aligned} and $I\not\subseteq \Df$. Then there exists an integer $2\leq m\leq r-1$ and $d_m$-forms $h_{m,1},\dots, h_{m,a_m}$ such that, setting 
    \[
    J := (\x_1)^{d_1} + \dots + (\x_{m-1})^{d_{m-1}} + (h_{m,1},\dots, h_{m,a_m}) + (\x_{m+1})^{d_{m+1}} + \dots + (\x_r)^{d_r},
    \]
    we have
    \begin{enumerate}[(i)]
        \item For all $1\leq j\leq a_m$, we have $h_{m,j} = h'_{m,j} + h''_{m,j}$, where $h'_{m,j}\in k[\x_m], h''_{m,j}\in (\x_1)$, and $w_d(y) < 1$ for all $y\in \supp(h''_{m,j})$;
        \item There exists some $j$ such that $h''_{m,j}\neq 0$;
        \item $J$ is $\mf$-primary with $E_n(J) = E_n(I)$;
        \item $c(J)\leq c(I)$.
    \end{enumerate}
\end{lemma}
\begin{proof}
Let $f'_{i,j}, f''_{i,j}$ be as in \Cref{lemma:nonempty-supp-element}. Define a semigroup homomorphism $\rho: \Mon(\x_1,\dots, \x_r)\to \Z_{\geq 0}^r$ by $\rho(x_{i,j}) = \bb_i$. For $2\leq i\leq r$, we define
\[
S_i:= \bigcup_{j=1}^{a_i} \rho\left(\supp(f''_{i,j})\right).
\]
Set $S = S_2\sqcup \dots \sqcup S_{r-1}$. We claim that $S$ satisfies the hypotheses (1)-(4) of \Cref{lemma:final-weight-partial-order}.
\begin{enumerate}[(1)]
    \item The claim follows from the fact that $f''_{i,j}$ is homogeneous of degree $d_i$ and $\deg(y) = |\rho(y)|$ for any monomial $y\in R$.
    \item Let $y\in \supp(f''_{i,j})$ be a monomial. If $y\in (\x_1)$, then $\rho(y)_1\geq 1$. Otherwise, as $I$ satisfies \Cref{eqn:aligned}, we have $1\leq w_d(y) = \nu_\db(\rho(y))$.
    \item By construction, if $y\in \supp(f''_{i,j})$ then $y\in (\x_1,\dots, \x_{i-1})$, so $\rho(y)_1 + \dots + \rho(y)_{i-1} \geq 1$, hence $\rho(y)\neq d_i\bb_i$.
    \item By assumption, we have $I\not\subseteq \Df$, thus we have $f_{i,j}\notin \Df$ for some $1\leq i\leq r,1\leq j\leq a_i$. By \Cref{eqn:aligned}, $I_1$ is extended from $k[\x_1]$, hence $I_1\subseteq (\x_1)^{d_1}\subseteq \Df$. Moreover, by \Cref{lemma:maxpower-subset} we have $\mf^{d_r}\subseteq \Df$, hence $I_r\subseteq \Df$. We conclude that there exists some $2\leq i\leq r-1$ and some $1\leq j\leq a_i$ such that $f_{i,j}\notin \Df$. As $f'_{i,j}\in \mf^{d_i}\cap k[\x_i]\subseteq \Df$ and $f'_{i,j} + f''_{i,j} = f_{i,j}\notin \Df$, it must be the case that $f''_{i,j}\notin \Df$. In particular, there exists a monomial $y$ in the support of $f''_{i,j}$ such that $y\notin \Df$. Recalling that $\Df$ consists precisely of the elements $f\in R$ with $w_d(f)\geq 1$, for this monomial $y$ we have $\rho(y)\in S$ and $\nu_\db(\rho(y)) = w_d(y) < 1$.
\end{enumerate}
We now apply \Cref{lemma:final-weight-partial-order} to the set $S$ to produce a map $\lambda: \Z^r\to \Z$ and an index $2\leq m\leq r-1$ such that 
\begin{enumerate}[label=\textbf{(A.\roman*)}]
    \item For all $2\leq i\leq r-1, i\neq m, \ub\in S_i$ we have $\lambda(\ub) < \lambda(d_m\bb_m)$.
    \item We have
    \begin{equation}
    \max_{\ub\in S_m} \lambda(\ub) = \lambda(d_m\bb_m).\tag{$21$}
    \end{equation}
    \item For any $\ub$ achieving the maximum in \Cref{eq:lambda-max} we have $\nu_\db(\ub) < 1$ and $u_1\geq 1$.
\end{enumerate}
\vspace{2ex}
Let $\ini_\lambda$ denote the monomial partial order given by $y <_\lambda z$ if $\lambda(\rho(y)) < \lambda(\rho(z))$. We consider the leading terms $\ini_{\lambda}(f_{i,j})$ for $2\leq i\leq r-1, 1\leq j\leq a_i$. For any $y\in \supp(f'_{i,j})$, we have $\rho(y) = d_i\bb_i$, and in particular $\lambda(\rho(y))\leq \lambda(d_i\bb_i)$. By (A.i) and (A.ii), we have $\lambda(\rho(y)) \leq \lambda(d_i\bb_i)$ for all $y\in \supp(f''_{i,j})$.  It follows that 
\[\max_{y\in \supp(f_{i,j})} \lambda(\rho(y)) \leq \lambda(d_i\bb_i).\] 
 By \Cref{lemma:nonempty-supp-element}, the elements $f'_{i,j}$ form a regular sequence. In particular, $f'_{i,j}\neq 0$, so $\supp(f_{i,j})\cap k[\x_i] = \supp(f'_{i,j})\neq \emptyset$. If $y_{i,j}\in \supp(f'_{i,j})$, then we have 
\[\max_{y\in \supp(f_{i,j})} \lambda(\rho(y)) \geq \lambda(\rho(y_{i,j})) = \lambda(d_i\bb_i).\] 
We now have an expression for the leading term $\ini_\lambda(f_{i,j})$. If we write $f_{i,j} = \sum_{y\in \supp(f_{i,j})} \beta_y^{(i,j)}y$, then $\ini_{\lambda}(f_{i,j})$ is the sum over all $\gamma_y^{(i,j)}y$ such that $\lambda(\rho(y))$ is as large as possible. To be precise, we have
\begin{equation}\label{eqn:leading-term-expression}
    \ini_{\lambda}(f_{i,j}) = \sum_{y\in \supp(f_{i,j}): \lambda(\rho(y)) = \lambda(d_i\bb_i)}\beta_y^{(i,j)}y = f'_{i,j} + \sum_{y\in \supp(f''_{i,j}): \lambda(\rho(y)) = \lambda(d_i\bb_i)}\beta_y^{(i,j)}y.
\end{equation}
For $1\leq i\leq a_m$, set $h_{m,j}:= \ini_\lambda(f_{m,j})$. With this choice of $h_{m,1},\dots, h_{m,a_m}$, we verify that $J$ satisfies conclusions (i)-(iv) of the lemma.
\begin{enumerate}[(i)]
    \item We set $h'_{m,j} := f'_{m,j}$ and $h''_{m,j} :=\sum_{y\in \supp(f''_{i,j}): \lambda(\rho(y)) = \lambda(d_i\bb_i)}\gamma_y^{(i,j)}y$. By \Cref{eqn:leading-term-expression}, we have $h'_{m,j} + h''_{m,j}=\ini_\lambda(f_{m,j}) = h_{m,j}$. By the definition of $f'_{m,j}$ in \Cref{lemma:nonempty-supp-element}, we have $\supp(h'_{m,j})\subseteq k[\x_m]$. For any $y\in \supp h''_{m,j}$, we have $y\in \supp(f''_{m,j})$, so $\rho(y)\in S_m$. We also have $\lambda(\rho(y)) = \lambda(d_m\bb_m)$, hence by (A.iii) we have $\rho(y)_1\geq 1$ and $\nu_\db(\rho(y)) < 1$, hence $y\in (\x_1)$ and $w_d(y) < 1$. 
    \item By (A.ii), there exists $\ub\in S_m$ such that $\lambda(\ub) = \lambda(d_m \bb_m)$. By the definition of $S_m$, there exists some index $1\leq j\leq a_m$ and some $y\in \supp(f''_{m,j})$ such that $\rho(y) = \ub$. By \Cref{eqn:leading-term-expression}, we have $y\in \supp(h''_{m,j})$.
    \item We argue similarly to \Cref{eqn:radical-comparison}. By \Cref{lemma:rees-subalgebra}, we have $\sqrt{(f'_{m,1},\dots, f'_{m,a_m})} = (\x_m)$. By \Cref{eqn:leading-term-expression}, we have $h_{m,j}\equiv f'_{m,j}\mod (\x_1)$, so it follows that
    \begin{align*}
    \sqrt{J} &= \sqrt{(\x_1) + \dots + (\x_{m-1}) + (h_{m,1},\dots, h_{m,a_m}) + (\x_{m+1}) + \dots + (\x_r)} \\&= \sqrt{(\x_1) + \dots + (\x_{m-1}) + (f'_{m,1},\dots, f'_{m,a_m}) + (\x_{m+1}) + \dots + (\x_r)} = \mf.
    \end{align*}
    To see that $E_n(J) = E_n(I)$, we note that $J$ has the same integral closure as the ideal
    \[
    J' = (x_{1,1}^{d_1},\dots, x_{m-1,a_{m-1}}^{d_{m-1}},h_{m,1},\dots, h_{m,a_m},x_{m+1,1}^{d_{m+1}},\dots, x_{r,a_r}^{d_r}).
    \]
    As $J'$ is a complete intersection, by \Cref{lemma:mm-hyperplane-section} we have \[E_n(J) = E_n(J') = \frac{a_1}{d_1} + \dots + \frac{a_r}{d_r} = E_n(I).\]
    \item Let $I' = (\ini_\lambda(f_{1,1}),\dots, \ini_\lambda(f_{r,a_r}))$. We claim that $J\subseteq \overline{I'}$. By \Cref{eqn:leading-term-expression} and (A.iii), we have $\ini_\lambda(f_{i,j}) = f'_{i,j}$ for $i\neq m$. For $i\neq m$, it follows from \Cref{lemma:nonempty-supp-element,lemma:rees-subalgebra} that $(\x_i)^{d_i} = \overline{(f'_{i,1},\dots, f'_{i,a_i})} \subseteq \overline{I'}$. The claim that $J\subseteq \overline{I'}$ follows once we note that $h_{m,j} = \ini_\lambda(f_{i,j})\in I'$. To see that $c(J) \leq c(I)$, by \Cref{prop:basic-properties-of-c,prop:semicontinuity} we have
    \[
    c(J) = c(J') \leq c(\overline{I'}) = c(I') \leq c(\ini_\lambda(I)) \leq c(I).
    \]
\end{enumerate}

\end{proof}
\begin{lemma}\label{lemma:simple-sufficient-strictness-condition}
    Assume the setting of \Cref{ass:mod-setup-2}. Let $J\subseteq R$ be an ideal satisfying conditions (i)-(iii) of \Cref{lem:application-of-weight-order}. Then $c(J) > E_n(J)$. 
\end{lemma}
  \begin{proof}
    Set $J_m = (h_{m,1},\dots, h_{m,a_m})$. We first prove this result in characteristic $p>0$. 
    
  We claim that the assumptions on $J_m$ imply that $\ef(J_m) > a_m$. To see this, suppose on the contrary that $\ef(J_m) \leq a_m$. As $\codim(J_m) = a_m$, it must be the case that $\ef(J_m) = a_m$. By \cite[Lemma 3.14]{baily_equigenerated_2025} we have $J_m\subseteq (\ell_1,\dots, \ell_{a_m})^d$ for some linearly independent 1-forms $\boldsymbol{\ell} = \ell_1,\dots, \ell_{a_m}\in R$. In particular, $\sqrt{J_m} = (\ell_1,\dots, \ell_{a_m})$. 
  
  We can compute $\sqrt{J_m}$ from a different perspective. By assumption, for all $y\in \supp(h''_{m,j})$ we have $y\in (\x_1)\setminus \Df$. In particular, $y\notin (\x_1)^{d_m}$, so we have instead $y\in (\x_1)\cap (\x_2,\dots, \x_r)$. It follows that 
\[
  (\boldsymbol{\ell}) = \sqrt{J_m} \subseteq \sqrt{(\mathbf{h}'_m) + (\mathbf{h}''_{m})}
  \subseteq (\x_m) + (\x_1)\cap (\x_2,\dots, \x_r)
  \]
  which implies that $(\boldsymbol{\ell}) = (\x_m)$. This contradicts the assumption that $h''_{m,j} \neq 0$ for some index $1\leq j\leq a_m$, so we conclude that $\ef(J_m) > a_m$. By \cite[Theorem 3.17]{baily_equigenerated_2025} we have $c(J_m) > \frac{a_m}{d_m}$.
    
Let $f= h_{1,m}^{t_1}\dots h_{m,a_m}^{t_{a_m}}$ be a homogeneous generator of $(J_m)^{\nu_{J_m}(p^e)}$ such that $f\notin \mf^{[p^e]}$. Write
    \[
    f = \sum_{y\in \supp(f)} \b_y y,\quad f':= \sum_{y\in \supp(f)\setminus \mf^{[p^e]}} \b_y y.
    \]
    As $f \equiv f'\mod \mf^{[p^e]}$, we have $f'J_m\subseteq \mf^{[p^e]}$. 

    Applying \Cref{thm:Briancon-Skoda} and \Cref{lemma:rees-subalgebra} to the ideal $\frac{J_m + (\x_1)}{(\x_1)}$, we have \[\frac{(\x_m)^{a_md_m} + (\x_1)}{(\x_1)}\subseteq \frac{J_m + (\x_1)}{(\x_1)}.\] Let $\mu_1,\dots, \mu_M$ be a minimal set of monomial generators for $(\x_m)^{a_md_m}$. Subsequently, we let $\widetilde{\mu_1},\dots, \widetilde{\mu_M}\in J_m$ be homogeneous elements such that $\mu_i - \widetilde{\mu_i}\in (\x_1)$ for all $1\leq i\leq M$. Let $\succ$ denote the reverse lexicographic order after putting the variables in the order $\x_1 \prec \dots \prec \x_r$. By definition of the reverse lexicographic order, we have $y \prec z$ for any $y\in (\x_1), z\in k[\x_2,\dots, \x_r]$. In particular, we have $\ini_\succ(\widetilde{\mu_i}) = \ini_\succ(\mu_i + (\widetilde{\mu_i}-\mu_i)) = \mu_i$. It follows that $(\x_m)^{a_md_m}\subseteq \ini_\succ(J_m)$.

    By construction, we have $f\notin \mf^{[p^e]}$ and $fJ_m\subseteq J_m^{\nu_{J_m}(p^e) + 1} \subseteq \mf^{[p^e]}$.
    As $f'\equiv f\mod \mf^{[p^e]}$, we have $f'\notin \mf^{[p^e]}$ and $f'J_m\subseteq \mf^{[p^e]} + J_m^{\nu_{J_m}(p^e) + 1} \subseteq \mf^{[p^e]}$. Let $y:= \ini_\succ(f')$. As no element of $\supp(f')$ is in $\mf^{[p^e]}$, we have $y\notin \mf^{[p^e]}$. Additionally, we have
    \[
    \ini_\succ(f')\ini_\succ(J_m)\subseteq \ini_\succ(f'J_m)\subseteq \ini_\succ(\mf^{[p^e]}) = \mf^{[p^e]}.
    \]
    By \Cref{lem:valuation-colon}, we have 
    \[
    (\mf^{[p^e]}:J_m) \subseteq (\mf^{[p^e]}:(\x_m)^{a_md_m}) = \mf^{[p^e]} + (\x_m)^{a_m(p^e-1) - a_md_m + 1}.
    \]
    As $y\notin \mf^{[p^e]}$, we have
    \begin{equation}\label{eq:ord_x_m(y)}
        \text{ord}_{(\x_m)}(y)\geq a_m(p^e-1) - a_md_m + 1.
    \end{equation}
    Additionally, we write
    \[
    f = \prod_{j=1}^{a_m} (h'_{m,j} + h''_{m,j})^{t_j} = \prod_{j=1}^{a_m} \sum_{\substack{0\leq t'_j, t''_j\leq t_j\\ t'_j + t''_j = t_j}} {t_j \choose t'_j}(h'_{m,j})^{t_j}(h''_{m,j})^{t''_j}.
    \]
    As $\supp(f')\subseteq \supp(f)$, there exist $(t'_1,t''_1),\dots, (t'_{a_m}, t''_{a_m})$ such that $t'_j + t''_j = t_j$ for all $1\leq j\leq a_m$ and 
    \[
    y\in \supp\left((h'_{m,1})^{t'_1}(h''_{m,1})^{t''_1}\dots (h'_{m,a_m})^{t'_{a_m}}(h''_{m,a_m})^{t''_{a_m}}\right).
    \]
    Set $J'_m = (h'_{m,1},\dots, h'_{m,a_m})$. As $J'_m$ is extended from $k[\x_m]$, by \Cref{lemma:rees-subalgebra} we have $J'_m\subseteq (\x_m)^{d_m}$, so 
    \[t'_1 + \dots + t'_{a_m}\leq \nu_{J'_m}(p^e)\leq \nu_{(\x_m)^{d_m}}(p^e)\leq \floor{\frac{a_m(p^e-1)}{d_m}}.\]
      By assumption, $J_m\not\subseteq \Df$. Define 
    \[\sigma  := \max_{\substack{1\leq j\leq a_m\\y\in \supp(h_{m,j})\setminus (\x_m)^{d_m}}}w_d(y),\] which satisfies $\sigma<1$ by conditions (ii) and (iv).  
    Consequently, we have 
    \begin{align}
    w_d(y) \leq  \sum_{j=1}^{a_m}t'_jw_d(h'_{m,j}) + \sum_{j=1}^{a_m}t''_j\left(\max_{z\in \supp(h''_{m,j})}w_d(z)\right)\notag\\
    = \sum_{j=1}^{a_m}t'_j + \sum_{j=1}^{a_m}t''_j\left(\max_{z\in \supp(h''_{m,j})}w_d(z)\right)\leq \sum_{j=1}^{a_m}t'_j + \sigma\sum_{j=1}^{a_m}t''_j \notag\\\label{eq:upper_bound_for_w(y)}
    \leq \floor{\frac{a_m(p^e-1)}{d_m}} + \sigma\left(\nu_{J_m}(p^e) - \floor{\frac{a_m(p^e-1)}{d_m}}\right).
  \end{align}
    As in \Cref{lem:valuation-colon}, let $\af_\b, \af_\b^+$ denote the ideals $\{f\in R: w_d(f) \geq \b\}, \{f\in R: w_d(f) > \b\}$ respectively. Let $t_e$ denote the quantity in \Cref{eq:upper_bound_for_w(y)}; as $w_d(y) \leq t_e$ we have $y\notin \af^+_{t_e}$. By assumption, $y\notin \mf^{[p^e]}$, and as $y$ is a monomial we have $y\notin \mf^{[p^e]} + \af_{t_e}^+$. Set $u_e:= (p^e-1)(\frac{a_1}{d_1}+\cdots+\frac{a_r}{d_r})$. It follows from \Cref{lem:valuation-colon} that
    \[
    y\notin \mf^{[p^e]} + \af^+_{t_e} = (\mf^{[p^e]}: \af_{u_e - t_e}).
    \]
    Let $z\in \af_{u_e-t_e}$ such that $yz\notin \mf^{[p^e]}$. Write $z = z'z''$ where $z''\in k[\x_m]$ and $z'\in k[\x_1,\dots, \widehat{\x_m},\dots, \x_r]$. As $yz\notin \mf^{[p^e]}$, by \Cref{eq:ord_x_m(y)} we have 
    \[\text{ord}_{(\x_m)}(z'') \leq (p^e-1)a_m - \text{ord}_{(\x_m)}(y) \leq a_md_m,\] hence $
    w_d(z') = w_d(z) - w_d(z'')\geq u_e - t_e - a_m$. As $\Df = \overline{(\x_1)^{d_1} + \dots + (\x_r)^{d_r}}$, by \Cref{thm:Briancon-Skoda} we have
    \[
    z' \in \Df^{\floor{u_e - t_e}-a_m} \subseteq ((\x_1)^{d_1} + \dots + (\x_r)^{d_r})^{\floor{u_e - t_e}-a_m-n}.
    \]
    Since $z'\notin (\x_m)$, we in fact have 
    \[z'\in ((\x_1)^{d_1} + \dots + (\x_{m-1})^{d_{m-1}} + (\x_{m+1})^{d_{m+1}} (\x_r)^{d_r})^{\floor{u_e - t_e}-a_m-n}\subseteq (J)^{\floor{u_e - t_e}-a_m-n}.\]
    As $z'f\notin \mf^{[p^e]}$, it follows that $\nu_{J}(p^e) \geq \nu_{J_m}(p^e) + \floor{u_e-t_e}-a_m-n$. Dividing by $p^e$ and letting $e\to\infty$, we obtain
    \begin{align*}
        c(J) & \geq c(J_m) + \lim_{e\to\infty}\frac{u_e}{p^e} - \lim_{e\to\infty}\frac{t_e}{p^e}\\
        & = c(J_m) + \left(\frac{a_1}{d_1}+\dots+\frac{a_r}{d_r}\right) - \left(\frac{a_m}{d_m}(1-\sigma) + \sigma c(J_m)\right)\\
        & = (1-\sigma)\left(c(J_m) - \frac{a_m}{d_m}\right) + \left(\frac{a_1}{d_1}+\dots+\frac{a_r}{d_r}\right)\\
    \end{align*}
    Since $\sigma < 1$ and $c(J_m) > \frac{a_m}{d_m}$, it follows that the above quantity exceeds $E_n(J)$.

    In characteristic zero, one notes that for any ideal $J$ satisfying conditions (i)-(iii), the reduction of the pair $(R, J)$ to characteristic $p\gg 0$ satisfies conditions (i)-(iii). Moreover, the quantity $\sigma$ is constant for $p\gg 0$. Assuming the reduction notation of \Cref{ass:mod-setup}, we have
    \begin{align*}
    c(J) = \lim_{\substack{
    \mu\in \Spec A\\\text{char}A/\mu\to\infty}} c(J_\mu) &\geq (1-\sigma)\left(\lim_{\substack{
        \mu\in \Spec A\\\text{char}A/\mu\to\infty}}c(J_{m,\mu}) - \frac{a_m}{d_m}\right) +E_n(J) \\&= (1-\sigma)\left(c(J_m)-\frac{a_m}{d_m}\right) + E_n(J) > E_n(J).
    \end{align*}
  \end{proof}
  \Cref{lemma:change-of-vars,lem:application-of-weight-order,lemma:simple-sufficient-strictness-condition} combine to give us a proof of \Cref{thm:min-char} in the case of an $\mf$-primary complete intersection.
\begin{prop}\label{prop:CI-min-char}
    Assume the setup of \Cref{ass:mod-setup-2} and suppose $c(I) = E_n(I)$. Then there exists $\gamma\in \GL_n(k)$  such that $\overline{\gamma I} = \Df$.
\end{prop}
\begin{proof}
    Using \Cref{lemma:change-of-vars}, we produce $\gamma\in \GL_n(k)$ such that $\gamma I$ satisfies \Cref{eqn:aligned}. By \Cref{lem:application-of-weight-order,lemma:simple-sufficient-strictness-condition}, we have $\overline{\gamma I} = \Df$.
\end{proof}
\section{Proof of Theorem B}
In this section, we prove the main result of this article, which we restate below.
\begin{theorem}\label{thm:min-char}
    Let $k$ be a perfect field. Let $R = k[x_1,\dots, x_n]$ and let $I\subseteq R$ be a homogeneous ideal with $\codim(I)\geq l$. If $E_l(I) = c(I)$, then there exist integers $d_1,\dots, d_l$ and $\gamma\in \GL_n(k)$ such that 
    \[
    \gamma\overline{I} = \overline{\left(x_1^{d_1},\dots, x_l^{d_l}\right)}.
    \]
\end{theorem}
\subsection{Algebraically closed fields} Throughout this subsection, fix the following assumption.
\begin{assumption}
\label{ass:mod-setup-3}
     Let $L$ be an uncountable, algebraically closed field. Let $S = L[x_1,\dots, x_n]$ and let $I\subseteq S$ be a homogeneous.  As in \Cref{lem:revlex-proj}, for $1\leq j\leq n$, let $S_j := L[x_1,\dots, x_j], \pi_j:S\to S_j$ the natural projection map and $\i_j:S_j\to S$ the natural embedding. Let $>$ denote the graded reverse lexicographic order $>_{\text{grlex}}$.
\end{assumption}
\begin{defn}
    Let $k$ be a field, $R = k[x_1,\dots, x_n]$, $I\subseteq R$ a homogeneous ideal, and $t\in \Z^+$. We let $[I]_t$ denote the vector space of $t$-forms in $I$ and we let $[I]_{\leq t}$ denote the direct sum of the $[I]_s$ for $s\leq t$.
\end{defn}
\begin{defn}
	Let $k$ be a field, $R = k[x_1,\dots, x_n]$, and $I\subseteq R$ a homogeneous ideal of height $l>0$. For $1\leq i\leq l,$ let $d_i(I) = \min\{j: \codim [I]_{\leq j}R \geq i\}$.
\end{defn}

\begin{lemma}[c.f. \cite{mayes_limiting_2014}]\label{lem:p-i-reduction}
    Assume the setting of \Cref{ass:mod-setup-3}. Let $\gamma\in \GL_n(L)$ be very general: for now, we impose the condition that for all $m > 0$, we have $\ini_>(\gamma I^m) = \gin_>(I^m)$; we will impose countably many additional conditions in \Cref{lemma:p_j-computes-d_j}. For $1\leq j\leq n, m>0$, set $\af_{j,m}:= \ini_>(\pi_j(\gamma I^m))\subseteq S_j$. For $j > 0, 1\leq i\leq j$, let $\bb_i^{(j)}$ denote the $i$th unit vector of $\R^j$. For $j\geq i$, let $p_j(i):= \inf\{t: t\bb_i^{(j)}\in \Gamma(\af_{j,\bullet})\}$. Then for all $j$, we have $p_j(j) = p_n(j)$. 
\end{lemma}
\begin{proof}
    By \Cref{lem:revlex-proj} we have $\i_j(\af_{j,\bullet})\subseteq \af_{n,\bullet}$ for all $1\leq j\leq n$, so  $p_n(j)\leq p_j(j)$. For the reverse direction, set $t = p_n(j)$. Since $t\bb_{j}^{(n)}\in \overline{\bigcup_{m>0}\frac{1}{2^m}\Gamma(\af_{n,2^m})}$, there exists a sequence $\{\ab_m = (a_{m,1},\dots,a_{m,n})\}_{m>0}$ such that $\ab_m\in \Gamma(\af_{n,2^m})$ for all $m$ and $\lim_{m\to\infty} 2^{-m}\ab_m = t\bb_{j}^{(n)}$. For any choice of $\{\ab_m\}_{m>0}$, we also have $(\ceil{a_{m,1}},\dots,\ceil{a_{m,n}})\in \Gamma(\af_{n,2^m})$ and  $\lim_{m\to\infty} \frac{(\ceil{a_{m,1}},\dots, \ceil{a_{m,n}})}{2^m} = t\bb_{j}^{(n)}$. We may therefore assume without loss of generality that $\ab_m\in \Z_{\geq 0}^n$ for all $m>0, 1\leq i\leq n$, hence for all $m > 0$, we have $\x^{\ab_m}\in \overline{\af_{n,2^m}}$.

    By \cite[Theorem 2.1]{guo_monomial_2015}, $\overline{\af_{n,2^m}}$ is Borel-fixed, so \[x_1^{a_{m,1}}\cdots x_{j-1}^{a_{m,j-1}}x_j^{a_{m,j}+\dots+a_{m,n}}\in \overline{\af_{n,2^m}}.\] We further note that $\af_{j,m} = \pi_j(\af_{n,m})$ by \Cref{lem:revlex-proj}. By \Cref{prop:int-closure}(iii), we conclude 
    \[
        \pi_j(x_1^{a_{m,1}}\cdots x_{j-1}^{a_{m,j-1}}x_j^{a_{m,j}+\dots+a_{m,n}})\in\pi_j(\overline{\af_{n,2^m}}) \subseteq \overline{\pi_j(\af_{n,2^m})} = \overline{\af_{j,2^m}}.\]
    It follows that 
    \[t\bb_j^{(j)} = \lim_{m\to\infty} \frac{(a_{m,1},\dots, a_{j-1},a_j + \dots + a_n)}{m}\in \Gamma(\af_{j,\bullet}),\]
    which proves $p_j(j) \leq t = p_n(j)$.
\end{proof}

\begin{remark}\label{rmk:real-sequence}
Assume the setup of \Cref{lem:p-i-reduction}. A priori, there exists a sequence \[\{\ab_m = (a_{m,1},\dots, a_{m,j})\}_{m>0}\subseteq \R^j_{\geq 0}\] such that for all $m > 0$ we have $\lim_{m\to\infty}2^{-m}\ab_m = t\bb_j^{(j)}$ and $\ab_m\in \Gamma(\af_{2^m})$. \Cref{lem:removal-of-integral-closure} is a minor refinement, allowing us to assume that $\x^{\ab_m}\in \af_{2^m}$.
\end{remark}
\begin{lemma}\label{lem:removal-of-integral-closure}
    Assume the setup of \Cref{lem:p-i-reduction}. There exists a sequence $\{\ab'_m\}_{m>0}\subseteq \Z^j$ such that  $\x^{\ab'_m}\in \af_{j,2^m}$ for all $m>0$ and $\lim_{m\to\infty}2^{-m}\ab'_m = p_j(j)\bb_j^{(j)}$. 
\end{lemma}
\begin{proof}
     Let $\{\ab_m\}_m$ be as in \Cref{rmk:real-sequence} and set 
     $\ab'_m = (2^{m - \floor{m/2}} + j - 1)\ab_{\floor{m/2}}$. By \Cref{thm:Briancon-Skoda}, we have 
     \[
     \x^{\ab'_m} = (\x^{\ab_{\floor{m/2}}})^{2^{m-\floor{m/2}} + j - 1} \in \overline{\af_{2^{\floor{m/2}}}^{2^{m-\floor{m/2}} + j - 1}}\subseteq \af_{2^{\floor{m/2}}}^{2^{m-\floor{m/2}}} \subseteq \af_{2^m}.
     \]
     The result follows once we note that
     \[
     \lim_{m\to\infty} 2^{-m}\ab'_m = \lim_{m\to\infty} \frac{2^{m - \floor{m/2}}+j-1}{2^{m-\floor{m/2}}}2^{-\floor{m/2}}\ab_{\floor{m/2}} = p_j(j)\bb_j^{(j)}.
     \]
\end{proof}

\begin{lemma}\label{lem:gin-coheight-1}
    Let $k$ be an algebraically closed field and $R = k[x_1,\dots, x_j]$. Let $\qf$ be a homogeneous prime ideal of height $j-1$ with $x_j\notin \qf$. If $>$ denotes the reverse lexicographic order, then for all $m>0$ we have $\ini_>(\qf^m) = (x_1,\dots, x_{j-1})^m$. 
\end{lemma}
\begin{proof}
    Since $k$ is algebraically closed, there exist linear forms $\ell_1,\dots, \ell_{j-1}\in R_1$ such that $q = (\ell_1,\dots, \ell_{j-1})$. By \cite[Theorem 15.17]{eisenbud_commutative_1995}, $\qf$ and $\ini_>(\qf)$ have the same Hilbert series. Moreover, $\ini_>(\qf)$ is a monomial ideal not containing $x_j$, so we must have $\ini_>(\qf) = (x_1,\dots, x_{j-1})$. For $m > 1$, we have the standard containment $(x_1,\dots, x_{j-1})^m = \ini_>(\qf)^m\subseteq \ini_>(\qf^m)$. As $(x_1,\dots, x_{j-1})^m$ has the same Hilbert series as $\qf^m$, the result follows.
\end{proof}
\begin{lemma}\label{lemma:uth-graded-piece}
    Let $k$ be a field, $R = k[x_1,\dots, x_j],$ and $\nf = (x_1,\dots, x_j)$. If $I\subseteq R$ is a homogeneous ideal, then
    \[
    [I\nf^t]_u = [[I]_{u-t}R]_u.
    \]
\end{lemma}
\begin{proof}
    First, we verify the following: if $f$ is homogeneous of degree $u-t$, then $[f\nf^t]_u = [fR]_u$. The containment $\subseteq$ is clear. For $\supseteq$, if $g$ is homogeneous and $gf\in [fR]_u$, then $g\in\nf^t$, so $gf\in [f\nf^t]_u$.

    Now let $f_1,\dots, f_r$ be homogeneous elements spanning $[I]_{u-t}$, and let $g_1,\dots, g_s$ be homogeneous elements such that $f_1,\dots, f_r,g_1,\dots, g_s$ generate $I$ and $\deg(g_i) > u-t$. By the first paragraph, we compute 
     \[[I\nf^t]_u = [(f_1,\dots, f_r)\nf^t]_u + [(g_1,\dots, g_s)\nf^t]_u = \sum_{i=1}^r [f_i\nf^t]_u = \sum_{i=1}^r[f_i]_u = [[I_{u-t}]R]_u.\] 
\end{proof}
\begin{lemma}\label{lemma:p_j-computes-d_j}
    Assume the setup of \Cref{lem:p-i-reduction,lem:removal-of-integral-closure}. Then for all $1\leq j\leq n$, we have $p_n(j) = d_j(I)$. 
\end{lemma}
\begin{proof}
     Set $d_j:= d_j(I)$. Before we begin the proof, we first state the additional generality conditions on $\gamma$. For all $m>0$, assume that $\ini_>(\pi_j(\gamma I^m)) = \gin_>(\pi_j(\gamma I^m))$. Since $\codim \pi_j(\gamma[I]_{\leq d_j-1}R) \leq \codim \gamma\left[I]_{\leq d_j-1}R\right) < j$, there is some $1\leq i\leq j$ such that $x_i\notin \sqrt{\pi_j(\gamma[I]_{\leq d_j-1})R}$; we choose $\gamma$ such that $x_j\notin \sqrt{\pi_j(\gamma[I]_{\leq d_j-1}R)}$. Each of these conditions is satisfied by a general choice of $\gamma$, so they may be realized simultaneously by a very general choice of $\gamma$.
    
    Set $J = \pi_j(\gamma I)$. Write $J_1 = [J]_{\leq d_j-1}R$, and let $J_2$ be the ideal generated by the minimal homogeneous generators of $J$ with degree at least $d_j$ so that $J = J_1 + J_2$ and $J_2\subseteq \mf^{d_j}$. By construction of $\gamma$, in the language of \Cref{lem:p-i-reduction} we have $\af_{j,m} = \ini_>(J^m)$. Also by construction, we have $x_j\notin \sqrt{[J]_{\leq d_j-1}}$, so we may choose a minimal prime $\pf$ over $J_1$ such that $x_j\notin \pf$. As $\codim \pf\leq j-1$, we may choose a homogeneous prime ideal $\qf\supseteq \pf$ such that $\codim \qf = j-1$ and $x_j\notin \qf$. By \Cref{lem:gin-coheight-1}, we have $\ini_>(\qf^m) = (x_1,\dots, x_{j-1})^m$ for all $m > 0$. 

    By \Cref{lem:removal-of-integral-closure}, choose a sequence $\{\ab_m\}_{m>0}$ such that $\x^{\ab_m}\in \af_{2^m}$ for all $m > 0$ and $\lim_{m\to\infty}2^{-m}\ab_m = p_j(j)\bb_j^{(j)}$. Let $e_m:= a_{m,1} + \dots + a_{m,j}$. For all $m > 0$, we have
    \begin{equation}\label{eqn:J_1-expansion}
        [J^{2^m}]_{e_m} = \left[\sum_{\substack{a+b = 2^m\\a+bd_j\leq e_m}} J_1+J_2 \right]_{e_m}
    \end{equation}
    We have shown that $J_1\subseteq \qf$ and $J_2\subseteq \mf^{d_j}$. By \Cref{eqn:J_1-expansion} and \Cref{lemma:uth-graded-piece}, we have
    \begin{align*}
       \x^{\ab_m} \in [J^{2^m}]_{e_m}  \subseteq \left[\sum_{\substack{a+b = 2^m\\a+bd_j\leq e_m}} \qf^a\mf^{bd_j}\right]_{e_m} &= \sum_{\substack{a + b = 2^m\\a+b d_j\leq e_m}}\left[\qf^a\mf^{bd_j}\right]_{e_m} \\&= \sum_{\substack{a+b=2^m\\a+bd_j\leq e_m}} = \left[\qf^{2^m - \floor{\frac{e_m}{d_j}}}\right]_{e_m}.
    \end{align*}
    Taking initial terms of both sides, we have $\x^{\ab_m}\in (x_1,\dots, x_{j-1})^{2^m - \floor{\frac{e_m}{d_{j}}}}$. Consequently, we have $a_{m,1} + \dots + a_{m,j-1} \geq 2^m - \floor{\frac{e_m}{d_{j}}}$. As $\lim_{m\to\infty} 2^{-m}(a_{m,1}+\dots+a_{m,j-1}) = 0$, this yields
    \[
    0  \geq \liminf_{m\to\infty} 2^{-m}\left(2^m - \floor{\frac{e_m}{d_{j}}}\right) = 1 - \frac{1}{d_j}\limsup_{m\to\infty} \frac{a_{m,j}}{2^m} = 1 - \frac{p_j(j)}{d_j}.
    \]
    From the above equation, we have $p_j(j) \geq d_j$. For the reverse containment, we have by \Cref{lemma:maxpower-subset} that $\mf^{d_j}\subseteq\overline{J}$. It follows from \Cref{thm:Briancon-Skoda} that $x_j^{(m+j-1)d_j}\in J^m$ for all $m > 0$, hence $p_j(j)\leq d_j$.
\end{proof}

\begin{lemma}\label{lem:integral-krull-intersection}
    Let $R = k[x_1,\dots, x_n]$ and let $\mf = (x_1,\dots, x_n)$. Let $I\subseteq R$ be a homogeneous ideal and $J\subseteq \mf$ any ideal. Then we have
    \[
    \bigcap_{t>0} \overline{I + J^t} = \overline{I}.
    \]
\end{lemma}
\begin{proof}
    By \cite[Corollary 6.8.5]{huneke_integral}, we have
    \[
        \overline{IR_\mf}\subseteq \bigcap_{t > 0}\overline{IR_\mf + J^tR_\mf} \subseteq \bigcap_{t>0} \overline{IR_\mf + \mf^tR_\mf} = \overline{IR_\mf}.
    \]
    By \Cref{prop:int-closure} (v), we have $\overline{I}R_\mf = \overline{IR_\mf}$. The ideal $\overline{I}$ is homogeneous by \cite[Corollary 5.2.3]{huneke_integral}, so $\overline{I}R_\mf\cap R = \overline{I}$. We have the following, from which the claim follows.
    \[
        \overline{I}\subseteq \bigcap_{t > 0}\left(\overline{IR_\mf + J^tR_\mf}\cap R\right) \subseteq \overline{IR_\mf}\cap R = \overline{I}.
    \]
\end{proof}

We are now able to prove \Cref{thm:min-char} in the case that $k$ is algebraically closed. 
\begin{proof}
    If $\codim(I) \geq l+1$, then $c(I) \geq E_{l+1}(I) > E_l(I),$ so we must have $\codim(I) = l$. Let $\ell_{l+1},\dots, \ell_n$ be general linear forms and set $L = V(\ell_{l+1},\dots, \ell_n)$. By \Cref{prop:sigma-poly-ring}, we have $E_l(I) = E_l(I|_L) \leq c(I|_L)\leq c(I)$, so we again have $E_l(I) = c(I|_L)$. We'll compute $d_1(I|_L),\dots, d_l(I|_L)$. Field extensions are faithfully flat, so $\codim([I]_{\leq t}|_L)$ is preserved under field extension, hence so are the quantities $d_1(I|_L),\dots, d_l(I|_L)$. We may without loss of generality assume $k$ is uncountable and algebraically closed. 
    
    Let $\gamma\in \GL_n(k)$ be as in \Cref{lem:p-i-reduction,lemma:p_j-computes-d_j} with respect to the ideal $I|_L$, and set $\af_m = \ini_>(\gamma (I|_L)^m) = \gin_>((I|_L)^m)$. As $c(I|_L) = E_l(I),$ by \Cref{cor:asymptotic-monomial} we have
    \begin{equation}\label{eqn:gamma-is-a-simplex}
    \overline{\R_{\geq 0}^l\setminus \Gamma(\af_\bullet)} =  \conv\left(\mathbf{0}, e_1(I|_L)\bb_1^{(l)}, \frac{e_2(I|_L)}{e_1(I|_L)}\bb_2^{(l)},\dots, \frac{e_{l}(I|_L)}{e_{l-1}(I|_L)}\bb_l^{(l)}\right).
    \end{equation}
    The quantities $p_n(1),\dots, p_l(l)$ can be read off from \Cref{eqn:gamma-is-a-simplex}, so we have $\frac{e_j(I|_L)}{e_{j-1}(I|_L)} = d_j(I|_L)$ for all $1\leq i\leq l$. 

	For $1\leq j\leq l$, let $f_i$ be an element of $[I]_{d_j}$ whose image mod $(\ell_{l+1},\dots, \ell_n)$ is general. As $\codim [I|_L]_{d_j} \geq j$, it follows that $f_1,\dots, f_l, \ell_{l+1},\dots, \ell_n$ is a regular sequence. Let $J = (f_1,\dots, f_l)$. By \Cref{lemma:mm-hyperplane-section} we have $E_l(J) = E_l(I)$. For $t > 0$, set $J_t = J + (\ell_{l+1}^t,\dots, \ell_n^t)$ and set $I_t = I + (\ell_{l+1}^t,\dots, \ell_n^t)$. Since $J_t\subseteq I_t$ is a complete intersection, \Cref{lemma:mm-hyperplane-section} and \Cref{prop:dp-rees,prop:basic-properties-of-c} give
	\begin{equation}\label{eqn:complete-subintersection}
		E_l(J) + \frac{n-l}{t} = E_l(J_t) \leq E_l(I_t).
	\end{equation}
	Combining \Cref{eqn:complete-subintersection} with \Cref{thm:bound} and \Cref{prop:basic-properties-of-c} (v) gives
\[
    E_l(J) + \frac{n-l}{t} \leq E_l(I_t) \leq c(I_t) \leq c(I) + c((\ell_{l+1}^t,\dots, \ell_n^t)) = E_l(I) + \frac{n-l}{t}.
    \]
    In particular, we have $E_l(I_t) = E_l(J_t) = c(J_t)$. Set 
    \[\Df_t := \overline{(x_1^{d_1(I)},\dots,x_l^{d_l(I)},x_{l+1}^t,\dots, x_l^t)}.\] By \Cref{prop:CI-min-char}, there exists $\gamma\in \GL_n(k)$ such that $\gamma_t \overline{J_t} = \Df_t$. Consequently, as $\Df_t\subseteq \gamma_t\overline{I_t}$ and $E_l(I_t) = E_l(\Df_t)$, by \Cref{prop:dp-rees} we have $\overline{\Df_t} = \overline{I_t}$.
    
    Set $\Df = \overline{(x_1^{d_1(I)},\dots, x_l^{d_l(I)})}$.  Let $d$ be the maximum degree of an irredundant generator of $\Df$ or $\overline{I}$. The $k$-vector space of forms of degree $\leq d$ in $R$ is finite-dimensional, so the infinite descending chain $[\Df_1]_{\leq d}\supseteq [\Df_2]_{\leq d} \supseteq \dots$ --- the limit of which is $[\Df]_{\leq d}$ by \Cref{lem:integral-krull-intersection} --- must stabilize; there exists $t_0\gg 0$ such that $[\Df_t]_{\leq d} = [\Df]_{\leq d}$ for all $t\geq t_0$. Similarly, there exists $t_1\gg 0$ such that $[\overline{I_t}]_{\leq d} = [\overline{I}]_{\leq d}$ for all $t\geq t_1$. For $t = \max(t_1,t_2)$, we have
    \[
    \gamma_t[\overline{I}]_{\leq d} = \gamma_t[\overline{I_t}]_{\leq d} = [\Df_t]_{\leq d} = [\Df]_{\leq d}.
    \]
    As $\gamma_t\overline{I}, \Df$ are generated in degree $\leq d$, it follows that $\gamma_t\overline{I} = \Df$. 
\end{proof}

\subsection{Perfect fields}
The following lemmas, together with the algebraically closed case of \Cref{thm:min-char}, allow us to generalize to the case of a perfect field.
\begin{lemma}\label{lemma:fixed-flag-fixed-ideal}
	Let $L$ be a field. For $1\leq i\leq r$, let $\x_i = x_{i,1},\dots, x_{i,a_i}$ be a tuple of variables over $L$. Set $S = L[\x_1,\dots, \x_r, y_1,\dots, y_s]$. Let $e_1 < \dots < e_r$ be positive integers and set 
	\[
	\Df = \overline{(x_{1,1}^{e_1},\dots, x_{r,a_r}^{e_r})}.
	\]
	If $\gamma\in \GL_{a_1+\dots + a_r + s}$ fixes the flag
	\[
	\bigg(0 \subseteq \text{span}_L(\x_1)\subseteq \dots \text{span}_L(\x_1,\dots, \x_r)\bigg),\
	\]
	then $\gamma \Df = \Df$.
\end{lemma}
\begin{proof}
	Let $w_\mathbf{e}$ denote the monomial valuation with $w_\mathbf{e}(\x_{i,j}) = \frac{1}{e_i}$ and $w_\mathbf{e}(y_i) = 0$. Then $\Df$ consists precisely of the elements $f\in S$ such that $w_{\mathbf{e}}(f)\geq 1$. By the assumption on $\gamma$, we have $\gamma(x_{i,j}^{e_i}) \in (\x_1,\dots, \x_i)^{e_i}$, hence $w_{\mathbf{e}}(\gamma x_{i,j}^{e_i}) \geq 1$. It follows that
	\[
	\gamma \Df = \overline{(\gamma x_{1,1}^{e_1},\dots, \gamma x_{r,a_r}^{e_r})} \subseteq \overline{\Df} = \Df.
	\] 
	By symmetry, we conclude that $\gamma \Df = \Df$.
\end{proof}
\begin{lemma}\label{lemma:sep-extension-coordinates}
	Let $k$ be a field and $L/k$ a separable field extension. Set $R= k[x_1,\dots, x_n]$ and $S = L[x_1,\dots, x_n]$. Let $d_1,\dots, d_l\in \Z^+$ and set $\Df = \overline{(x_1^{d_1},\dots, x_l^{d_l})}\subseteq R$. Suppose that $I\subseteq R$ is a homogeneous ideal and that there exists $\gamma_L\in GL_n(L)$ such that
	\begin{equation}\label{eqn:gamma_L}
	\gamma_L \overline{IS} = \Df S.
	\end{equation}
	Then there exists $\gamma_k\in \GL_n(k)$ such that
	\begin{equation}\label{eqn:gamma_k}
	\gamma_k\overline{I} = \Df.
	\end{equation}
\end{lemma}
\begin{proof}
	By \cite[Propositions 19.1.1,  19.1.2 and Corollary 19.5.2]{huneke_integral}, we have $\overline{I}S = \overline{IS}$, so we may replace $I$ by $\bar{I}$ and ignore the integral closure in \Cref{eqn:gamma_L,eqn:gamma_k}.
	
	Suppose that $d_1,\dots, d_l$ take the values $e_1,\dots, e_r$ with multiplicities $a_1,\dots, a_r$ respectively and that $e_1 < \dots < e_r$.  Re-index the variables $x_1,\dots, x_l, x_{l+1},\dots, x_n$ as $\x_1,\dots, \x_r, y_{l+1},\dots, y_n$ where $\x_i = x_{i,1},\dots, x_{i,a_i}$. For $1\leq i\leq r$, let $I_i$ denote the ideal generated by $[I]_{\leq e_i}$. Since $L/k$ is separable, the ring $S/\sqrt{I_i}S$ is reduced, so $\sqrt{I_i}S = \sqrt{I_iS}$. As $\sqrt{I_iS} = \gamma_L^{-1}(\x_1,\dots, \x_i)$, it follows that $\sqrt{I_i}$ is generated by $a_1+\dots+a_i$ linear forms. 

	Let $\gamma_k\in \GL_n(k)$ such that 
	\[
	\gamma_k\bigg(\left[\sqrt{I_1}\right]_1 \subseteq \dots \subseteq \left[\sqrt{I_r}\right]_1\bigg) = \bigg(\text{span}_k(\x_1)\subseteq \dots \subseteq \text{span}_k(\x_1,\dots, \x_r)\bigg)\]
	as flags in $\text{span}(\x_1,\dots, \x_r)$. By construction, $\gamma_k\gamma_L^{-1}\in \GL_n(L)$ fixes the flag
	\[
	\bigg(\text{span}_L(\x_1)\subseteq \dots \subseteq \text{span}_L(\x_1,\dots, \x_r)\bigg).
	\]
	By \Cref{lemma:fixed-flag-fixed-ideal}, we deduce that
	\[
	\gamma_k \overline{IS} = \gamma_k \gamma_L^{-1}\gamma_L \overline{IS} = \gamma_k\gamma_L^{-1}\Df S = \Df S.
	\]
By faithful flatness of $R\to S$, we conclude
	\[
	\gamma_k I = (\gamma_k I)S\cap R= (\Df S)\cap R = \Df.
	\]
	\end{proof}

We now deduce \Cref{thm:min-char}.
\begin{proof}[Proof of \Cref{thm:min-char}] 
Set $L = \bar{k}$ and let $S = L[x_1,\dots, x_n]$. By \Cref{prop:basic-properties-of-c} (vi) and \Cref{lemma:sigma_extension} we have $c(IS) = c(I) = E_l(I) = E_l(IS)$. By \Cref{thm:min-char} in the algebraically closed case, it follows that there exists $\gamma_L\in \GL_n(L), d_1,\dots, d_l\in \Z^+$ such that \Cref{eqn:gamma_L} holds. Because $k$ is perfect, the extension $L/k$ is separable, so \Cref{lemma:sep-extension-coordinates} implies that there exists $\gamma_k\in \GL_n(k)$ such that \Cref{eqn:gamma_k} holds, proving the theorem.
\end{proof}
To conclude this section, we note that the assumption on $k$ cannot be weakened.
\begin{example}\label{ex:must-be-perfect}
	Let $k$ be a field of characteristic $p > 0$ and $t\in k\setminus k^p$. Let $R = k[x,y]$ and set $I = (x^p + ty^p)$. Then $c(I) = \frac{1}{p} = E_1(I)$, but $I$ is not generated by a power of a linear form.
\end{example}

\section{Extensions of Theorem B in Characteristic Zero}
One can ask whether \Cref{thm:min-char} can be generalized to the local case. We note that the answer is ``no'' in positive characteristic, and we pose a conjecture generalizing \Cref{thm:min-char} over the complex numbers. 
\begin{example}\label{ex:pos-char-failure}
    Let $R = \overline{\F_p}\llbracket x, y\rrbracket$, and let $I = (x^p + y^{p+1})\subseteq R$. Then $E_1(I) = \frac{1}{p} =\fpt(I)$, but $I = \overline{I}$ and there are no coordinates for $R$ in which $I$ is a monomial ideal. 
\end{example}
For a complete description of when $E_1(I) = c(I)$ in positive characteristic, see \cite{baily_extremal_2026}. Whereas \Cref{thm:min-char} fails to generalize to the local case in char $p > 0$, we are optimistic that a stronger result is possible in characteristic zero. Our conjecture extends \cite[Remark 3.7]{bivia-ausina_Lojasiewicz} to the case of arbitrary codimension.
\begin{conj}\label{conj:ideal-theoretic}
    Let $R = \C\llbracket x_1,\dots, x_n\rrbracket$. Let $I\subseteq R$ be an ideal with $\codim(I)\geq l$. Then $\lct(I) = E_l(I)$ if and only if there exists a regular system of parameters $y_1,\dots, y_n$ for $R$ and positive integers $d_1,\dots, d_l$ such that $\overline{I} = \overline{(y_1^{d_1},\dots, y_l^{d_l})}$.
\end{conj}
We note that \Cref{conj:ideal-theoretic}, if proven, would be strong enough to give an alternate proof of \Cref{thm:min-char} over $\C$. 
\begin{proof}
    Let $R = \C[x_1,\dots, x_n]$. Let $I$ be a homogeneous ideal with $\codim(I)\geq l$ and $E_l(I) = \lct(I)$. Set $S = \C\llbracket x_1,\dots, x_n\rrbracket$. \Cref{notation:singularity_threshold,defn:mm-dp-poly-ring} imply that $E_l(IS) = E_l(I)$ and $\lct(IS) = \lct(I)$, so by \Cref{conj:ideal-theoretic} there exists a regular system of parameters $y_1,\dots, y_n$ for $S$ such that $\overline{IS} = \overline{(y_1^{d_1},\dots, y_l^{d_l})}$.

		Write each $y_i$ as a power series in $x_1,\dots, x_n$ and let $z_1,\dots, z_n$ be the homogeneous degree-1 terms of $y_1,\dots, y_n$. As $y_i\equiv z_i\mod \mf^2$, the forms $z_1,\dots, z_n$ in $R$ generate the homogeneous maximal ideal of $R$.
		
	As $I$ is homogeneous, so is $\bar{I}$, hence $\bar{I} = \bar{I}S\cap R$. In particular, we have $z_1^{d_1},\dots, z_l^{d_l}\in \overline{I}$. By considering the homogeneous ideals $\overline{I + (z_{l+1}^t,\dots, z_n^t)}$ for $t>0$ (which evidently satisfy $E_n = \lct$) and running an argument similar to \Cref{thm:min-char}, we deduce that $\overline{I} = \overline{(z_1^{d_1},\dots, z_l^{d_l})}$.\end{proof}

We propose a stronger conjecture in terms of valuations. Let $R = \C\llbracket x_1,\dots, x_n\rrbracket$. Let $v: R\to [0, \infty]$ be a valuation and let $A_R(v)$ be the log discrepancy as in \cite{jonsson_valuations_2012}. For each positive integer $n$, define $\af_n(v) = \{f\in R: v(f)\geq n$. By \cite[Corollary 6.9]{jonsson_valuations_2012} and \cite[Lemma 3.5]{blum_existence_2018}, we have 
\begin{equation}\label{eqn:lct_vs_logdisc}
	\lct(\af_\bullet(v))\leq A_R(v)
\end{equation}
If $v$ is centered on a prime ideal $\pf$ of height at least $l$, then for $m \gg 0$ we have $\sqrt{\af_m(v)} = \pf$ and so $\codim(\af_m(v)) \geq l$. Consequently, we have $E_l(\af_\bullet(v)) \leq A_R(v)$.
\begin{conj}\label{conj:valuation-theoretic}
	Let $R, v, \af_n(v)$ be as above and assume $v$ is centered on a prime ideal of height at least $l$. If $E_l(\af_\bullet(v))  = A_R(v)$ then $v$ is a \textit{monomial} valuation.
\end{conj}

We conclude this article by demonstrating that \Cref{conj:valuation-theoretic} implies \Cref{conj:ideal-theoretic}. First, a few preparatory results.
\begin{defn}
	Let $(R, \mf)$ be a Noetherian local ring and $I\subseteq R$ an ideal. The \textbf{fiber cone} of $I$ is the ring
	\[
	\mathcal{F}_I = \bigoplus_{n\geq 0}\frac{I^n}{\mf I^n}.
	\]
	The \textbf{analytic spread} of $I$, denoted $\ell(I)$, is the dimension of $\mathcal{F}_I$.
\end{defn}
\begin{theorem}[\cite{huneke_integral}, Theorem 8.3.7]\label{thm:spread}
	Let $(R, \mf)$ be a Noetherian local ring such that $R/\mf$ is infinite. If $I\subseteq R$ is a proper ideal, then $\ell(I)$ is equal to the least number of generators of an ideal $J\subseteq I$ such that $\overline{I} = \overline{J}$. 
\end{theorem}
\begin{lemma}\label{lemma:lower-height-dp-rees}
	Let $(R, \mf)$ be a regular local ring of dimension $n$. Let $1\leq l\leq n$ and let $I\subseteq J$ be ideals of height $l$. If $\sigma_l(I) = \sigma_l(J)$ and the analytic spread of $J$ is equal to $l$, then $\overline{I} = \overline{J}$. 
\end{lemma}
\begin{proof}
Replacing $R$ with $R[X]_{\mf R[X]}$ changes neither the hypotheses nor the conclusion by \cite[Lemma 8.4.2]{huneke_integral}; we may without loss of generality assume that $R/\mf$ is infinite. By \Cref{prop:sigma-properties} (ii) and (iii), choose $f_1,\dots, f_l\in I$ such that, with $I' = (f_1,\dots, f_l)$, we have $\sigma_l(I) = \sigma_l(I')$. Moreover, by \Cref{thm:spread} we have $\ell(I') = l$.

Let $c_0(I),\dots, c_n(I)$ denote the \textbf{multiplicity sequence} of Achilles and Manaresi \cite[Definition 2.2]{achilles_bigraded_1997}. By \cite[Corollaries 2.6 and 2.7]{achilles_mm_2019}, for $K\in \{I', J\}$ we have $c_{n-l}(K) = \sigma_l(K)$. Moreover, by \cite[Proposition 2.3 (i)]{achilles_bigraded_1997} we have $c_i(K) = 0$ for $i\neq n-d$, so $I', J$ have the same multiplicity sequence. It follows from \cite[Theorem 4.2]{polini_dependence_2020} that $\overline{I'} = \overline{J}$, which implies the claim.
\end{proof}
\begin{prop}
	If \Cref{conj:valuation-theoretic} holds in dimension $n$ and codimension $l$, then \Cref{conj:ideal-theoretic} holds in dimension $n$ and codimension $l$.
\end{prop}
\begin{proof}
	Let $I\subseteq R$ be an ideal of height at least $l$ such that $E_l(I) = \lct(I)$. By \Cref{defn:lct}, choose a divisorial valuation $v: R\to [0, \infty]$ such that $\lct(I) = \frac{A_R(I)}{v(I)}$. Letting $w = \frac{v}{v(I)}$, we have $w(I) = 1$ and $A_R(w) = \lct(I)$. As $w(I) = 1$, we have $I^n\subseteq \af_n(w)$ for all $n > 0$, so $\codim(w) \geq l$ and $E_l(\af_\bullet(w)) \geq E_l(I)$. 
	
	It follows from \Cref{thm:bound,eqn:lct_vs_logdisc} that
	\begin{equation}\label{eqn:lct<=logdisc<=lct}
	\lct(I) = E_l(I)\leq E_l(\af_\bullet(w)) \leq \lct(\af_\bullet(w))\leq A_R(w) = \lct(I),
	\end{equation}
	so in particular $E_l(\af_\bullet(w)) = A_R(w)$. By \Cref{conj:valuation-theoretic}, $w$ is a monomial valuation in some local coordinates $z_1,\dots, z_n$ for $R$ with $w(z_1)\geq w(z_2)\geq \dots \geq w(z_n)$.  Since $\sigma_j(I)\leq \sigma_j(\af_\bullet(w))$ for $1\leq j\leq l$ and $E_l(I) = E_l(\af_\bullet(w))$, we conclude $\sigma_j(I) = \sigma_j(\af_\bullet(w))$ by \cite[Proposition 10]{bivia-ausina_log_2016}. We can read off the invariants $\sigma_1(\af_\bullet(w)),\dots, \sigma_l(\af_\bullet(w))$ from the weights $w(z_i)$. In particular, we have $\sigma_j(\af_\bullet(w)) = \frac{1}{w(z_1)\dots w(z_j)}$, hence $w(z_j) = \frac{\sigma_{j-1}(I)}{\sigma_j(I)}$ for all $1\leq j\leq l$. Moreover, by definition of the log discrepancy for quasi-monomial valuations we have $A_R(w) = w(z_1) + \dots + w(z_n)$, so we must have $w(z_{l+1}) = \dots = w(z_n) = 0$. 
	
	By construction, $I\subseteq \af_1(w)$. If the numbers $\frac{1}{w(z_1)},\dots, \frac{1}{w(z_l)}$ are not all \textit{integers}, then $\Gamma(\af_1(w))\subsetneq \Gamma(\af_\bullet(w))$ -- but this would imply \[\lct(I) \leq \lct(\af_1(w)) < \lct(\af_\bullet(w)),\] which is impossible by \Cref{eqn:lct<=logdisc<=lct}. Write $d_i = \frac{1}{w(z_i)}$ for $1\leq i\leq l$. We are now in the situation where $I\subseteq \overline{(z_1^{d_1},\dots, z_l^{d_l})} =: \Df$ and $E_l(I) = E_l(\Df)$. Since $\Df$ is the integral closure of an ideal generated by $l$ elements, we have $\ell(\Df) = l$. By \Cref{lemma:lower-height-dp-rees} we deduce that $\overline{I} = \overline{\Df}$.
	\end{proof}
	The following argument is well-known to experts.
	\begin{prop}\label{prop:codimension-1}
		\Cref{conj:valuation-theoretic} holds in codimension $l = 1$. 
	\end{prop}
	\begin{proof}
		If $v$ is centered on a prime ideal of height $2$ or more, then $\codim(\af_\bullet(v)) \geq 2$ and hence $A_R(v) \geq \lct(\af_\bullet(v)) \geq E_2(v) > E_1(v)$. Consequently, in order to have $E_1(v) = A_R(v)$, it must be the case that $v$ is centered on a principal prime ideal $fR$, hence $v = c\ord_f$ for some irreducible $f\in R$ and some $c\in \R^+$. We deduce that $E_1(f) =\lct(f)$, hence $f$ is a power of a local coordinate by \cite{guan_lelong_2015}.
			\end{proof}
\begin{prop}
	\Cref{conj:valuation-theoretic} holds in dimension $2$. 
\end{prop}
\begin{proof}
	By \Cref{prop:codimension-1}, it remains to consider the case $l = 2$. Suppose $E_2(v) = A_R(v)$. Without loss of generality, normalize $v$ so that $v(\mf) = 1$. Let $\alpha(v) = \sup_{f\in \mf}v(f)/\ord_{\mf}(f)$ denote the \textit{skewness} of $\alpha$ as in \cite{favre_tree_2004}. By \cite[Remark 4.9]{boucksom_refinement_2014}, we have $e_1(\af_\bullet(v)) = e_2(\af_\bullet(v)) = \frac{1}{\alpha(v)}$, so $E_2(v) = 1 + \alpha(v)$.
		By \cite[Proposition 3.48 (i)]{favre_tree_2004}, we have $1 + \alpha(v)\leq A_R(v)$ with equality if and only if $v$ is a monomial valuation, proving the claim.
\end{proof}
	
\printbibliography
\end{document}